\documentclass[final,onefignum,onetabnum,letterpaper]{siamonline190516}

\usepackage{lipsum}
\usepackage{amsfonts}
\usepackage{epstopdf}
\ifpdf
  \DeclareGraphicsExtensions{.eps,.pdf,.png,.jpg}
\else
  \DeclareGraphicsExtensions{.eps}
\fi

\usepackage{datetime}
\newdateformat{monthyeardate}{%
  \monthname[\THEMONTH] \THEDAY, \THEYEAR}

\usepackage{academicons}
\usepackage{xcolor}
\renewcommand{\orcid}[1]{\href{https://orcid.org/#1}{\textcolor[HTML]{A6CE39}{orcid.org/#1}}}

\usepackage{amsmath}
\allowdisplaybreaks
\usepackage{amssymb}
\usepackage{commath}
\usepackage{mathtools}
\usepackage{bbm}

\usepackage{color}
\usepackage{graphicx}
\usepackage[small]{caption}
\usepackage{subcaption}

\usepackage{relsize}
\usepackage{adjustbox}
\usepackage{algorithm}
\usepackage[noend]{algpseudocode}
\usepackage{booktabs}
\usepackage{tikz}

\usepackage{verbatim}

\usepackage{enumitem}
\setlist[enumerate]{leftmargin=.5in}
\setlist[itemize]{leftmargin=.5in}

\newsiamremark{remark}{Remark}
\newsiamremark{hypothesis}{Hypothesis}
\crefname{hypothesis}{Hypothesis}{Hypotheses}
\newsiamthm{claim}{Claim}

\headers{Provable positive and exact CFs}{Jan Glaubitz}

\title{Construction and application of provable positive and exact cubature formulas
\thanks{
\monthyeardate\today 
\funding{This work was partially supported by AFOSR \#F9550-18-1-0316 and ONR \#N00014-20-1-2595.}
}}

\author{Jan Glaubitz\thanks{Department of Mathematics, Dartmouth College, Hanover, NH 03755, USA (\email{Jan.Glaubitz@Dartmouth.edu}, \orcid{0000-0002-3434-5563})}
}

\usepackage{amsopn}

\newcommand{\Span}{\mathrm{span}}
\newenvironment{eq}{\begin{equation}}{\end{equation}} 
\DeclareMathOperator{\rank}{rank}
\DeclareMathOperator{\diag}{diag}
\DeclareMathOperator*{\argmin}{arg\,min} 
\newcommand{\scp}[2]{\left\langle{#1,\, #2}\right\rangle} 
\newcommand{\intd}{\, \mathrm{d}}
\newcommand{\N}{\mathbb{N}}

\newcommand{\R}{\mathbb{R}}

\ifpdf
\hypersetup{
  pdftitle={Glaubitz2021positiveCFs},
  pdfauthor={Jan Glaubitz}
}
\fi

\begin{document}

\maketitle

\begin{abstract}
Many applications require multi-dimensional numerical integration, often in the form of a cubature formula.  
These cubature formulas are desired to be positive and exact for certain finite-dimensional function spaces (and weight functions). 
Although there are several efficient procedures to construct positive and exact cubature formulas for many standard cases, it remains a challenge to do so in a more general setting.  
Here, we show how the method of least squares can be used to derive provable positive and exact formulas in a general multi-dimensional setting. 
Thereby, the procedure only makes use of basic linear algebra operations, such as solving a least squares problem. 
In particular, it is proved that the resulting least squares cubature formulas are ensured to be positive and exact if a sufficiently large number of equidistributed data points is used. 
We also discuss the application of provable positive and exact least squares cubature formulas to construct nested stable high-order rules and positive interpolatory formulas. 
Finally, our findings shed new light on some existing methods for multi-variate numerical integration and under which restrictions these are ensured to be successful.  
\end{abstract}

\begin{keywords}
	Multivariate integration, numerical integration, quadrature, cubature, least squares, discrete orthogonal functions, equidistributed sequence, low-discrepancy sequence, Caratheodory--Tchakaloff measure
\end{keywords}

\begin{AMS}
	65D30, 65D32, 41A55, 41A63, 42C05
\end{AMS}

\begin{DOI}
    \url{https://doi.org/10.1093/imanum/drac017}
\end{DOI}

\section{Introduction} 
\label{sec:introduction} 

Numerical integration is an omnipresent technique in applied mathematics, engineering, and many other sciences. 
Prominent examples include numerical differential equations \cite{hesthaven2007nodal,ames2014numerical}, machine learning \cite{murphy2012machine}, finance \cite{glasserman2013monte}, and biology \cite{manly2006randomization}.

\subsection{Problem Statement}

In many cases, the problem can be formulated as follows. 
Let $d \geq 2$ and $\Omega \subset \R^d$ be bounded with positive volume and boundary of measure zero (in the sense of Lebesgue).
Given a function $f:\Omega \to \R$, we seek to approximate the continuous integral 
\begin{equation}\label{eq:int}
	I[f] := \int_{\Omega} f(\boldsymbol{x}) \omega(\boldsymbol{x}) \intd \boldsymbol{x} 
\end{equation} 
with Riemann integrable \emph{weight function} $\omega: \Omega \to \R^+_0$ that is assumed to be positive almost everywhere. 
A prominent approach is to approximate \cref{eq:int} by a weighted finite sum over the function values $f(\mathbf{x}_n)$ at some \emph{data points} $\mathbf{x}_1,\dots,\mathbf{x}_N$, denoted by 
\begin{equation}\label{eq:CF}
	C_N[f] := \sum_{n=1}^N w_n f(\mathbf{x}_n). 
\end{equation}
Usually, $C_N$ is referred to as an \emph{$N$-point cubature formula (CF)} and ${w_1,\dots,w_N \in \R}$ are called \emph{cubature weights}. 
For a 'good' CF, the following properties are often required:
\begin{enumerate}[label=(P\arabic*)] 
	\item \label{item:P1}
	All data points should lie inside of $\Omega$. That is, $\mathbf{x}_n \in \Omega$ for all $n=1,\dots,N$. 
	
	\item \label{item:P2} 
	The CF should be \emph{positive}. 
	That is, $w_n > 0$ for all $n=1,\dots,N$.\footnote{Some authors require the cubature weights to only be nonnegative. 
However, if $w_n = 0$, then the corresponding data point can---and should---be removed from the CF to avoid an unnecessary loss of efficiency.} 

\end{enumerate}  
See \cite{engels1980numerical,cools1997constructing,davis2007methods} and references therein.
Moreover, in many applications, CFs are desired which are exact for certain finite-dimensional function spaces. 
Let $\mathcal{F}_K(\Omega)$ denote a $K$-dimensional function space spanned by ${\varphi_1,\dots,\varphi_K:\Omega \to \R}$. 
It shall be assumed that all the moments $m_k := I[\varphi_k]$, $k=1,\dots,K$, exist.
Then, the $N$-point CF \cref{eq:CF} is said to be \emph{$\mathcal{F}_K(\Omega)$-exact} if 
\begin{equation}\label{eq:exactness}
	C_N[f] = I[f] \quad \forall f \in \mathcal{F}_K(\Omega). 
\end{equation} 
Usual choices for $\mathcal{F}_K(\Omega)$ include the space of all algebraic and trigonometric polynomials up to a certain degree $m$, respectively denoted by $\mathbb{P}_m(\R^d)$ and $\Pi_m(\R^d)$. 
Another prominent example is approximation spaces of radial basis functions  \cite{sommariva2006numerical,fuselier2014kernel,reeger2018numerical,sommariva2021rbf}.

\subsection{Previous Works}

The existence of positive and $\mathcal{F}_K(\Omega)$-exact CFs with at most $K$ data points is ensured by the following theorem originating from \cite{tchakaloff1957formules}. 
Also see \cite{davis1968nonnegative,bayer2006proof}.
\begin{theorem}[Tchakaloff, 1957]\label{thm:Tchakaloff}
	Given is \cref{eq:int} with nonnegative $\omega$ and a $K$-dimensional function space $\mathcal{F}_K(\Omega)$. 
	Then there exist $N$ data points $\mathbf{x}_1,\dots,\mathbf{x}_N \in \Omega$ and positive weights $w_1,\dots,w_N$ with $N \leq K$ such that the corresponding $N$-point CF is $\mathcal{F}_K(\Omega)$-exact. 
\end{theorem}
However, it should be noted that Tchakaloff's original proof is not constructive in nature. 
A constructive, yet in general not computational practical, prove was provided in \cite{davis1967construction}. 
At the same time, it should be pointed out that \cref{thm:Tchakaloff} only provides an upper bound for the number of data points that are needed for a positive and $\mathcal{F}_K(\Omega)$-exact CF. 
Indeed, for standard domains (e.\,g.\ $\Omega = [0,1]^d$) and weight functions (e.\,g.\ $\omega \equiv 1$) as well as classic function spaces (e.\,g.\ algebraic polynomials), it is possible to construct CFs that use even fewer points \cite{maeztu1989symmetric,bos2017polynomial,benouahmane2019near}. 
Such CFs are referred to as minimal or near-minimal CFs and usually utilize some kind of symmetry in the domain, weight function, and function space. 
Unfortunately, they are often limited to algebraic polynomials of low total degrees or specific domains and weight functions. 
We also mention Smolyak (sparse grid) CFs \cite{smolyak1963quadrature,bungartz2004sparse}, which rely on a product domain and a separable weight function.  
Another approach is the use of optimization-based methods \cite{taylor2007cardinal,ryu2015extensions,jakeman2018generation,keshavarzzadeh2018numerical} to construct near-minimal CFs (sometimes based on heuristic arguments). 
Unfortunately, the success of these methods often depends on the initial guess for the data points and they can sometimes yield points outside of the computational domain or negative cubature weights. 
We shall also mention some recent works on constructing nested sampling-based positive CFs \cite{van2020adaptive,van2020generating}. 
However, it should be noted that these are developed based on an approximate notion of exactness; in the sense that $I[f]$ in the exactness condition \cref{eq:exactness} is replaced by a discrete approximation $I^{(M)}[f] = \frac{1}{M} \sum_{m=1}^M f(\mathbf{y}_m)$ with $M > N$ and random samples $\mathbf{y}_m$, $m=1,\dots,M$. 
In fact, one might argue that the method proposed in \cite{van2020generating} is closer to subsampling \cite{davis1967construction,wilson1969general,wilhelmsen1976nearest,piazzon2017caratheodoryLS,van2017non}.
A similar argument can also be made for CFs based on nonnegative least squares (NNLS); see \cite{huybrechs2009stable,sommariva2015compression,glaubitz2020stableQRs}. 
Finally, a fairly general approach to construct ``probable" positive and exact CFs was proposed in \cite{migliorati2020stable}. 
The idea there was to derive randomized CFs based on first approximating the unknown integrand $f$ by a discrete least squares (LS) approximation \cite{cohen2013stability,cohen2017optimal,guo2020constructing} and to exactly integrate this approximation then. 
Thereby, the discrete LS approximation was based on the values of $f$ at random samples.  
The authors proved in \cite{migliorati2020stable} that the resulting randomized CFs are positive and exact with a high probability if the number of random samples is sufficiently large. 
To the best of our knowledge, this result has not been carried over to the deterministic setting (in the sense of the data points not coming from fully-probabilistic random samples) yet. 
We also mention the two works \cite{nakatsukasa2018approximate,hayakawa2021monte}, which we shall address in more detail below. 
In a nutshell, although many approaches exist to construct \emph{provable} positive and exact CFs for certain special cases, it remains a challenge to do so in a general deterministic setting.

\subsection{Our Contribution}

We propose a simple procedure to construct provable positive and exact CFs in a general setting. 
The procedure is based on the idea of least squares quadrature formulas (LS-QFs). 
These were introduced in \cite{wilson1970necessary,wilson1970discrete} (originally for equidistant points and $\omega \equiv 1$) and generalized in \cite{huybrechs2009stable,glaubitz2020stableQRs} (in one dimension though). 
Recently, LS formulas have been extended to higher dimensions in \cite{glaubitz2020stableCFs}. 
However, this was done under the restriction of the LS-CF being exact for algebraic polynomials up to a certain (total) degree, rather than a general finite-dimensional function space. 

The theoretical backbone of the present work is the extension of this result to a general class of finite-dimensional function spaces. 
We show that LS-CF can be ensured to be positive if a sufficiently large number of equidistributed data points is used. 
Numerically, we found that the number of data points $N$ has to scale roughly like the squared dimension of the function space $\mathcal{F}_K(\Omega)$. 
Moreover, we provide an error analysis for the corresponding positive and exact LS-CF, discuss potential applications, and address some connections to other integration methods.       

To prove that LS-CFs are positive, we leverage different tools from linear algebra, LS problems, discrete orthonormal functions, and equidistributed sequences from number theory. 
The sufficient conditions for this result (summarized in \cref{cor:main}) to hold, are the following: 
\begin{enumerate}[label=(R\arabic*)] 
	\item \label{item:restr_domain} 
	The integration domain $\Omega \subset \R^d$ is bounded with boundary of measure zero.

	\item \label{item:restr_weight} 
	The weight function $\omega: \Omega \to \R_0^+$ is Riemann integrable and positive almost everywhere. 

	\item \label{item:restr_space} 
	The function space $\mathcal{F}_K(\Omega)$ is spanned by a basis $\{ \varphi_k \}_{k=1}^K$ of continuous and bounded functions. 
	Furthermore, $\mathcal{F}_K(\Omega)$ contains constants. In particular, $1 \in \mathcal{F}_K(\Omega)$. 
	
\end{enumerate} 
\ref{item:restr_domain} essentially ensures the existence and simple construction of an equidistributed sequence in $\Omega$. 
\ref{item:restr_weight} warrants the existence of a certain continuous inner product (and corresponding orthonormal functions) that can be approximated by a sequence of discrete inner products (and corresponding discrete orthonormal functions). 
\ref{item:restr_space} is needed for technical reasons and is utilized in the proof of the preliminary \cref{lem:2}. 
Well-conditioned computation of the LS-CFs can be ensured if a basis $\{ \varphi_k \}_{k=1}^K$ for $\mathcal{F}_K(\Omega)$ of orthonormal functions is available (see \cref{rem:conditioning}).

\subsection{Implications and Potential Applications}

Our findings imply a simple\footnote{Indeed, the procedure only utilizes simple operations from linear algebra, such as solving an LS problem.} procedure to construct (deterministic) stable high-order CFs in a general setting. 
This can be interpreted as an extension of the results on stable high-order randomized CF from \cite{migliorati2020stable} to a deterministic setting. 
Moreover, we discuss the application of the provable positive and exact LS-CF to construct interpolatory CFs (see \cref{sub:interpol_CFs}). 
These use a smaller number of $N = K$ data points and can be obtained by several different subsampling strategies. 
It should also be noted that for a fixed function space $\mathcal{F}_K(\Omega)$ and an increasing number $N$ of data points, the LS-CFs discussed here might be interpreted as high-order corrections to Monte Carlo (MC)---for random data points---and quasi-Monte Carlo (QMC)---for low-discrepancy data points---methods. 
See \cref{rem:MC_correction} for more details. 
In particular, this reveals an interesting connection to \cite{nakatsukasa2018approximate} and potential applications of LS-CFs for variance reduction in MC and QMC methods. 
We shall also briefly mention the implication of the present work to several other approaches to find positive (interpolatory) CFs. 
These include NNLS \cite{huybrechs2009stable,sommariva2015compression,glaubitz2020stableQRs} and different optimization strategies \cite{glaubitz2020stableCFs,hayakawa2021monte}. 
Sometimes, these approaches are justified by Tchakaloff's theorem (\cref{thm:Tchakaloff}), which ensures the existence of a solution to the respective optimization problem if an appropriate set of data points is considered. 
That said, Tchakaloff's theorem is providing no information about which data points should be used (it only states their existence and an upper bound for their number). 
Hence, when applying the above-mentioned procedures without care to an arbitrary set of data points, they cannot always be expected to actually result in a positive interpolatory CF.  
The present work is providing such an ensurance in the sense that \cref{cor:main} combined with the subsampling strategies discussed in \cref{sub:interpol_CFs} is telling us that at least some of the positive interpolatory CFs predicted by Tchakaloff's theorem are supported on a sufficiently large set of equidistributed points in $\Omega$.

\subsection{Advantages and Pitfalls}

The advantage of the positive and exact CFs discussed in the present work lies in their generality and simple construction. 
That said, they are neither minimal nor near-minimal. 
Hence, if the reader is only interested in a certain standard case for which efficient CFs are readily available it is certainly advantageous to use these. 
The positive and exact CFs presented here, on the other hand, find their greatest utility when a CF is desired for a non-standard domain, weight function, or function space. 
They might also be of advantage when one is given a fixed and prescribed set of (scattered) data points, which is often the case in applications.   
Finally, for reasons of computational efficiency, in general, we only recommend using LS-CFs in moderate dimensions (usually $d=1,2,3$). 
In higher dimensions, they--- like many other CFs---might fall victim to the curse of dimensionality \cite{bellman1966dynamic,kuo2005lifting}.
Also see \cref{sec:error} for more details. 
Finally, the construction of the positive and exact CFs discussed here relies on knowledge of certain moments, which is discussed in \cref{rem:moments}.

\subsection{Outline} 

The rest of this work is organized as follows. 
In \cref{sec:prelim}, we collect some preliminaries, in particular, on unisolvent and equidistributed sequences. 
Next, \cref{sec:LS} contains our theoretical main results, i.\,e., exactness and conditional positivity of LS-CFs is proved. 
In \cref{sec:error}, an error analysis for these CFs is provided. 
Two specific applications of the provable positive and exact LS-CFs are discussed in \cref{sec:applications}. 
These include the simple construction of stable high-order sequences of CFs (\cref{sub:const_procedure}) and positive interpolatory CFs (\cref{sub:interpol_CFs}). 
Numerical experiments are presented in \cref{sec:numerical} and some concluding thoughts are offered in \cref{sec:summary}.   
\section{Preliminaries: Unisolvent and Equidistributed Sequences} 
\label{sec:prelim} 

Here, we shall provide a few preliminary results. 
In particular, these address the connection between the exactness of CFs and unisolvent and equidistributed sequences.

\subsection{Exactness and Unisolvence}
\label{sub:prelim_unisolvent}

Let $\{\varphi_k\}_{k=1}^K$ be a basis of the function space $\mathcal{F}_K(\Omega)$ and $m_k = I[\varphi_k]$, $k=1,\dots,K$, the corresponding \emph{moments}. 
Then the exactness condition \cref{eq:exactness} is equivalent to the data points ${X_N = \{ \mathbf{x}_n\}_{n=1}^N}$ and weights of a CF solving the nonlinear system 
\begin{equation}\label{eq:ex-nonlin-system}
  	\underbrace{
  	\begin{pmatrix}
    		\varphi_1(\mathbf{x}_1) & \dots & \varphi_1(\mathbf{x}_N) \\ 
    		\vdots & & \vdots \\ 
    		\varphi_K(\mathbf{x}_1) & \dots & \varphi_K(\mathbf{x}_N)
  	\end{pmatrix}}_{=: \Phi(X_N)}
  	\underbrace{
  	\begin{pmatrix} 
    		w_1 \\ \vdots \\ w_N 
  	\end{pmatrix}}_{=: \mathbf{w}} 
  	= 
  	\underbrace{
  	\begin{pmatrix} 
    		m_1 \\ \vdots \\ m_K 
  	\end{pmatrix}}_{=: \mathbf{m}}.
\end{equation} 
However, solving \cref{eq:ex-nonlin-system} can be highly nontrivial, and doing so by brute force may result in a CF which points lie outside of the integration domain or with negative weights \cite{haber1970numerical}. 
That said, the situation changes if a fixed set of data points is considered. 
Then $\Phi(X_N) = \Phi$, and \cref{eq:ex-nonlin-system} becomes a linear system:
\begin{equation}\label{eq:ex-system}
  	\Phi \mathbf{w} = \mathbf{m}
\end{equation} 
Note that for $K < N$ this is an underdetermined linear system. 
These are well-known to either have no or infinitely many solutions. 
The latter case arises when we restrict ourselves to unisolvent nodes.  

\begin{definition}[Unisolvent Nodes] 
  The nodes $X_N = \{\mathbf{x}_n\}_{n=1}^N \subset \R^d$ are called \emph{$\mathcal{F}_K(\Omega)$-unisolvent} if 
  \begin{equation}
    f(\mathbf{x}_n) = 0, \ n=1,\dots,N \implies f(\boldsymbol{x}) = 0, \ \forall \boldsymbol{x} \in \Omega
  \end{equation} 
  holds for all $f \in \mathcal{F}_K(\Omega)$. 
\end{definition} 

Assuming that the set of data points $X_N$ is $\mathcal{F}_K(\Omega)$-unisolvent the following result follows. 

\begin{lemma}\label{lem:solution-space}
  Let  $K < N$ and $X_N = \{\mathbf{x}_n\}_{n=1}^N$ be $\mathcal{F}_K(\Omega)$-unisolvent. 
  Then, the linear system \cref{eq:ex-system} induces an $(N-K)$-dimensional 
affine linear subspace of solutions,  
  \begin{equation}\label{eq:sol-space}
    W := \left\{ \mathbf{w} \in \R^N \mid \Phi\mathbf{w}=\mathbf{m} \right\}.
  \end{equation}
\end{lemma} 

\begin{proof} 
	The case $\mathcal{F}_K(\Omega) = \mathbb{P}_m(\R^d)$ was shown in \cite{glaubitz2020stableCFs}. 
	It is easy to verify that the same arguments carry over to the general case discussed here. 
\end{proof} 

Next, a simple sufficient criterion for $X_N \subset \Omega$ to be $\mathcal{F}_K(\Omega)$-unisolvent is provided. 
The criterion is based on sequences that are dense in $\Omega \subset \R^d$. 
Recall that $(\mathbf{x}_n)_{n \in \N} \subset \R^d$ is called \emph{dense in $\Omega$} if 
\begin{equation}\label{eq:dense}
	\forall \boldsymbol{x} \in \Omega \ \, \forall \varepsilon > 0 \ \, \exists n \in \N: \quad 
	\norm{ \boldsymbol{x} - \mathbf{x}_n } < \varepsilon.
\end{equation}
That is, every point in $\Omega$ can be approximated arbitrarily accurate by an element of $(\mathbf{x}_n)_{n \in \N}$.\footnote{The condition \cref{eq:dense} is independent of the norm since $\Omega$ is located in a finite-dimensional space.}  

\begin{lemma}\label{lem:unisolvent}
	Let $(\mathbf{x}_n)_{n \in \N} \subset \Omega$ be dense in $\Omega$ and let $X_N = \{ \mathbf{x}_n \}_{n=1}^N$. 
	Moreover, let $\mathcal{F}_K(\Omega)$ be spanned by continuous functions $\varphi_1,\dots,\varphi_K:\Omega \to \R$. 
	Then there exists an $N_0 \in \N$ such that $X_N$ is $\mathcal{F}_K(\Omega)$-unisolvent for every $N \geq N_0$. 
\end{lemma}

\begin{proof}
	Assume that the assertion is wrong. 
	Then, there exists an $f \in \mathcal{F}_K(\Omega)$ with $f \not\equiv 0$ such that ${f(\mathbf{x}_n) = 0}$ for all $n \in \N$. 
	Yet, since $(\mathbf{x}_n)_{n \in \N}$ is dense in $\Omega$, this either contradicts $f \not\equiv 0$ or $f$ being continuous. 
\end{proof}

If there exists an $N_0 \in \N$ such that $X_N = \{ \mathbf{x}_n \}_{n=1}^N$ is $\mathcal{F}_K(\Omega)$-unisolvent for every $N \geq N_0$, as in \cref{lem:unisolvent}, we say that $(\mathbf{x}_n)_{n \in \N}$ is an \emph{$\mathcal{F}_K(\Omega)$-unisolvent} sequence. 
Thus, \cref{lem:unisolvent} states that every dense sequence is also $\mathcal{F}_K(\Omega)$-unisolvent.

\subsection{Equidistributed Sequences}
\label{sub:prelim_equidistributed}

We just saw that density is a sufficient condition for unisolvence. 
This will be handy for the subsequent construction of provable positive and exact LS-CFs. 
Another important property will be for the sequence of data points $(\mathbf{x}_n)_{n \in \N} \subset \Omega$ to satisfy 
\begin{equation}\label{eq:equi} 
	\lim_{N \to \infty} \frac{|\Omega|}{N} \sum_{n=1}^N g(\mathbf{x}_n) 
		= \int_\Omega g(\boldsymbol{x}) \intd \boldsymbol{x} 
\end{equation} 
for all measurable bounded functions $g:\Omega \to \R$ that are continuous almost everywhere (in the sense of Lebesgue). 
Here, $|\Omega|$ denotes the $d$-dimensional volume of $\Omega$. 
Observe that $(\mathbf{x}_n)_{n \in \N}$ being dense in $\Omega$ is not sufficient for \cref{eq:equi} to hold.\footnote{However, every dense sequence can be rearranged into an (equidistributed) sequence satisfying \cref{eq:equi}.}
Yet, in \cite{weyl1916gleichverteilung} it was showed that \cref{eq:equi} can be connected to $(\mathbf{x}_n)_{n \in \N}$ being \emph{equidistributed} (also called \emph{uniformly distributed}). 
We shall recall that there are many well-known equidistributed sequences for $d$-dimensional hypercubes with radius $R$, denoted by $C_R^{(d)} = [-R,R]^d$. 
These include (i) grids of equally spaced points with an appropriate ordering and (ii) low-discrepancy sequences. 
The latter were developed to minimize the upper bound provided by the famous Koksma--Hlawak inequality \cite{hlawka1961funktionen,niederreiter1992random}, used in QMC methods \cite{caflisch1998monte,dick2013high,trefethen2017cubature}. 
A special case of low-discrepancy sequences are the Halton points \cite{halton1960efficiency}, which are a generalization of the one-dimensional van der Corput points, see for example \cite[Erste Mitteilung]{van1935verteilungsfunktionen}). 
To not exceed the scope of this work, we refer to the monograph \cite{kuipers2012uniform} for more details on equidistributed sequences.
That said, we shall demonstrate how equidistributed sequences can be constructed for general bounded domains with a boundary of measure zero. 

\begin{remark}[Construction of Equidistributed Sequences for General Domains]\label{rem:eq-points}
	Let ${\Omega \subset \R^d}$ be bounded with a boundary of measure zero. 
	Then we can find an $R > 0$ such that $\Omega$ is contained in the hypercube $C_R^{(d)}$. 
	Let $(\mathbf{y}_n)_{n \in \N}$ be an equidistributed sequence in $C_R^{(d)}$, then an equidistributed sequence in $\Omega$, $(\mathbf{x}_n)_{n \in \N}$, is given by the subsequence of $(\mathbf{y}_n)_{n \in \N}$ for which all elements outside of $\Omega$ have been removed:	
	\begin{equation}
		\mathbf{y}_n \in (\mathbf{x}_n)_{n \in \N} \iff \mathbf{y}_n \in \Omega.
	\end{equation} 
	Let us briefly verify that the resulting subsequence $(\mathbf{x}_n)_{n \in \N}$ is an equidistributed sequence in $\Omega$. 
	To this end, let $g:\Omega \to \R$ be a measurable bounded function that is continuous almost everywhere. 
	We can extend $g$ to the hypercube $C_R^{(d)}$ by setting $g$ equal to zero outside of $\Omega$. 
	This extension, denoted by $\tilde{g}: C_R^{(d)} \to \R$, is measurable, bounded, and continuous almost everywhere. 
	The latter follows from $\Omega$ having a boundary of measure zero. 
	Also observe that 
	\begin{eq}\label{eq:equid_aux}
		\lim_{M \to \infty} \frac{ | \{\mathbf{y}_1,\dots,\mathbf{y}_M\} \cap \Omega | }{M} = \frac{|\Omega|}{|C_R^{(d)}|}.
	\end{eq} 
	Then, by construction of $(\mathbf{x}_n)_{n \in \N}$ and \cref{eq:equid_aux}, we get  
	\begin{eq}\label{eq:const_eq}
		\lim_{N \to \infty} \frac{|\Omega|}{N} \sum_{n=1}^N g(\mathbf{x}_n) 
			= \lim_{M \to \infty} \frac{| C_R^{(d)} |}{M} \sum_{n=1}^M \tilde{g}(\mathbf{y}_n) 
			= \int_{C_R^{(d)}} \tilde{g}(\boldsymbol{x}) \intd \boldsymbol{x} 
			= \int_{\Omega} g(\boldsymbol{x}) \intd \boldsymbol{x}. 
	\end{eq} 
	Here, $M$ is the unique integer such that $\mathbf{x}_N = \mathbf{y}_M$, and $| \{\mathbf{y}_1,\dots,\mathbf{y}_M\} \cap \Omega | = N$.\footnote{The elements of $(\mathbf{y}_n)_{n \in \N}$ are assumed to be distinct.}
	Again, we refer to \cite{kuipers2012uniform} for more details.
\end{remark} 

Finally, it should be stressed that equidistributed sequences are dense sequences with a specific ordering. 
We close this section with the following corollary. 

\begin{corollary}\label{cor:equid}
	Let $\Omega \subset \R^d$ be bounded with a boundary of measure zero. 
	Furthermore, let $(\mathbf{y}_n)_{n \in \N}$ be an equidistributed sequence in the hypercube $C_R^{(d)}$, where $\Omega \subset C_R^{(d)}$, and let $(\mathbf{x}_n)_{n \in \N}$ be the subsequence of $(\mathbf{y}_n)_{n \in \N}$ that only contains the elements in $\Omega$. 
	Then $(\mathbf{x}_n)_{n \in \N}$ is equidistributed in $\Omega$ and $\mathcal{F}_K(\Omega)$-unisolvent. 
\end{corollary} 

\begin{proof} 
	The assertion that $(\mathbf{x}_n)_{n \in \N}$ is equidistributed in $\Omega$ follows from \cref{rem:eq-points}. 
	Finally, $(\mathbf{x}_n)_{n \in \N}$ being $\mathcal{F}_K(\Omega)$-unisolvent follows by every equidistributed sequence being dense and \cref{lem:unisolvent}. 
\end{proof} 
\section{Provable Positive and Exact Least Squares Cubature Formulas}
\label{sec:LS} 

In this section, it is demonstrated how provable positive and $\mathcal{F}_K(\Omega)$-exact LS-CFs can be constructed by using a sufficiently large set of equidistributed data points.  
This is done by generalizing the LS approach from \cite{wilson1970necessary,wilson1970discrete,huybrechs2009stable,glaubitz2020shock,glaubitz2020stableQRs,glaubitz2020stableCFs}. 
Finally, the procedure only relies on determining the LS solution of an underdetermined linear system.

\subsection{Formulation as a Least Squares Problem}
\label{sub:LS-problem}

Let $(\mathbf{x}_n)_{n \in \N}$ be a $\mathcal{F}_K(\Omega)$-unisolvent sequence in $\Omega$ and let $X_N = \{\mathbf{x}_n\}_{n=1}^N$. 
As noted before, an $\mathcal{F}_K(\Omega)$-exact CF with data points $X_N$ can be constructed by determining a vector of cubature weights that solves the linear system of exactness conditions \cref{eq:ex-system}. 
For $N > K$, \cref{eq:ex-system} induces an $(N-K)$-dimensional affine linear subspace space of solutions $W \subset \R^N$. 
All of these yield an $\mathcal{F}_K(\Omega)$-exact CF.
The LS approach consists of finding the unique solution $\mathbf{w} \in W$ that minimizes a weighted Euclidean norm: 
\begin{equation}\label{eq:LS-sol} 
	\mathbf{w}^{\mathrm{LS}} = \argmin_{\mathbf{w} \in W} \ \norm{ R^{-1/2} \mathbf{w} }_{2},
\end{equation} 
where $R^{-1/2}$ is a diagonal \emph{weight matrix}, given by 
\begin{equation} 
	R^{-1/2} = \diag\left( 1/\sqrt{r_1}, \dots, 1/\sqrt{r_N} \right), \quad 
	r_n > 0, \quad n=1,\dots,N.
\end{equation} 
The vector $\mathbf{w}^{\mathrm{LS}}$ in \cref{eq:LS-sol} is called the \emph{LS solution} of $\Phi \mathbf{w} = \mathbf{m}$. 
The corresponding $N$-point CF 
\begin{equation}\label{eq:LS-CF}
	C^{\mathrm{LS}}_N[f] = \sum_{n=1}^N w_n^{\mathrm{LS}} f(\mathbf{x}_n)
\end{equation}
is called an \emph{LS-CF}. 
In \cref{sub:positivity} it is shown that choosing the \emph{discrete weights} $r_n$ as 
\begin{equation} 
	r_n = \frac{\omega(\mathbf{x}_n) |\Omega|}{N}, \quad n=1,\dots,N,
\end{equation} 
results in the LS-CFs to be $\mathcal{F}_K(\Omega)$-exact and positive if $N$ is sufficiently large. 
At least formally, the LS solution is explicitly given by (\cite{cline1976l_2})
\begin{equation}\label{eq:LS-sol2}
  \mathbf{w}^{\mathrm{LS}} = R \Phi^T (\Phi R \Phi^T)^{-1} \mathbf{m}.    
\end{equation} 
Here, $R \Phi^T (\Phi R \Phi^T)^{-1}$ is the Moore--Penrose pseudoinverse of $R^{-1/2}\Phi$; see \cite{ben2003generalized}. 
In \cref{sub:char}, \cref{eq:LS-sol2} will be simplified by utilizing discrete orthonormal bases.

\begin{remark}[Computation of the Moments]\label{rem:moments}
	The construction of LS-CFs requires the computation of the moments $m_k = I[\varphi_k]$ for a basis $\{ \varphi_k \}_{k=1}^K$, which is a requirement for many CFs. 
	Depending on the domain $\Omega$, the weight function $\omega$, and the basis $\{ \varphi_k \}_{k=1}^K$, the exact evaluation of $m_k$ might be impractical. 
	However, we can always approximate $m_k$ by another unrelated CF, such as the QMC method, using a larger set of nodes. 
	In some cases, it might also be possible to use certain recurrence relations or differential/difference equations of the basis functions $\varphi_k$ to simplify the computation of the moments. 
\end{remark}

\subsection{Continuous and Discrete Orthonormal Bases}
\label{sub:ON-bases}

Under the restrictions \ref{item:restr_domain}--\ref{item:restr_space},
\begin{equation}\label{eq:cont-inner-prod}
  \scp{u}{v} = \int_{\Omega} u(\boldsymbol{x}) v(\boldsymbol{x}) \omega(\boldsymbol{x}) \intd \boldsymbol{x}, \quad 
  \norm{u} = \sqrt{\scp{u}{u}}
\end{equation} 
defines an inner product and a corresponding norm on $\mathcal{F}_K(\Omega)$.
In particular, \cref{eq:cont-inner-prod} allows to define an orthonormal basis $\{ \pi_k \}_{k=1}^K$ of $\mathcal{F}_K(\Omega)$. 
That is, the functions $\pi_1,\dots,\pi_K$ span $\mathcal{F}_K(\Omega)$ and  satisfy
\begin{equation}
  \scp{\pi_k}{\pi_l} = \delta_{k,l} := 
  	\begin{cases} 
		1 &; \ k = l, \\ 
		0 &; \ k \neq l,
	\end{cases} 
	\quad k,l = 1,\dots,K.
\end{equation} 
Henceforth, we refer to such a basis as a \emph{continuous orthonormal basis}.
%
Analogously, assuming that ${X_N^+ = \{ \, \mathbf{x}_n \mid \omega(\mathbf{x}_n) > 0, \ n=1,\dots,N \, \}}$ is $\mathcal{F}_K(\Omega)$-unisolvent,
\begin{equation}\label{eq:disc-inner-prod}
	[u,v]_N = \sum_{n=1}^N r_n u(\mathbf{x}_n) v(\mathbf{x}_n), \quad 
	\norm{u}_N = \sqrt{[u,u]_N}
\end{equation} 
defines a discrete inner product and a corresponding norm on $\mathcal{F}_K(\Omega)$. 
Also \cref{eq:disc-inner-prod} induces an orthonormal basis. 
This basis satisfies 
\begin{equation}
	[\pi_k,\pi_l]_N = \delta_{k,l}, \quad k,l = 1,\dots,K,
\end{equation} 
while spanning $\mathcal{F}_K(\Omega)$. 
It is therefore referred to as a \emph{discrete orthonormal basis} and its elements are denoted by $\pi_k^{(N)}$. 
%
Both bases can be constructed, for instance, by Gram--Schmidt orthonormalization applied to the same initial basis $\{\varphi_k\}_{k=1}^K$ \cite{gautschi1997numerical,trefethen1997numerical}:
\begin{equation}\label{eq:Gram-Schmidt}
\begin{aligned}
  \tilde{\pi}_k & = \varphi_k - \sum_{l=1}^{k-1} \scp{\varphi_k}{\pi_l^{(N)}} \pi_l, \quad 
  && \pi_k = \frac{\tilde{\pi}_k}{\norm{\tilde{\pi}_k}}, \\ 
  \tilde{\pi}^{(N)}_k & = \varphi_k - \sum_{l=1}^{k-1} [\varphi_k,\pi_l^{(N)}]_N \pi_l^{(N)}, \quad 
  && \pi_k^{(N)} = \frac{\tilde{\pi}_k^{(N)}}{ \norm{ \tilde{\pi}_k^{(N)} }_N }.
\end{aligned}
\end{equation} 

\begin{remark}
	Note that we only utilize Gram--Schmidt orthonormalization for theoretical purposes. 
	In our implementation, the LS solution $\mathbf{w}^{\mathrm{LS}}$ is computed using the Matlab function \emph{lsqminnorm}, which uses a pivoted QR decomposition of $A = \Phi R^{1/2}$; see \cite{trefethen1997numerical,golub2012matrix}. 
	The cost for this is $\mathcal{O}(N K^2)$. 
	While we have not tested this in our implementation, it should still be noted that also iterative solvers, such as the preconditioned conjugate gradient (PCG) method, might be used to reduce the costs to $\mathcal{O}(N K)$. 
	However, such methods usually rely on sparsity to be efficient and can be more prone to numerical inaccuracies. 
\end{remark}

\subsection{Characterization of the Least Squares Solution}
\label{sub:char}

We now collect one more important ingredient to subsequently prove the positivity of LS-CFs. 
Observe that the matrix product $\Phi R \Phi^T$ in the explicit representation of the LS solution \cref{eq:LS-sol2} can be interpreted as a Gram matrix with respect to the discrete inner product \cref{eq:disc-inner-prod}: 
\begin{equation}
  \Phi R \Phi^T = 
  \begin{pmatrix}
    [\varphi_1,\varphi_1]_N & \dots & 
    [\varphi_1,\varphi_K]_N \\ 
    \vdots & & \vdots \\ 
    [\varphi_K,\varphi_1]_N & \dots & 
    [\varphi_K,\varphi_K]_N \\ 
  \end{pmatrix}
\end{equation}
Thus, if the linear system \cref{eq:ex-system} is formulated with respect to the discrete orthonormal basis $\{\varphi_k^{(N)}\}_{k=1}^K$, one gets $\Phi R \Phi^T = I$, where $I$ denotes the $N \times N$ identity matrix. 
This yields \cref{eq:LS-sol2} to become  
\begin{equation}\label{eq:LS-sol3}
	\mathbf{w}^{\mathrm{LS}} = R \Phi^T \mathbf{m}.  
\end{equation}
In particular, the LS weights are then explicitly given by 
\begin{equation}\label{eq:LS-sol-explicit}
	w_n^{\mathrm{LS}} = r_n \sum_{k=1}^K \pi_k^{(N)}( \mathbf{x}_n) I[ \pi_k^{(N)} ], \quad n=1,\dots,N, 
\end{equation} 
where $I[ \pi_k^{(N)} ] = \int_{\Omega} \pi_k^{(N)}(\boldsymbol{x}) \intd \boldsymbol{x}$.

\subsection{Positivity of Least Squares Cubature Formulas}
\label{sub:positivity} 

We start by presenting two technical lemmas, which will enable us to show that the LS weights are all positive if a sufficiently large number of equidistributed data points is used. 

\begin{lemma}\label{lem:1}
	Let $\Omega \subset \R^d$ be bounded and assume that 
	\begin{equation}\label{eq:cond1}
		\lim_{N \to \infty} [u,v]_N = \scp{u}{v} \quad \forall u,v \in \mathcal{F}_K(\Omega).
	\end{equation} 
	Furthermore, let $(u_N)_{N \in \N}, (v_N)_{N \in \N} \subset \mathcal{F}_K(\Omega)$ and $u, v \in \mathcal{F}_K(\Omega)$ such that 
	\begin{equation}\label{eq:cond2}
		\lim_{N \to \infty} u_N = u, \quad 
		\lim_{N \to \infty} v_N = v \quad 
		\text{in } \ ( \mathcal{F}_K(\Omega), \| \cdot \|_{L^\infty(\Omega)} ),
	\end{equation} 
	where $u,v: \Omega \to \R$ are assumed to be bounded. 
	Then, 
	\begin{equation} 
		\lim_{N \to \infty} [u_N,v_N]_N = \scp{u}{v}.
	\end{equation}
\end{lemma}

Recall that $[\cdot,\cdot]_N$ and $\scp{\cdot}{\cdot}$ denote the continuous and discrete inner product \cref{eq:cont-inner-prod} and \cref{eq:disc-inner-prod}, respectively. 
The corresponding norms are denoted by $\|\cdot\|$ and $\|\cdot\|_N$.
Moreover, $\| \cdot \|_{L^\infty(\Omega)}$ is the usual supremum norm with $\norm{f}_{L^\infty(\Omega)} = \sup_{\boldsymbol{x} \in \Omega} |f(\boldsymbol{x})|$.

\begin{proof} 
  	We start by noting that 
  	\begin{equation} 
  	\begin{aligned}
    		\left| \scp{u}{v} - [u_N,v_N]_N \right| 
      		\leq & \left| \scp{u}{v} - [u,v]_N \right| + \left| [u,v]_N - [u_N,v]_N \right| \\ 
			& + \left| [u_N,v]_N - [u_N,v_N]_N \right|. 
	\end{aligned}
  	\end{equation}
  	The first term on the right-hand side converges to zero due to \cref{eq:cond1}. 
  	For the second term, the Cauchy--Schwarz inequality gives 
  	\begin{equation}
    		\left| [u,v]_N - [u_N,v]_N \right|^2  
      		= \left| [u-u_N,v]_N \right|^2  
      		\leq \norm{ u-u_N }_N^2 \norm{ v }_N^2.
  	\end{equation} 
  	Furthermore, \cref{eq:cond1} implies $\| v \|_N^2 \to \| v \|^2$ for $N \to \infty$.
  	Finally, the H\"older inequality and \cref{eq:cond2} yield 
  	\begin{equation}
    		\norm{ u-u_N }_N^2 \leq \norm{ 1 }_N^2 \norm{ u-u_N }_{L^\infty(\Omega)}^2 \to 0, 
    		\quad N \to \infty.
  	\end{equation}
  	Thus, the second term converges to zero as well. 
  	A similar argument can be used to show that the third term converges to zero. 
\end{proof}

Next, we demonstrate that the discrete orthonormal functions $\pi_k^{(N)}$ converge uniformly to the corresponding continuous orthonormal functions $\pi_k$ if the corresponding discrete inner product converges to the continuous one for all elements of $\mathcal{F}_K(\Omega)$. 

\begin{lemma}\label{lem:2}
	Let $\Omega \subset \R^d$ be bounded, let $\{ \varphi_k \}_{k=1}^K$ be a basis of $\mathcal{F}_K(\Omega)$ consisting of continuous and bounded functions, and assume that \cref{eq:cond1} holds. 
	Moreover, let $\{ \pi_k \}_{k=1}^K$ and $\{ \pi_k^{(N)} \}_{k=1}^K$ respectively denote the continuous and discrete orthonormal bases constructed from $\{ \varphi_k \}_{k=1}^K$ by Gram--Schmidt orthonormalization \cref{eq:Gram-Schmidt}. 
	Then, 
	\begin{equation} 
		\lim_{N \to \infty} \pi_k^{(N)} = \pi_k \quad 
		\text{in } \ ( \mathcal{F}_K(\Omega), \|\cdot\|_{L^\infty(\Omega)} ).
	\end{equation} 
\end{lemma} 

\begin{proof}
  	The assertion is proven by induction. 
	For $k=1$, recall that $\pi_1 = \varphi_1/\|\varphi_1\|$ and $\pi_1^{(N)} = \varphi_1/\|\varphi_1\|_N$.
	Hence, the assertion follows from \cref{eq:cond1} implying that ${\|\varphi_1\|_N \to \|\varphi_1\|}$ for ${N \to \infty}$. 
	Next, it is argued that if the assertion holds for the first $k-1$ orthonormal basis functions, then it also holds for the $k$-th one. 
  	To this end, let $l \in \{1,2,\dots,k-1\}$ and assume that 
  	\begin{equation}
    		\pi_l^{(N)} \to \pi_l \ \text{ in } \ ( \mathcal{F}_K(\Omega), \|\cdot\|_{L^\infty(\Omega)} ), \quad N \to \infty.
  	\end{equation}
  	Recall that by Gram--Schmidt orthonormalization, the $k$-th orthonormal basis functions are given by \cref{eq:Gram-Schmidt}. 
  	Hence, \cref{lem:1} implies   
  	\begin{equation}
  		[\varphi_k,\pi_l^{(N)}]_N \to \scp{\varphi_k}{\pi_l}, \quad N \to \infty, 
  	\end{equation}
  	and therefore 
  	\begin{equation}
    		\tilde{\pi}_k^{(N)} \to \tilde{\pi}_k \ \text{ in } \ ( \mathcal{F}_K(\Omega), \|\cdot\|_{L^\infty(\Omega)} ), 
    \quad N \to \infty.
  	\end{equation} 
  	Here, $\tilde{\pi}_k^{(N)}$ and $\tilde{\pi}_k$ respectively denote the unnormalized basis function.  
  	\cref{lem:1} yields 
  	\begin{equation} 
		\| \tilde{\pi}_k^{(N)} \|_N \to \| \tilde{\pi}_k \|, \quad N \to \infty. 
	\end{equation} 
  	This implies 
  	\begin{equation}
    		\pi_k^{(N)} \to \pi_k \ \text{ in } \ ( \mathcal{F}_K(\Omega), \|\cdot\|_{L^\infty(\Omega)} ), \quad N \to \infty,
  	\end{equation}
  	which completes the proof. 
\end{proof}

\cref{lem:1} and \cref{lem:2} now enable us to prove the following theorem. 

\begin{theorem}[The LS-CF is Conditionally Positive]\label{thm:main}
	Given is a bounded domain $\Omega \subset \R^d$, $\omega: \Omega \to \R_0^+$, and $\mathcal{F}_K(\Omega) \subset C(\Omega)$ such that the restrictions \ref{item:restr_weight} and \ref{item:restr_space} are satisfied, i.\,e., 
	\begin{enumerate}
		\item[(R2)]
		The weight function $\omega: \Omega \to \R_0^+$ is Riemann integrable and positive almost everywhere. 

		\item[(R3)] 
		The $K$-dimensional vector space $\mathcal{F}_K(\Omega)$ is spanned by a basis $\{ \varphi_k \}_{k=1}^K$ of continuous and bounded functions $\varphi_k:	\Omega \to \R$, $k=1,\dots,K$. 
		Furthermore, $\mathcal{F}_K(\Omega)$ contains constants. In particular, $1 \in \mathcal{F}_K(\Omega)$. 
	
	\end{enumerate}
	Moreover, let $(\mathbf{x}_n)_{n \in \N} \subset \Omega$ be $\mathcal{F}_K(\Omega)$-unisolvent and $(r_n)_{n \in \N} \subset \R^+$ such that  
 	\begin{equation}\label{eq:cond-thm}
		\lim_{N \to \infty} \sum_{n=1}^N r_n u(\mathbf{x}_n) v(\mathbf{x}_n) 
			= \int_\Omega \omega(\boldsymbol{x}) u(\boldsymbol{x}) v(\boldsymbol{x}) \intd \boldsymbol{x} 
			\quad \forall u, v \in \mathcal{F}_K(\Omega). 
	\end{equation} 
	Then, there exists an $N_0 \in \N$ such that for all $N \geq N_0$ the corresponding LS-CF 
	\begin{equation}
		C^{\mathrm{LS}}_N[f] = \sum_{n=1}^N w_n^{\mathrm{LS}} f(\mathbf{x}_n) 
		\quad \text{with} \quad 
		\mathbf{w}^{\mathrm{LS}} = \argmin_{\Phi \mathbf{w} = \mathbf{m}} \ \norm{ R^{-1/2} \mathbf{w} }_{2}, 
	\end{equation} 
	where $R^{-1/2} = \diag\left( 1/\sqrt{r_1}, \dots, 1/\sqrt{r_N} \right)$, is positive and $\mathcal{F}_K(\Omega)$-exact. 
\end{theorem} 

\begin{proof} 
	First, we note that $(\mathbf{x}_n)_{n \in \N}$ being $\mathcal{F}_K(\Omega)$-unisolvent ensures the existence of a discrete inner product and a discrete orthonormal basis. 
	Let us denote such a basis by $\{ \pi_k^{(N)} \}_{k=1}^K$ and the corresponding continuous orthonormal basis by $\{ \pi_k \}_{k=1}^K$. 
	It can be assumed that both bases are constructed by applying Gram--Schmidt orthonormalization to the same initial basis $\{ \varphi_k \}_{k=1}^K$. 
	Hence, the LS weights are explicitly given by  
	\begin{equation}
		w_n^{\mathrm{LS}} = r_n \sum_{k=1}^K \pi_k^{(N)}( \mathbf{x}_n) I[ \pi_k^{(N)} ], \quad n=1,\dots,N, 
	\end{equation} 
	for $N \geq K$, where $I[ \pi_k^{(N)} ] = \int_{\Omega} \pi_k^{(N)}(\boldsymbol{x}) \intd \boldsymbol{x}$. 
	Next, let 
	\begin{equation} 
		\epsilon_k^{(N)} := [\pi_k^{(N)},1]_N - \scp{\pi_k^{(N)}}{1}. 
	\end{equation} 
	This allows us to rewrite the LS weights as follows: 
	\begin{equation} 
	\begin{aligned}
		w_n^{\mathrm{LS}} 
			& = r_n \sum_{k=1}^K \pi_k^{(N)}( \mathbf{x}_n) \scp{ \pi_k^{(N)} }{ 1 } \\ 
			& = r_n \sum_{k=1}^K \pi_k^{(N)}( \mathbf{x}_n) \left( [\pi_k^{(N)},1]_N - [\pi_k^{(N)},1]_N + \scp{ \pi_k^{(N)} }{ 1 } \right) \\ 
			& = r_n \sum_{k=1}^K \pi_k^{(N)}( \mathbf{x}_n) \left( [\pi_k^{(N)},1]_N - \epsilon_k^{(N)} \right)
	\end{aligned}
	\end{equation} 
	for $n=1,\dots,N$. 
	Because of \ref{item:restr_space}, we can assume that $\varphi_1 \equiv 1$ and consequently $\pi_1^{(N)} = 1/\|1\|_N$. 
	This implies 
	\begin{equation} 
		[\pi_k^{(N)},1]_N 
			= \|1\|_N [\pi_k^{(N)},\pi_1^{(N)}]_N 
			= \|1\|_N \delta_{1,k}, 
	\end{equation} 
	since $\{\pi_k^{(N)}\}_{k=1}^K$ is assumed to be orthonormal with respect to the discrete inner product $[\cdot,\cdot]_N$. 
	Thus, the LS weights can further be rewritten as 
	\begin{equation} 
	\begin{aligned}
		w_n^{\mathrm{LS}} 
			& = r_n \sum_{k=1}^K \pi_k^{(N)}( \mathbf{x}_n) \left( \|1\|_N \delta_{1,k} - \epsilon_k^{(N)} \right) \\ 
			& = r_n \left( \|1\|_N \pi_1^{(N)}( \mathbf{x}_n) - \sum_{k=1}^K \epsilon_k^{(N)} \pi_k^{(N)}( \mathbf{x}_n) \right) \\ 
			& = r_n \left( 1 - \sum_{k=1}^K \epsilon_k^{(N)} \pi_k^{(N)}( \mathbf{x}_n) \right).
	\end{aligned}
	\end{equation} 
	Hence, the assertion ($w_n^{\text{LS}} > 0$ for all $n=1,\dots,N$) is equivalent to 
	\begin{equation}\label{eq:assertion2-thm}
		\sum_{k=1}^K \varepsilon_k^{(N)} \pi_k^{(N)}(\mathbf{x}_n) < 1, \quad n=1,\dots,N.
  	\end{equation} 
	At the same time, \cref{eq:cond-thm} and \cref{lem:2} imply that every element of the discrete orthonormal basis converges uniformly to the corresponding element of the continuous orthonormal basis. 
	In particular, for every $k=1,\dots,K$, the function sequence ${(\pi_k^{(N)})_{N \in \N} \subset \mathcal{F}_K(\Omega)}$ is uniformly bounded.\footnote{If $(y_n)_{n \in \N}$ is a convergent sequence in the normed vector space $(Y,\| \cdot \|)$, then $(y_n)_{n \in \N}$ is bounded. This is because for any $\varepsilon > 0$ we can find an $N \in \N$ such that $\| y - y_n \| \leq \varepsilon$ for all $n > N$, where $y$ denotes the limit of $(y_n)_{n \in \N}$. One can then choose $C = \| y \| + \max\{ \| y - y_1 \|, \dots, \| y - y_N \|, \varepsilon \}$ to get $\|y_n\| \leq \|y\| + \| y - y_n \| \leq C$ for all $n$, which shows that $(y_n)_{n \in \N}$ is bounded.}
	Thus, there exists a constant $C > 0$ such that 
	\begin{equation}
    		\sum_{k=1}^K \varepsilon_k^{(N)} \pi_k^{(N)}(\mathbf{x}_n) 
      		\leq C \sum_{k=1}^K \left| \varepsilon_k^{(N)} \right|, \quad n=1,\dots,N.
  	\end{equation} 
	Moreover, \cref{lem:1} implies $\lim_{N \to \infty} \epsilon_k^{(N)} = 0$ for all  $k=1,\dots,K$. 
	Hence, there exists an $N_0 \geq K$ such that 
	\begin{equation}
    		\left| \varepsilon_k^{(N)} \right| < \frac{1}{K C}, \quad k=1,\dots,K,
  	\end{equation}
  	for all $N \geq N_0$. 
  	Finally, this yields \cref{eq:assertion2-thm} and therefore the assertion. 
\end{proof}

A simple consequence of \cref{thm:main} is the subsequent corollary in which the special case of equidistributed data points is considered. 

\begin{corollary}\label{cor:main}
	Given are $\Omega \subset \R^d$, $\omega: \Omega \to \R_0^+$, and $\mathcal{F}_K(\Omega) \subset C(\Omega)$ such that the restrictions \ref{item:restr_weight} and \ref{item:restr_space} are satisfied, i.\,e., 
	\begin{enumerate}
		\item[(R2)]
		The weight function $\omega: \Omega \to \R_0^+$ is Riemann integrable and positive almost everywhere. 

		\item[(R3)] 
		The $K$-dimensional vector space $\mathcal{F}_K(\Omega)$ is spanned by a basis $\{ \varphi_k \}_{k=1}^K$ of continuous and bounded functions $\varphi_k:	\Omega \to \R$, $k=1,\dots,K$. 
		Furthermore, $\mathcal{F}_K(\Omega)$ contains constants. In particular, $1 \in \mathcal{F}_K(\Omega)$. 
	
	\end{enumerate} 
	Let $(\mathbf{x}_n)_{n \in \N} \subset \Omega$ be a equidistributed sequence with $\omega(\mathbf{x}_n) > 0$ for all $n \in \N$. 
	Then, there exists an $N_0 \in \N$ such that for all $N \geq N_0$ and discrete weights 
	\begin{equation}\label{eq:discr_weights}
		r_n = \frac{|\Omega| \omega(\mathbf{x}_n)}{N}, \quad n=1,\dots,N,
	\end{equation} 
	the corresponding LS-CF 
	\begin{equation}
		C^{\mathrm{LS}}_N[f] = \sum_{n=1}^N w_n^{\mathrm{LS}} f(\mathbf{x}_n) 
		\quad \text{with} \quad 
		\mathbf{w}^{\mathrm{LS}} = \argmin_{\Phi \mathbf{w} = \mathbf{m}} \ \norm{ R^{-1/2} \mathbf{w} }_{2}, 
	\end{equation} 
	where $R^{-1/2} = \diag\left( 1/\sqrt{r_1}, \dots, 1/\sqrt{r_N} \right)$, is positive and $\mathcal{F}_K(\Omega)$-exact. 
\end{corollary} 

\begin{proof} 
	Recall that the equidistributed sequence $(\mathbf{x}_n)_{n \in \N} \subset \Omega$ satisfies \cref{eq:equi} for all measurable bounded functions that are continuous almost everywhere and, by \cref{cor:equid}, is $\mathcal{F}_K(\Omega)$-unisolvent.
	In particular, \cref{eq:cond-thm} holds for $(\mathbf{x}_n)_{n \in \N}$ and the discrete weights $(r_n)_{n \in \N}$ defined as in \cref{eq:discr_weights}. 
	In combination with \ref{item:restr_weight} and \ref{item:restr_space}, \cref{thm:main} therefore implies the assertion. 
\end{proof} 

\begin{remark}
	\ref{item:restr_domain} is not necessary for \cref{cor:main} to hold. 
	However, following \cref{rem:eq-points}, \ref{item:restr_domain} ensures the existence---and simple construction---of an equidistributed sequence in $\Omega$. \end{remark} 
	
\begin{remark}
	\cref{cor:main} implies that if $(\mathbf{x}_n)_{n \in \N} \subset \Omega$ is equidistributed, then for sufficiently large $N$, there exists a positive and exact CF on the set $\{\mathbf{x}_n\}_{n=1}^N$. 
	As we will discuss in \cref{sub:interpol_CFs}, thus there also exists a positive interpolatory CF---predicted by \cref{thm:Tchakaloff}---on $\{\mathbf{x}_n\}_{n=1}^N$. 
	Such sets, on which a positive and interpolatory CF is supported, are referred to as \emph{Tchakaloff sets}. 
	Hence, \cref{cor:main} also implies that if $(\mathbf{x}_n)_{n \in \N} \subset \Omega$ is equidistributed, then for sufficiently large $N$, $\{\mathbf{x}_n\}_{n=1}^N$ is a Tchakaloff set. 
	Similar results were obtained in \cite{davis1967construction} and \cite{wilhelmsen1976nearest} for everywhere dense sequences. 
	However, \cref{cor:main} does not just provide us with a Tchakaloff set, but also tells us that a positive and exact CF can be obtained in the form of a simle weighted LS solution of the linear system \cref{eq:ex-system}. 
\end{remark}  
\section{Convergence and Error Analysis} 
\label{sec:error}

Here, we address the convergence of positive and exact LS-CFs.

\subsection{The Lebesgue Inequality and Convergence} 
\label{sub:error_Leb}

Assume that the positive LS-CF ${C_N}$ is exact for all functions from the function space $\mathcal{F}_K(\Omega)$.
Moreover, let $f:\Omega \to \R$ be a continuous function, and let us denote a best approximation of $f$ from $\mathcal{F}_K(\Omega)$ with respect to the $L^\infty(\Omega)$-norm by $\hat{s}$.
That is, 
\begin{equation}\label{eq:Lebesgue}
	\hat{s} = \argmin_{s \in \mathcal{F}_K(\Omega)} \norm{ f - s }_{L^{\infty}(\Omega)} 
	\quad \text{with} \quad 
	\norm{ f - s }_{L^{\infty}(\Omega)} = \sup_{\boldsymbol{x} \in \Omega} | f(\boldsymbol{x}) - s(\boldsymbol{x}) |. 
\end{equation} 
Then, the following error bound holds: 
\begin{equation}\label{eq:L-inequality} 
\begin{aligned}
	| C_N[f] - I[f] | 
		& \leq \| I \|_{\infty} \| f - \hat{s} \|_{L^{\infty}(\Omega)} 
			+ \| C_N \|_{\infty} \| f - \hat{s} \|_{L^{\infty}(\Omega)} \\ 
		& = \left( \| I \|_{\infty} + \| C_N \|_{\infty} \right) \left( \inf_{ s \in \mathcal{F}_K(\Omega) } \norm{ f - s }_{L^{\infty}(\Omega)} \right)
\end{aligned}
\end{equation} 
Inequality \cref{eq:L-inequality} is commonly known as the Lebesgue inequality; see, e.\,g., \cite{van2020adaptive} or \cite[Theorem 3.1.1]{brass2011quadrature}. 
It is most often encountered in the context of polynomial interpolation \cite{brutman1996lebesgue,ibrahimoglu2016lebesgue} but straightforwardly carries over to numerical integration.
In this context, the operator norms $\|I\|_\infty$ and $\|C_N\|_{\infty}$ are respectively given by 
\begin{eq}
	\|I\|_\infty 
		= \int_\Omega \omega(\boldsymbol{x}) \intd \boldsymbol{x} 
		= I[1], \quad 
	\|C_N\|_{\infty} 
		= \sum_{n=1}^N |w_n|. 
\end{eq}
Recall that the CF $C_N$ is positive and exact for constants (we assume that $\mathcal{F}_K(\Omega)$ contains constants). 
Thus, we have 
\begin{eq} 
	\|C_N\|_{\infty} = C_N[1] = I[1] = \|I\|_\infty. 
\end{eq}
In particular, this implies that the Lebesgue inequality \cref{eq:L-inequality} simplifies to 
\begin{eq}\label{eq:L-inequality2} 
	| C_N[f] - I[f] | 
		\leq 2 \| I \|_{\infty} \left( \inf_{ s \in \mathcal{F}_K(\Omega) } \norm{ f - s }_{L^{\infty}(\Omega)} \right).
\end{eq}
Based on \cref{eq:L-inequality}, we can note the following: 
Assume that we are given a sequence of positive CFs $(C_N)_{N \in \N}$ with $C_N$ being exact for $\mathcal{F}_K(\Omega)$, where $K = K(N)$. 
Sequences of CFs are usually referred to as \emph{cubature rules (CRs)}. 
Let $\mathcal{F}_K(\Omega) \subset \mathcal{F}_{K+1}(\Omega)$ for all $K \in \N$
and let $\bigcap_{K \in \N} \mathcal{F}_K(\Omega)$ be dense in $C(\Omega)$ with respect to the $L^{\infty}(\Omega)$-norm. 
If $K(N) \to \infty$ for $N \to \infty$, then $(C_N)_{N \in \N}$ converges to the continuous integral, $I$, for all continuous functions. 
That is, for all $f \in C(\Omega)$, 
\begin{eq} 
	C_N[f] \to I[f], \quad N \to \infty, \quad \text{in } (\R,|\cdot|). 
\end{eq} 
It should be stressed that the particular rate of convergence of $(C_N[f])_{N \in \N}$ to $I[f]$ depends on the (smoothness) of the function $f$ as well as the finite-dimensional function spaces $\mathcal{F}_K(\Omega)$. 
In particular, a more detailed error analysis based on \cref{eq:L-inequality2} relies on some knowledge about the quality of the $L^{\infty}(\Omega)$ best approximation from $\mathcal{F}_K(\Omega)$. 
Results of this flavor are usually referred to as Jackson-type theorems; the subject of constructive function theory \cite{natanson1961constructive}.

\subsection{Error Analysis for Analytic Functions} 
\label{sub:error_analytic}

For simplicity, we now restrict ourselves to analytic functions on the $d$-dimensional hypercube $\Omega = [0,1]^d$. 
Moreover, we assume that $(C_N)_{N \in \N}$ is a CR with $C_N$ being positive and exact for all $d$-dimensional polynomials up to total degree $m = m(N)$.\footnote{
The relation between the number of data points $N$ and the maximum total degree $m$ remains to be addressed. 
}
That is, $\mathcal{F}_K(\Omega) = \mathbb{P}_m(\R^d)$. 
In this case, the following result holds for the LS-CFs. 

\begin{lemma}\label{lem:conv}
	Let $f:[0,1]^d \to \R$ be analytic in an open set containing $[0,1]^d$. 
	Then the CR $(C_N)_{N \in \N}$, where $C_N$ is positive and $\mathbb{P}_m(\R^d)$-exact, with $m = m(N)$, satisfies 
	\begin{eq} 
		|I[f] - C_N[f]| = \mathcal{O} \left( \exp( -c m / \sqrt{d} ) \right)
	\end{eq} 
	for some constant $c > 0$.
\end{lemma}

\begin{proof}
	Since $f$ is analytic in an open set containing $[0,1]^d$ it can be approximated by a $d$-dimensional polynomial of total degree $m$ as 
	\begin{eq}\label{eq:conv_proof} 
		\inf_{s \in \mathbb{P}_m(\R^d)} \| f - s \|_{L^{\infty}([0,1]^d)} 
			= \mathcal{O} \left( \exp( -c m / \sqrt{d} ) \right);
	\end{eq} 
	see \cite[Equation 5.8]{nakatsukasa2018approximate}. 
	Here, $c$ is a constant depending on the location of the singularity of $f$ (if there is any) nearest to $[0,1]^d$ with respect to the radius of the Bernstein ellipse. 
	Finally, combining \cref{eq:conv_proof} with the Lebesgue inequality \cref{eq:L-inequality2} immediately yields the assertion.
\end{proof}

Let us briefly address the relation between the number of data points $N$ and the maximum total degree $m$ for which the LS-CF is positive.  
First, it should be noted that the dimension of $\mathbb{P}_m(\R^d)$ is $K = \binom{d+m}{d}$. 
This implies the asymptotic relation ${\lim_{m \to \infty} K/m^d = 1/d!}$; in particular, $K = \mathcal{O}(m^d)$. 
Furthermore, in \cref{sub:num_ratio}, we observe the relation between $K$ and $N$ to be of the form $N = \mathcal{O}(K^2)$. 
Assuming this relation, we get the following version of \cref{lem:conv}. 

\begin{corollary} 
	Let $f:[0,1]^d \to \R$ be analytic in an open set containing $[0,1]^d$ and let $(C_N)_{N \in \N}$ be a CR with $C_N$ being positive and $\mathbb{P}_m(\R^d)$-exact. 
	Assume that we have the asymptotic relation $N = \mathcal{O}(K^2)$ with $K = \dim \mathbb{P}_m(\R^d)$, then 
	\begin{eq}\label{eq:conv_rate} 
		|I[f] - C_N[f]| = \mathcal{O} \left( \exp( -c N^{1/2d} / \sqrt{d} ) \right)
	\end{eq} 
	for some constant $c > 0$.
\end{corollary}

\begin{proof} 
	Recall that $g(N) \in \mathcal{O}(h(N))$ if and only if there exists a constant $M>0$ such that $g(N) \leq M h(N)$ for sufficiently large $N$. 
	Henceforth, let $c>0$ be a generic constant.
	The assertion follows from noting that there exists a constant $M > 0$ such that 
	\begin{eq} 
	\begin{aligned}
		|I[f] - C_N[f]| 
			& \leq M \exp( -c m / \sqrt{d} ) \\
			& \leq M \exp( -c K^{1/d} / \sqrt{d} ) \\
			& \leq M \exp( -c N^{1/2d} / \sqrt{d} ) 
	\end{aligned}
	\end{eq} 
	for sufficiently large $N$. 
	Here, the first inequality follows from \cref{lem:conv}, the second one from the fact that $K = \mathcal{O}(m^d)$, and the third one from the assumption that $N = \mathcal{O}(K^2)$.  
\end{proof}

We already point out that for the positive interpolatory CFs discussed in \cref{sub:interpol_CFs}, we have $N = K$ (rather than just $N = \mathcal{O}(K^2)$) and therefore 
\begin{eq}\label{eq:conv_rate2} 
	|I[f] - C_N[f]| = \mathcal{O} \left( \exp( -c N^{1/d} / \sqrt{d} ) \right)
\end{eq} 
instead of \cref{eq:conv_rate}. 

\begin{remark}
	The above error analysis can easily be extended whenever a relation analogue to \cref{eq:conv_proof} is available. 
	See \cite{schultz1969multivariate,bagby2002multivariate} for a discussion of some other classes of functions $f$, domains $\Omega$, and function spaces $\mathcal{F}_K(\Omega)$.
\end{remark}

\begin{remark} 
	The 'impossibility' theorem proved in \cite{platte2011impossibility} states that any procedure for approximating univariate functions from equally spaced samples that converges exponentially fast must also be exponentially ill-conditioned.  
	Observe that for $d=1$, \cref{eq:conv_rate2} implies root-exponential convergence for the LS-CF, which allows us to avoid such inherent stability issues. 
\end{remark}

\begin{remark} 
	It was argued in some recent works \cite{trefethen2017cubature,trefethen2017multivariate,trefethen2021exactness}---also see \cite{bos2018bernstein}---that for a certain class of functions (analytic in the hypercube with singularities outside), the Euclidean degree should be considered instead of the total or maximum degree. 
	However, we did not observe any advantage in using the Euclidean degree in our numerical tests. 
\end{remark} 
\section{Some Applications}
\label{sec:applications}

We discuss two applications of the provable positive and exact LS-CFs. 
These address the simple construction of positive high-order CRs (\cref{sub:const_procedure}) and positive interpolatory CFs (\cref{sub:interpol_CFs}).  
In both cases, the procedure again only relies on basic linear algebra operations.

\subsection{A Simple Procedure To Construct Positive High-Order Cubature Rules}
\label{sub:const_procedure}

Henceforth, we make the same assumptions as in \cref{cor:main}. 
That is, $\Omega \subset \R^d$ is bounded with a boundary of measure zero and $\omega: \Omega \to \R_0^+$ is Riemann integrable and positive almost everywhere. 
Moreover, let $(\mathbf{x}_n)_{n \in \N}$ be an equidistributed sequence in $\Omega$ with $\omega(\mathbf{x}_n) > 0$ for all $n \in \N$. 
Given is a sequence of increasing function spaces $( \mathcal{F}_K(\Omega) )_{K \in \N}$ with $1 \in \mathcal{F}_K(\Omega) \subset C(\Omega)$, $\mathcal{F}_K(\Omega) \subset \mathcal{F}_{K+1}(\Omega)$, and $\dim \mathcal{F}_K(\Omega) = K$ for all $K \in \N$. 
For simplicity, we shall assume that $\bigcap_{K \in \N} \mathcal{F}_K(\Omega)$ is dense in $C(\Omega)$ with respect to the $L^{\infty}(\Omega)$-norm. 
Following the discussion in \cref{sub:error_Leb}, this will ensure convergence of the subsequent CR for all continuous functions. 

Under these assumptions, the procedure works as follows: 
In every step, we increase the dimension $K$and find a positive and $\mathcal{F}_K(\Omega)$-exact CF by increasing the number of data points until the corresponding LS-CF $C_N$ is positive. 
\cref{algo:LS-CF} contains an informal summary of the procedure for fixed $K$. 

\begin{algorithm}
\caption{Constructing Positive High-Order Cubature Formulas}
\label{algo:LS-CF}
\begin{algorithmic}[1]
    \State{$N = K$, $r = 0$, and $w_{\text{min}} = 0$} 
    \While{$r < K$ or $w_{\text{min}} < 0$} 
      	\State{$X_N = \{\mathbf{x}_n\}_{n=1}^N$} 
      	\State{Compute the matrix $\Phi = \Phi(X_N)$}
      	\State{Compute the rank of $\Phi$: $r = \text{rank}(\Phi)$} 
      	\If{$r = K$}
      		\State{Compute the LS weights $\mathbf{w}^{\text{LS}}$ as in \cref{eq:LS-sol}} 
      		\State{Determine the smallest weight: $w_{\text{min}} = \min( \mathbf{w}^{\text{LS}} )$}
		\EndIf 
	\State{$N = 2N$} 
    \EndWhile 
\end{algorithmic}
\end{algorithm} 

\cref{algo:LS-CF} is ensured to terminate due to the theoretical findings presented in \cref{sec:prelim} and \cref{sec:LS}. 
In particular, \cref{cor:equid} tells us that for a sufficiently large number of (equidistributed) nodes, these will be $\mathcal{F}_K(\Omega)$-unisolvent. 
This is equivalent to the rows of $\Phi$ to be linearly independent ($\rank \Phi = K$). 
Hence, $r = K$ is ensured for sufficiently large $N$. 
At the same time, \cref{cor:main} implies that the LS weights $\mathbf{w}^{\text{LS}}$ are positive for a sufficiently large $N$. 

\begin{remark}[Monte Carlo CFs]\label{rem:MC_correction}
	The LS-CFs discussed above can be seen as high-order corrections to QMC methods \cite{caflisch1998monte,dick2013high} in the case that low-discrepancy data points are used.
	Recall that the weights in QMC integration are $w_n = |\Omega| \omega(\mathbf{x}_n) / N$. 
	At the same time, the LS weights are explicitly given by 
	$w_n^{\mathrm{LS}} = r_n \sum_{k=1}^K \pi_k( \mathbf{x}_n ; \mathbf{r}) I[ \pi_k( \, \cdot \, ; \mathbf{r} ) ]$, see \cref{eq:LS-sol-explicit}. 
	Here, $\{ \pi_k( \, \cdot \, ; \mathbf{r}) \}_{k=1}^K$ is a discrete orthonormal basis. 
	For fixed $K$ and an increasing number of data points, $\pi_k( \mathbf{x}_n ; \mathbf{r}) I[ \pi_k( \, \cdot \, ; \mathbf{r} ) ]$ converges to the Kronecker delta $\delta_{1,k}$. 
	Hence, the difference between the QMC and LS weights converges to zero. 
\end{remark}

\begin{remark}[Exact Integration of Discrete LS Approximations]\label{rem:DLS_approx}
	The LS-CF $C_N[f]$ corresponds to exact integration of the following discrete LS approximation of $f$ from the function space $\mathcal{F}_K(\Omega) = \mathrm{span} \{\varphi_1,\dots,\varphi_K\}$ (assuming it is unique): 
	\begin{equation} 
		\hat{f}(\boldsymbol{x}) = \sum_{k=1}^K c_k \varphi_k(\boldsymbol{x}) 
		\quad \text{s.t.} \quad 
		\norm{ R^{1/2} ( \Phi^T \mathbf{c} - \mathbf{f} ) }_2 
		\ \text{is minimized}, 
	\end{equation} 
	where $\mathbf{c} = (c_1,\dots,c_K)^T$.
	That is, $C_N[f] = I[\hat{f}]$. 
	This can be noted by representing $\hat{f}$ with respect to to a basis of discrete orthonormal functions $\{ \pi_k(\cdot;\mathbf{r}) \}_{k=1}^K$: 
	\begin{equation} 
		\hat{f}(\boldsymbol{x}) = \sum_{k=1}^K [ f, \pi_k(\cdot;\mathbf{r}) ]_N \pi_k(\boldsymbol{x};\mathbf{r}) 
	\end{equation} 
	Integration therefore yields 
	\begin{equation} 
		I[\hat{f}] 
			= \sum_{k=1}^K [ f, \pi_k(\cdot;\mathbf{r}) ]_N I[\pi_k(\cdot;\mathbf{r})] 
			= \sum_{n=1}^N \left( \sum_{k=1}^K \pi_k(\mathbf{x}_n;\mathbf{r}) I[\pi_k(\cdot;\mathbf{r})] \right) f(\mathbf{x}_n) 
			= C_N[f]. 
	\end{equation} 
	The last equality follows from \cref{eq:LS-sol-explicit}. 
	Building upon this connection, in \cite{migliorati2020stable} high-order randomized CFs for independent random points were constructed. 
	These were shown to be positive and exact with a high probability if the number of (random) data points is sufficiently larger than $K$. 
	In particular, it was stated that the proportionality between $N$ and $K$ should be at least quadratic. 
	This is in accordance with the results presented here. 
	Also the fundamental work \cite{sloan1995polynomial} on hyperinterpolation should be mentioned in this context.
\end{remark}

\subsection{Constructing Positive Interpolatory Cubature Formulas} 
\label{sub:interpol_CFs}

Next, we describe how provable positive and exact LS-CFs can be used to construct interpolatory CFs with the same properties. 
In contrast to LS-CFs, interpolatory CFs use a smaller subset of $N = K$ data points, where $K$ denotes the dimension of the function space $\mathcal{F}_K(\Omega)$ for which the original LS-CF is exact. 
In fact, there exist many different approaches to this task. 
For instance, in \cite{bremer2010nonlinear} a nonlinear optimization procedure is used to downsample formulas one data point at a time (although the work focuses on the one-dimensional case). 
It is also possible to formulate \cref{eq:LS-sol} as a basis pursuit problem which can then be solved by linear programming tools \cite{glaubitz2020stableCFs}.
Another option is Caratheodory--Tchakaloff subsampling \cite{piazzon2017caratheodoryLS,bos2019catchdes}, which may be implemented using linear (or quadratic) programming. 
Finally, NNLS \cite{lawson1995solving} could be used to recover a sparse nonnegative weight vector (see \cite{huybrechs2009stable} for the univariate case and \cite{sommariva2015compression} for the multivariate case). 
Another procedure is a method due to Steinitz \cite{steinitz1913bedingt} (also see \cite{davis1967construction,wilson1969general}). 
While this method might be less efficient than some of the aforementioned approaches, it is fairly simple and---as the construction of the positive LS-CFs---only relies on basic operations from linear algebra. 
Details on Steinitz' method can found in \cref{app:Steinitz}. 
\section{Numerical Results} 
\label{sec:numerical} 

\begin{figure}[tb]
	\centering
  	\begin{subfigure}[b]{0.45\textwidth}
    		\includegraphics[width=\textwidth]{%
      		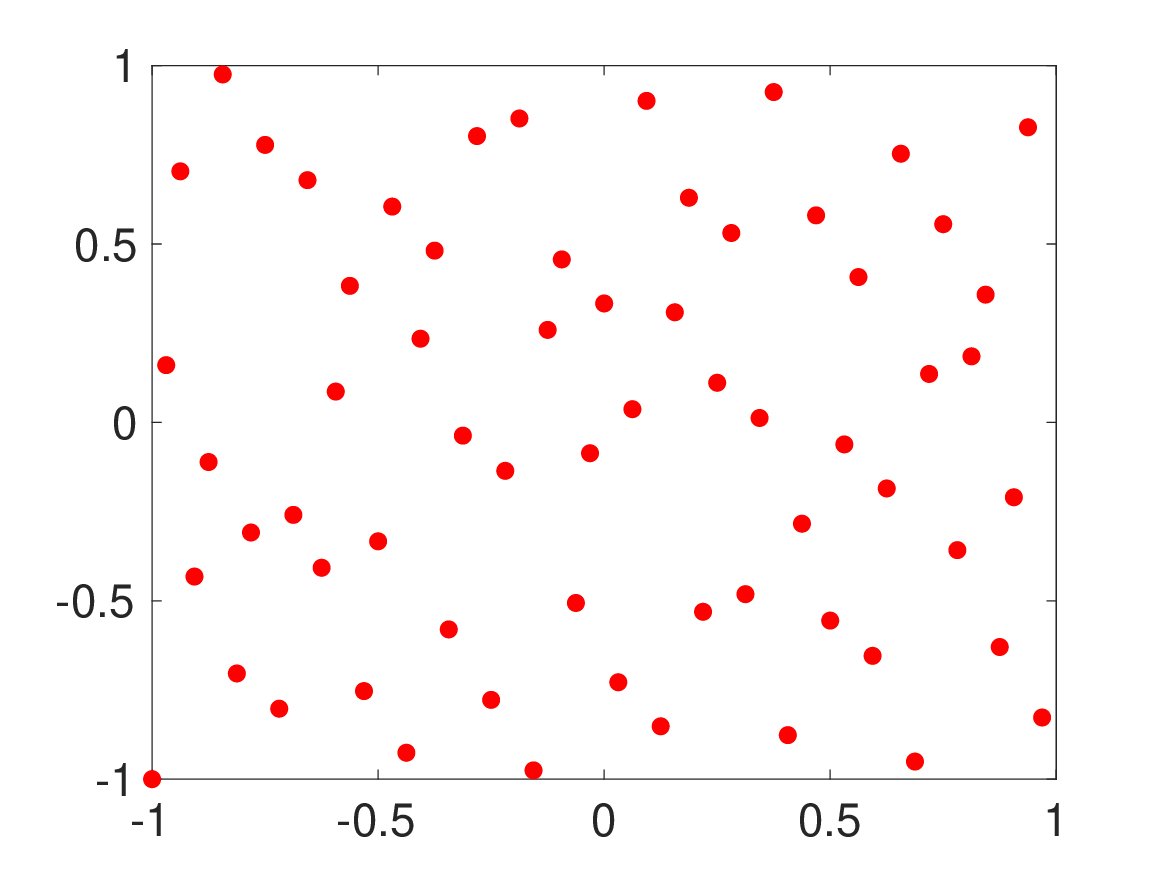} 
    		\caption{Halton points}
    		\label{fig:points_Halton}
  	\end{subfigure}%
	~ 
	\begin{subfigure}[b]{0.45\textwidth}
		\includegraphics[width=\textwidth]{%
      		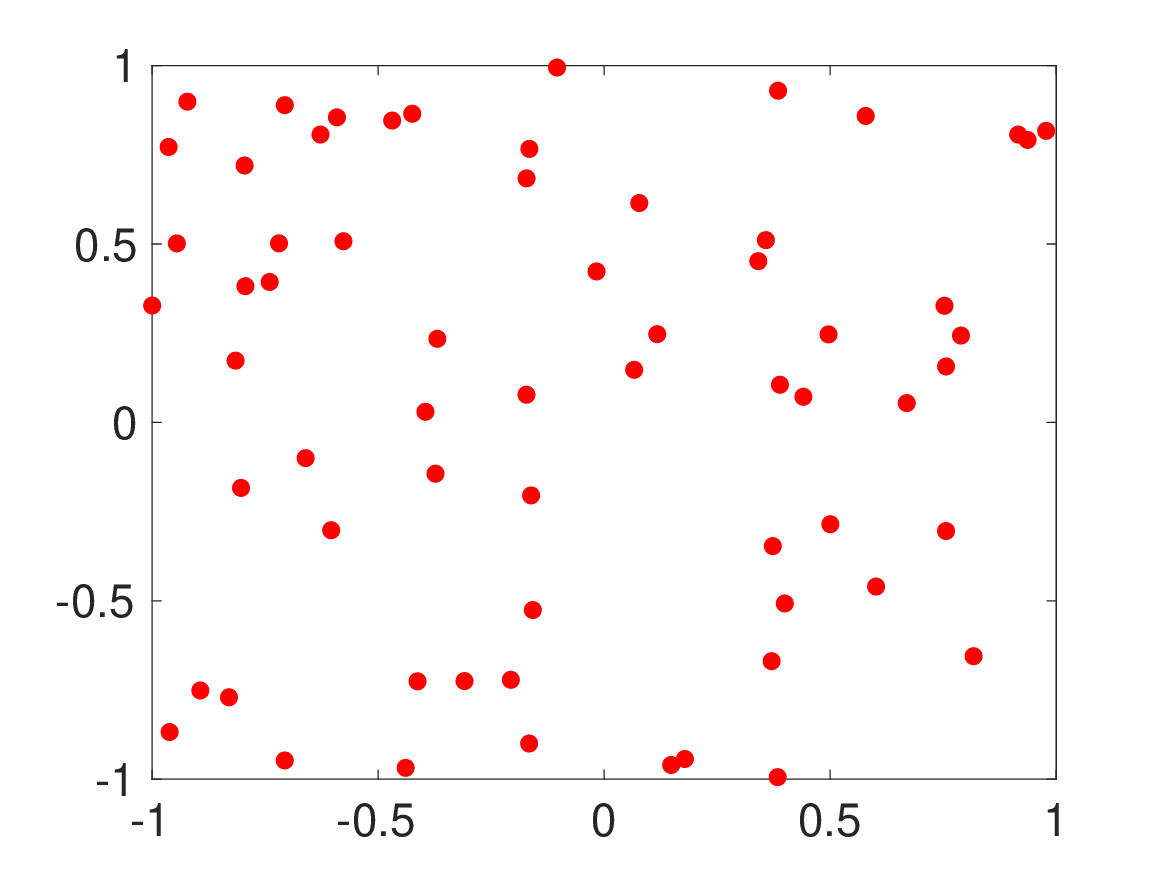} 
    		\caption{Random points}
    		\label{fig:points_ranom}
  	\end{subfigure}%
  	\caption{Illustration of ($N=64$) Halton and random points on $\Omega = [-1,1]^2$}
  	\label{fig:points}
\end{figure}

We now come to present several numerical tests to illustrate our theoretical findings as well as to demonstrate the performance of the positive high-order LS-CFs and corresponding interpolatory CFs. 
All experiments were performed using a \emph{2.6 GHz 6-Core Intel Core i7} processor with 32 GB of random access memory (RAM). 
The MATLAB code used to generate the subsequent numerical results can be found on GitHub\footnote{See \url{https://github.com/jglaubitz/positive_CFs}}. 
We consider two different types of data points: 
(1) Halton points, which are deterministic and from a low-discrepancy sequence\footnote{Recall that low-discrepancy sequences are a subclass of equidistributed sequences.}  \cite{halton1960efficiency,niederreiter1992random,kuipers2012uniform}, and  
(2) uniformly distributed random points.  
An illustration for these points is provided by \cref{fig:points}. 

\begin{remark}\label{rem:conditioning}
To avoid ill-conditioning caused by a poor choice of $\{\varphi_k\}_{k=1}^K$ we recommend to use continuous orthonormal basis functions satisfying 
\begin{eq}
	\int_{\Omega} \varphi_k(\boldsymbol{x}) \varphi_l(\boldsymbol{x}) \intd \boldsymbol{x} = \delta_{kl}, \quad k,l=1,\dots,K, 
\end{eq} 
to formulate the linear system \cref{eq:ex-system} that is solved in the LS sense to obtain the weights of the LS-CFs. 
Observe that in this case 
\begin{eq}
	\lim_{N \to \infty} \frac{1}{N} \Phi \Phi^T 
		= \lim_{N \to \infty} \left( \frac{1}{N} \sum_{n=1}^N \varphi_k(\mathbf{x}_n) \varphi_l(\mathbf{x}_n) \right)_{k,l=1}^K 
		= \left( \int_{\Omega} \varphi_k(\boldsymbol{x}) \varphi_l(\boldsymbol{x}) \intd \boldsymbol{x} \right)_{k,l=1}^K 
		= I,
\end{eq} 
where $I$ denotes the $(K \times K)$ identity matrix. 
Then, we further find\footnote{If $A \in \R^{K \times K}$ has eigenvalues $\lambda_1,\dots,\lambda_K$ and $p$ is a polynomial, then the matrix $p(A)$ has eigenvalues $p(\lambda_1),\dots,p(\lambda_K)$.}
\begin{eq}
	\lim_{N \to \infty} \kappa(\Phi) 
		= \lim_{N \to \infty} \frac{\lambda_{\rm max}(\Phi \Phi^T)}{\lambda_{\rm min}(\Phi \Phi^T)} 
		= \lim_{N \to \infty} \frac{\lambda_{\rm max}(\frac{1}{N} \Phi \Phi^T)}{\lambda_{\rm min}(\frac{1}{N} \Phi \Phi^T)} 
		= 1 
\end{eq} 
for the condition number of $\Phi$, where $\lambda_{\rm max}(A)$ and $\lambda_{\rm min}(A)$ are the largest and smallest eigenvalue of the matrix $A$. 
\end{remark}

\subsection{Comparison of Different Subsampling Methods} 
\label{sub:comp_subsample}

We start by debate the subsampling method used to obtain a positive interpolatory CF from a given positive LS-CF, both exact for the same function space $\mathcal{F}_K(\Omega) \subset C(\Omega)$. 
A few options to construct such positive interpolatory CFs were discussed in \cref{sub:interpol_CFs}. 
An exhaustive comparison of all these approaches would exceed the scope of this work, but we at least provide a rudimentary comparison of three especially simple methods: 
(i) Steinitz' method, see \cite{steinitz1913bedingt,davis1967construction,wilson1969general} as well as \cref{app:Steinitz}; 
(ii) NNLS \cite{lawson1995solving} (also see \cite{huybrechs2009stable} for the univariate case and \cite{sommariva2015compression} for the multivariate case); and 
(iii) basis pursuit formulated as a linear programming problem \cite{glaubitz2020stableCFs}. 

\begin{figure}[tb]
	\centering
  	\begin{subfigure}[b]{0.4\textwidth}
		\includegraphics[width=\textwidth]{%
      		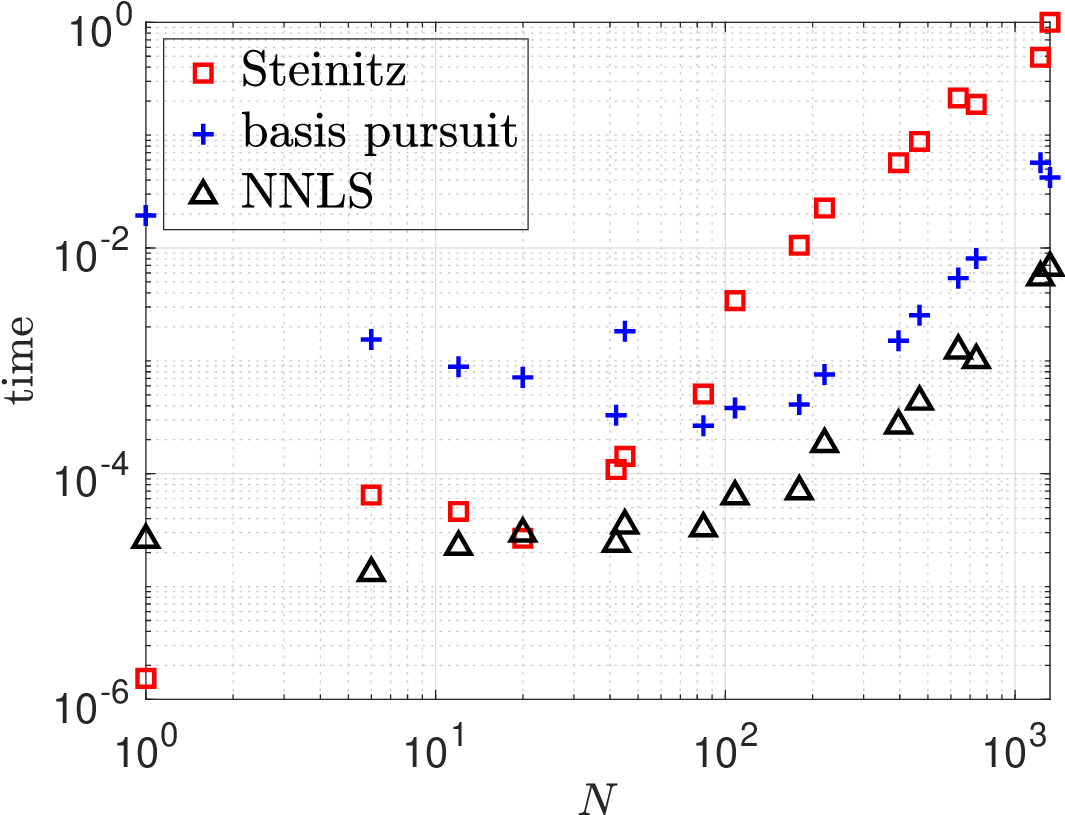} 
    		\caption{Efficiency, Halton}
    		\label{fig:comparison_efficiency_Halton}
  	\end{subfigure}%
	~ 
	\begin{subfigure}[b]{0.4\textwidth}
		\includegraphics[width=\textwidth]{%
      		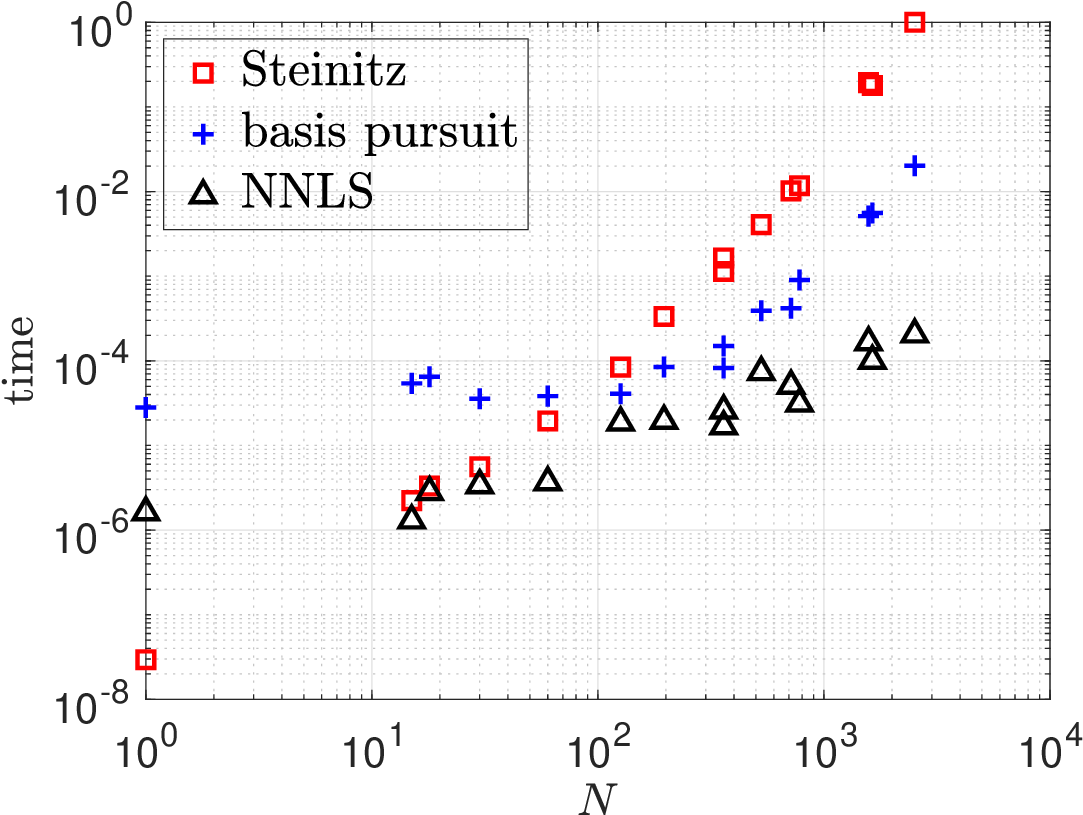} 
    		\caption{Efficiency, random}
    		\label{fig:comparison_efficiency_random}
  	\end{subfigure}%
  	\\
  	\begin{subfigure}[b]{0.4\textwidth}
		\includegraphics[width=\textwidth]{%
      		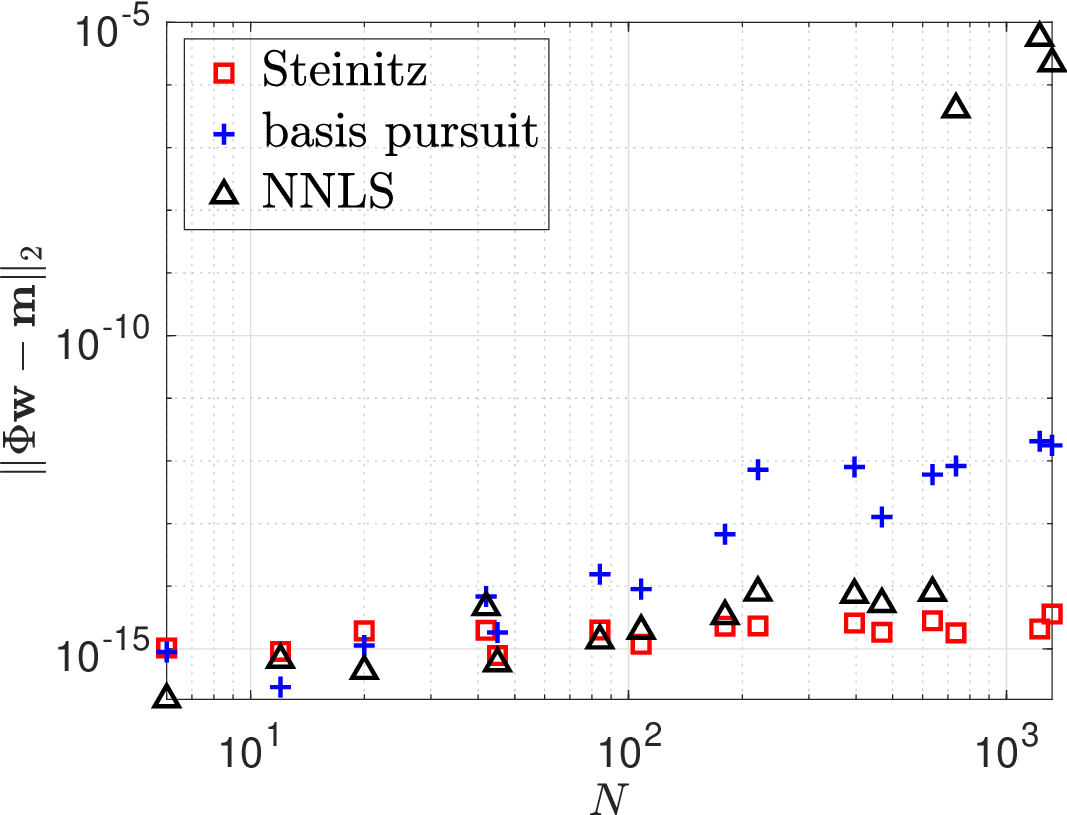} 
    		\caption{Exactness, Halton}
    		\label{fig:comparison_error_Halton}
  	\end{subfigure}%
  	~ 
	\begin{subfigure}[b]{0.4\textwidth}
		\includegraphics[width=\textwidth]{%
      		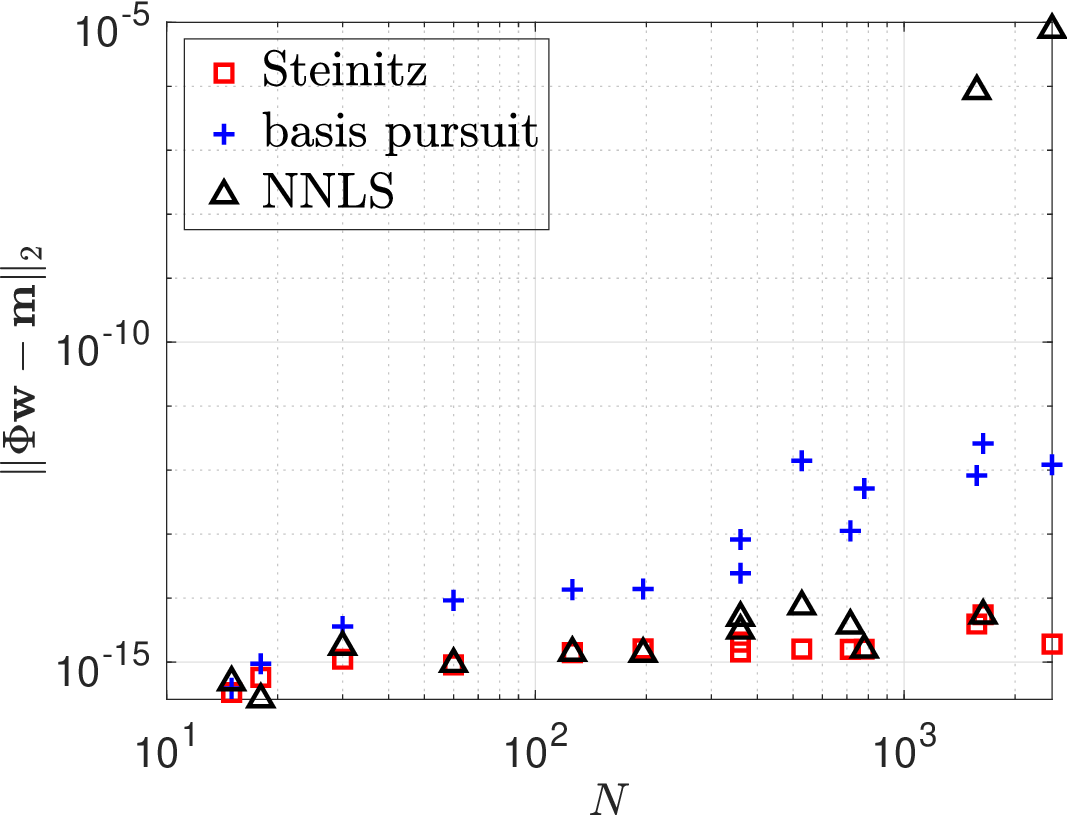} 
    		\caption{Exactness, random}
    		\label{fig:comparison_error_random}
  	\end{subfigure}%
  	\caption{
		A comparison between different subsampling methods to construct a positive interpolatory CF from a positive LS-CF. 
		All tests are performed for the domain $\Omega = [-1,1]^2$ with $\omega \equiv 1$. 
	}
  	\label{fig:subsampling_comparison}
\end{figure}

\cref{fig:subsampling_comparison} provides a comparison of the interpolatory CFs resulting from these three methods. 
The underlying positive LS-CFs was constructed for $\Omega = [-1,1]^2$ with $\omega \equiv 1$ and to be exact for algebraic polynomials up to a total degree $m \in \{0,1,\dots,14\}$. 
For this case all three methods produced positive and interpolatory ($N = K$) CFs. 
Thus, in \cref{fig:subsampling_comparison} we only focus on the efficiency and ``exactness" of these methods. 
Efficiency is measured by the time it took the respective method to produce a positive interpolatory CF from a given positive LS-CF. 
The normalized times are displayed in \cref{fig:comparison_efficiency_Halton,fig:comparison_efficiency_random}. 
Based on these results it might be argued that NNLS is the most efficient method, while the Steinitz approach is the least efficient one. 
On the other hand, ``exactness" is measured by the error in the exactness conditions, that is present in the produced interpolatory CF. 
This is measured by the residual of the moment conditions, $\|\Phi \mathbf{w} - \mathbf{m}\|_2$.
Such errors can be introduced due to computing in finite arithmetics (rounding errors) or a method not arriving at a solution within the performed number of iterations. 
We observe in \cref{fig:comparison_error_Halton,fig:comparison_error_random} that especially the NNLS method suffers from undesirably high errors. 
Regarding exactness, we observe that the Steinitz method performs best. 
We also report that the basis pursuit approach sometimes did not produce a positive interpolatory CF for total degrees $m > 14$.   
Henceforth, we restrict ourselves to positive interpolatory CFs constructed by the Steinitz method, although we observe it to be inferior in terms of efficiency. 
This does not take into account that these methods are often used to solve partial differential equations and must be fast.
A more detailed investigation and comparison of different subsampling strategies would therefore be of interest. 

\begin{remark} 
	In our numerical tests, we used the Matlab built-routine \href{https://www.mathworks.com/help/matlab/ref/lsqnonneg.html}{lsqnonneg}. 
	However, it should be pointed out that there are NNLS methods available, which might perform even faster. 
	See \cite{dessole2020accelerating} and references therein as well as \cite{slawski2022web}.  
\end{remark}

\subsection{Ratio Between $N$ and $K$} 
\label{sub:num_ratio} 

We now invetsigate the relation between the number of data points $N$ and the dimension of the function space $K$ for which the LS-CF is positive. 
Recall that \cref{cor:main} stated that the (weighted) LS-CF is ensured to be positive for fixed $K$ if a sufficiently large number $N$ of (equidistributed) data points is used. 
Let $N(K)$ denote the lowest number of data points for the LS-CF (with discrete weights as in \cref{cor:main}) to be positive.
Numerically, we found the relation between $K$ and $N(K)$ to be of the form $N(K) = C K^s$ with $s$ being close to or even below $2$ in many different cases.

\begin{figure}[tb]
	\centering
	\begin{subfigure}[b]{0.32\textwidth}
		\includegraphics[width=\textwidth]{%
      		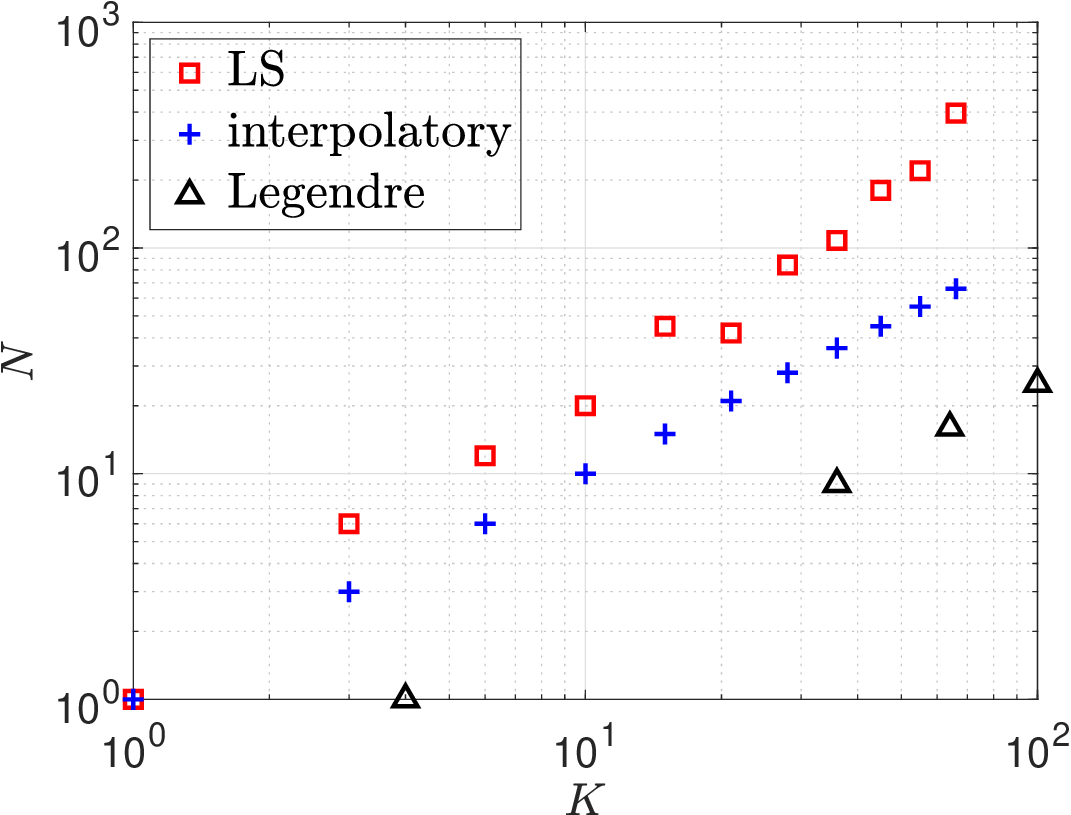} 
    		\caption{Halton, polynomials}
    		\label{fig:ratio_dim2_cube_1_algebraic_Halton}
  	\end{subfigure}%
	~
	\begin{subfigure}[b]{0.32\textwidth}
		\includegraphics[width=\textwidth]{%
      		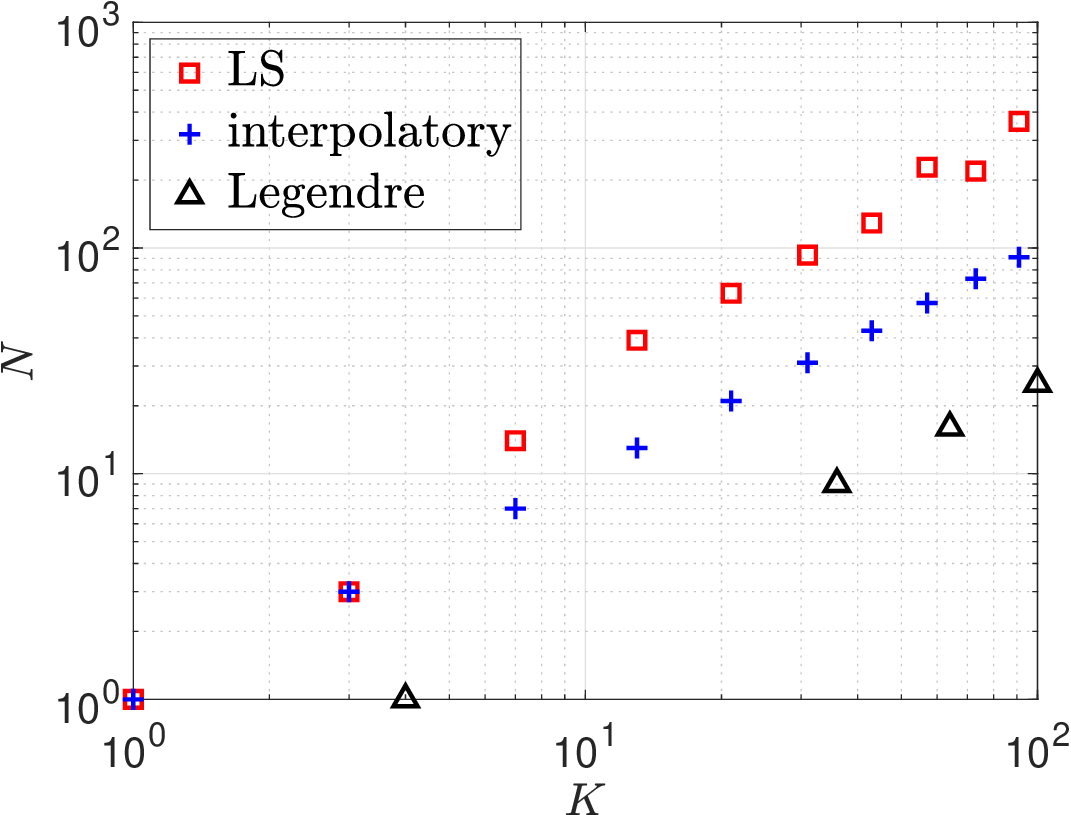} 
    		\caption{Halton, trigonometric}
    		\label{fig:ratio_dim2_cube_1_trig_Halton}
  	\end{subfigure}%
	~ 
	\begin{subfigure}[b]{0.32\textwidth}
		\includegraphics[width=\textwidth]{%
      		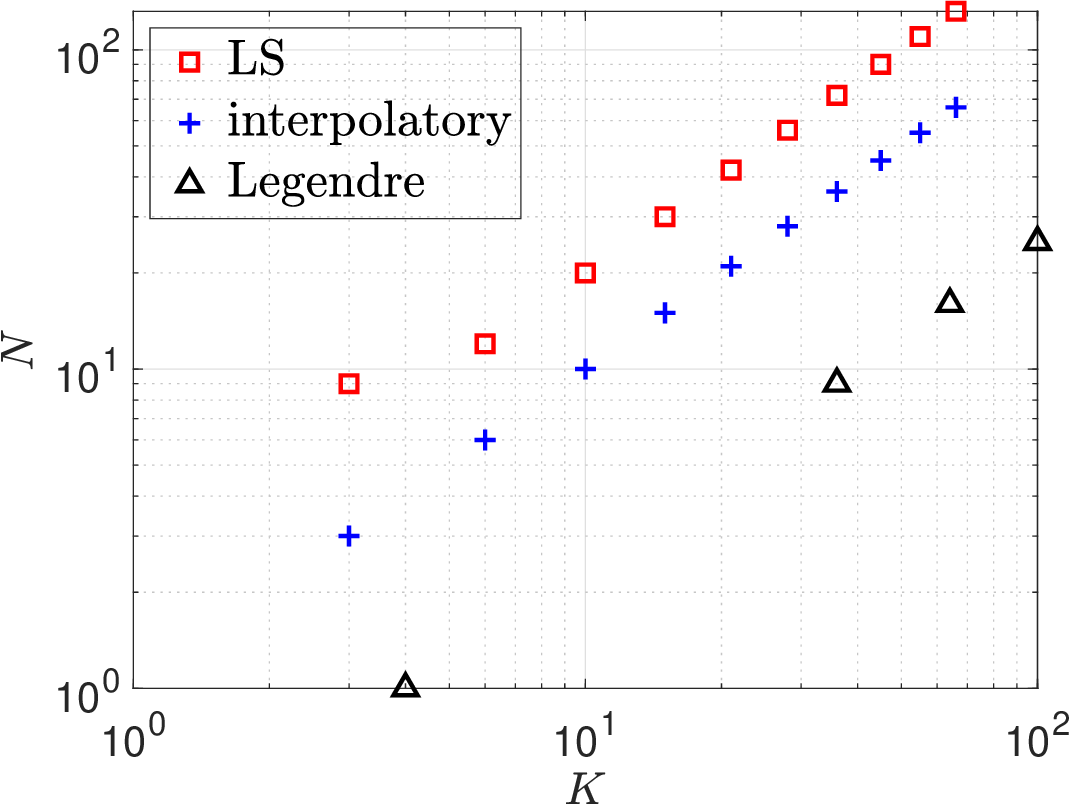} 
    		\caption{Halton, cubic PHS-RBFs}
    		\label{fig:ratio_dim2_cube_1_cubic_Halton}
  	\end{subfigure}%
	\\ 
	\begin{subfigure}[b]{0.32\textwidth}
		\includegraphics[width=\textwidth]{%
      		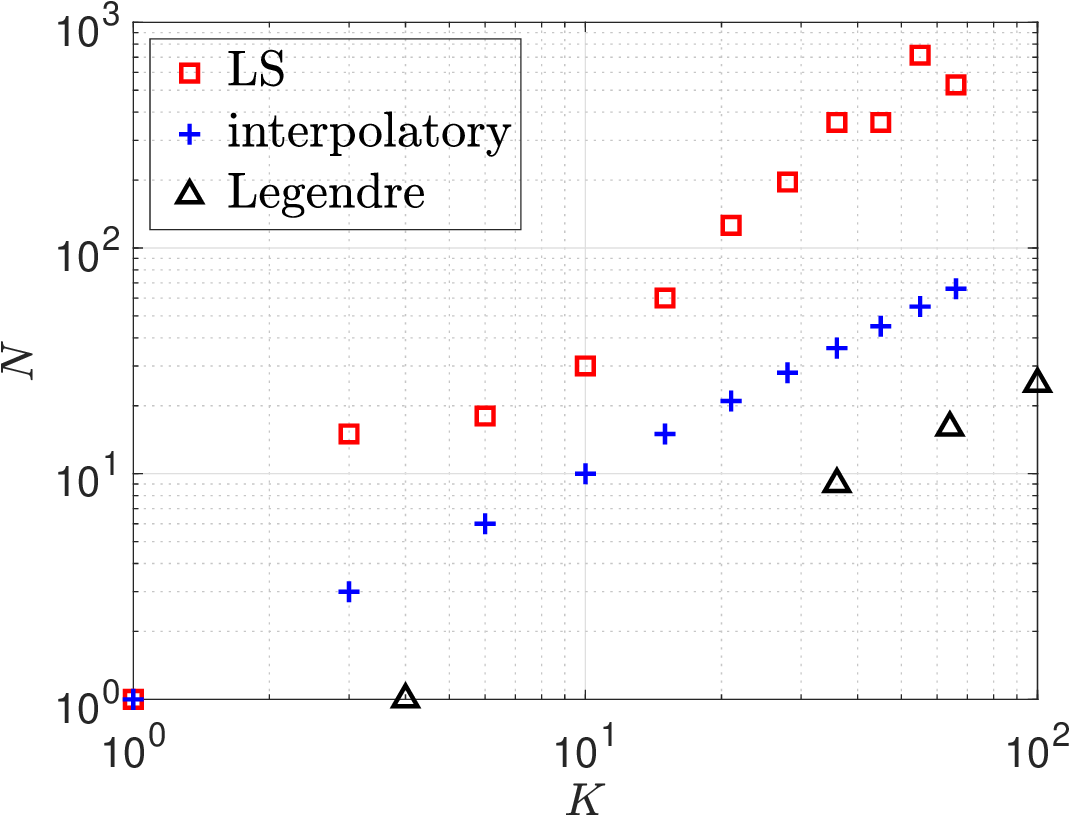} 
    		\caption{Random, polynomials}
    		\label{fig:ratio_dim2_cube_1_algebraic_random}
  	\end{subfigure}%
	~ 
	\begin{subfigure}[b]{0.32\textwidth}
		\includegraphics[width=\textwidth]{%
      		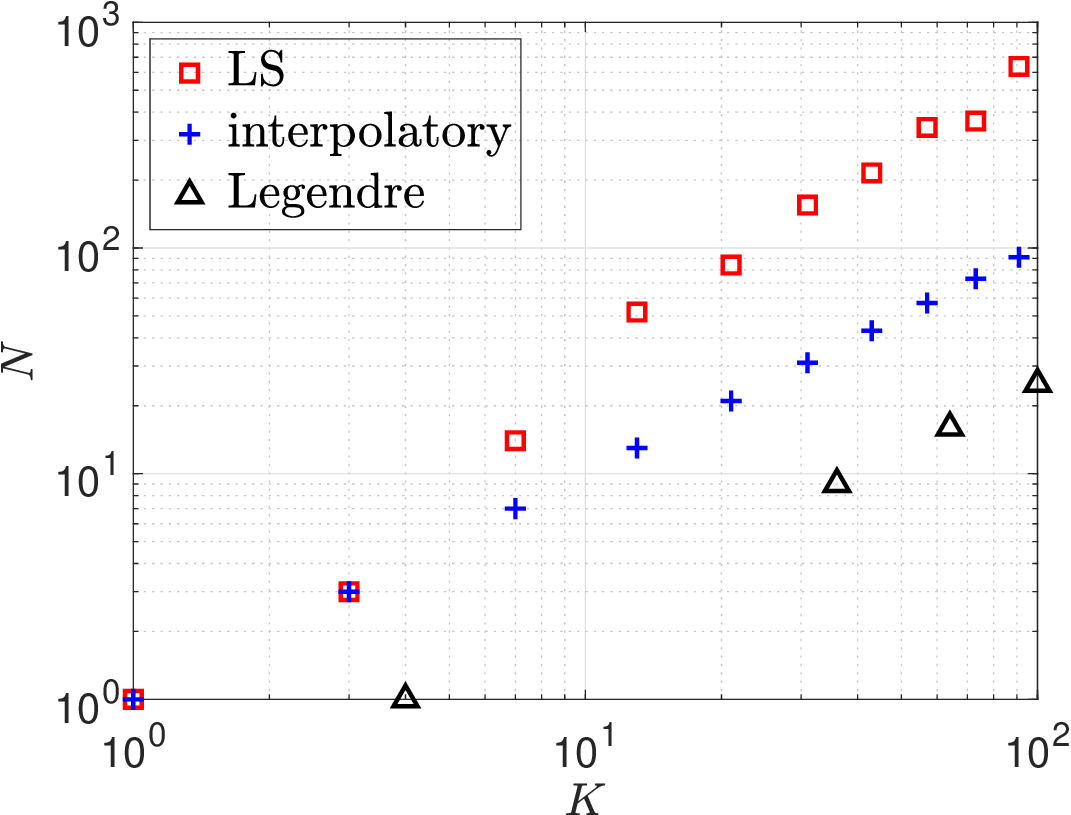} 
    		\caption{Random, trigonometric}
    		\label{fig:ratio_dim2_cube_1_trig_random}
  	\end{subfigure}%
	~
  	\begin{subfigure}[b]{0.32\textwidth}
		\includegraphics[width=\textwidth]{%
      		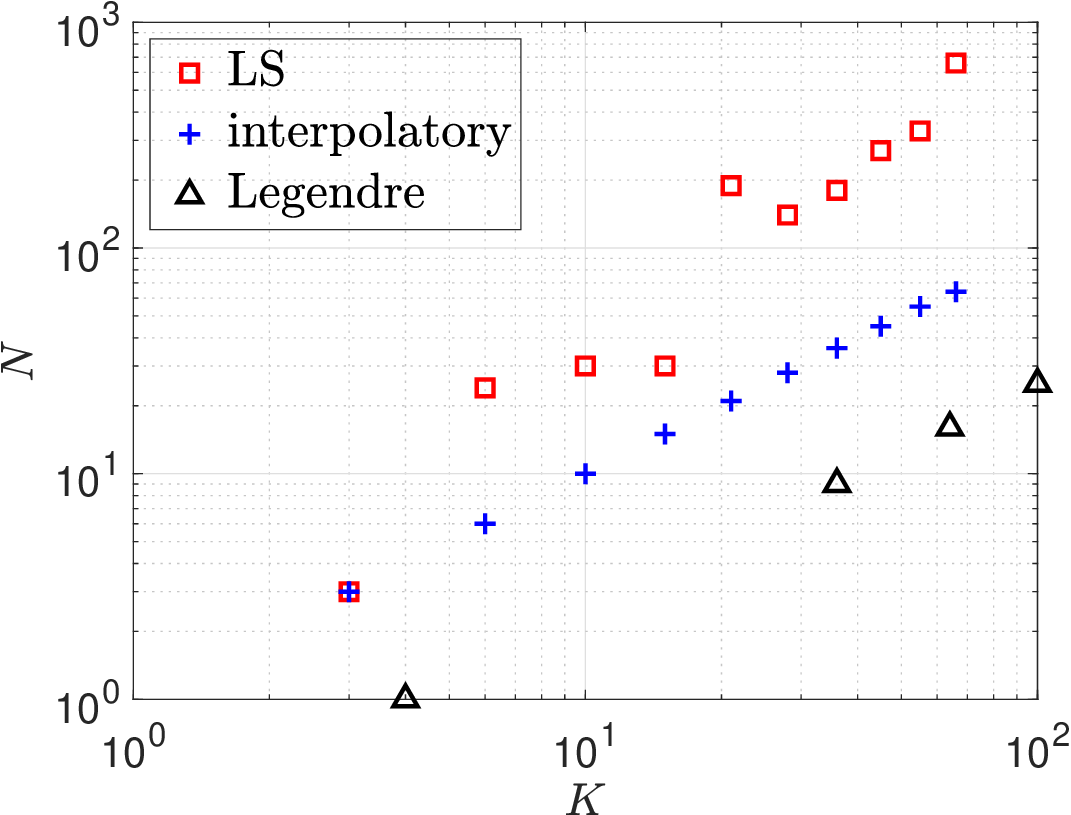} 
    		\caption{Random, cubic PHS-RBFs}
    		\label{fig:ratio_dim2_cube_1_cubic_random}
  	\end{subfigure}%
  	\caption{Ratio between $K$ and $N = N(K)$ for $\Omega = [-1,1]^2$ and $\omega \equiv 1$.  
	Compared are the LS-CF, the interpolatory CF obtained by subsampling, and the product Legendre rule.}
  	\label{fig:ratio_2d_cube_1}
\end{figure}

\begin{table}[tb]
\renewcommand{\arraystretch}{1.3}
\centering 
  	\begin{tabular}{c c c c c c c c c} 
		\toprule 
    		\multicolumn{9}{c}{LS-CF for the Cube Based on Algebraic Polynomials} \\ \hline 
    		\multicolumn{3}{c}{} & \multicolumn{3}{c}{$\omega \equiv 1$} & & \multicolumn{2}{c}{$\omega(\boldsymbol{x}) = (1-x_1^2)^{1/2} \dots (1-x_q^2)^{1/2}$} \\ \hline 
    		$q$ & \multicolumn{2}{c}{} & Legendre & Halton & random & & Halton & random \\ \hline 
		$2$ & s  & & 1.9e-0  	& 1.9e-0 	& 1.2e-0 
					  & 				& 1.9e-0 	& 1.7e-0 \\
    		   	   & C & &  3.0e-1	& 9.9e-2 	& 3.4e-0 
			   		 & 				& 9.2e-2 	& 1.3e-0 \\ \hline 
		$3$ & s  & & 2.8e-0 	& 1.4e-0 	& 1.4e-0 
					 & 				& 2.4e-0 	& 1.8e-0 \\
    		       & C & & 2.1e-1 	& 4.4e-1 	& 2.9e-0 
		       		 & 				& 2.9e-3 	& 7.1e-1 \\ \hline 
		 \\	  
	\end{tabular} 
	\\ 
  	\begin{tabular}{c c c c c}  
    		\multicolumn{5}{c}{LS-CF for the Two-Dimensional Cube, $\omega \equiv 1$} \\ \hline 
    		$\mathcal{F}_K(\Omega)$ & \multicolumn{2}{c}{} & Halton & random \\ \hline 
		algebraic polynomials 		
			& s  & & 1.9e-0 	& 1.2e-0 \\
    		    & C & & 9.9e-2 	& 3.4e-0 \\ \hline 
		trigonometric polynomials 
			& s  & & 1.2e-0 & 1.3e-0 \\
    		    & C & & 1.3e-0 & 1.3e-0 \\ \hline 
		cubic PHS-RBFs 
			& s  & & 9.9e-1 & 1.6e-0 \\
    		    & C & & 2.0e-0 & 6.0e-1 \\ \hline 
		\bottomrule
	\end{tabular} 
\caption{LS fit for the parameters $C$ and $s$ in the model $N(K) = C K^s$}
\label{tab:LS-fit}
\end{table} 

\cref{fig:ratio_2d_cube_1} illustrates this for the hyper-cube $\Omega = [-1,1]^2$ with weight function $\omega \equiv 1$. 
The corresponding LS-CFs were constructed to be exact for algebraic polynomials up to a fixed total degree (\cref{fig:ratio_dim2_cube_1_algebraic_Halton,fig:ratio_dim2_cube_1_algebraic_random}), trigonometric polynomials up to a fixed degree (\cref{fig:ratio_dim2_cube_1_trig_Halton,fig:ratio_dim2_cube_1_trig_random}), and cubic polyharmonic RBFs (PHS-RBFs) augmented with a constant (\cref{fig:ratio_dim2_cube_1_cubic_Halton,fig:ratio_dim2_cube_1_cubic_random}). 
See \cite{fornberg2015primer} and references therein for more details on RBFs. 
\cref{fig:ratio_2d_cube_1} also reports on the ratio between $K$ and $N(K)$ for the product Legendre rule and the positive interpolatory CF. 
The positive interpolatory CF was obtained from the LS-CF by subsampling using the Steinitz method.

The numerically observed values for $C$ and $s$ in the assumed relation $N(K) = C K^s$ for some further cases are listed in \cref{tab:LS-fit}. 
The reported parameters $C$ and $s$ were obtained by performing an LS fit for these based on the values of $K$ and $N = N(K)$ for maximum total degree ${m \in \{0,1,\dots,10\}}$. 
We note that the results reported here appear to be in accordance with similar observations made in previous works for certain special cases; see \cite{wilson1970necessary,huybrechs2009stable,glaubitz2020stable,glaubitz2020stableQRs,glaubitz2020shock} (for univariate LS-QFs) and \cite{glaubitz2020stableCFs} (for multivariate LS-CFs based on polynomials). 
Interestingly, similar ratios were also observed in other contexts, such as discrete LS function approximations \cite{cohen2013stability,cohen2017optimal} and stable high-order randomized CFs based on these \cite{migliorati2020stable}.

\subsection{Polynomial Based Cubature Formulas on the Hyper-Cube}

We now investigate the accuracy of the positive high-order LS-CFs and corresponding interpolatory CFs. 
To this end, we consider the hyper-cube $\Omega = [0,1]^d$ with $\omega \equiv 1$ and the following Genz test functions \cite{genz1984testing} (also see \cite{van2020adaptive}): 
\begin{equation}\label{eq:Genz}
\begin{aligned}
	g_1(\boldsymbol{x}) 
		& = \prod_{i=1}^d \left( a_i^{-2} + (x_i - b_i)^2 \right)^{-1} \quad 
		&& \text{(product peak)}, \\
	g_2(\boldsymbol{x}) 
		& = \left( 1 + \sum_{i=1}^d a_i x_i \right)^{-(d+1)} \quad 
		&& \text{(corner peak)}, \\
	g_3(\boldsymbol{x}) 
		& = \exp \left( - \sum_{i=1}^d a_i^2 ( x_i - b_i )^2 \right) \quad 
		&& \text{(Gaussian)}
\end{aligned}
\end{equation} 
These functions are designed to have different difficult characteristics for numerical integration routines.
The vectors $\mathbf{a} = (a_1,\dots,a_q)^T$ and $\mathbf{b} = (b_1,\dots,b_q)^T$ respectively contain (randomly chosen) shape and translation parameters. 
For each case, the experiment was repeated $20$ times, with the vectors $\mathbf{a}$ and $\mathbf{b}$ being drawn randomly from $[0,1]^d$.

\begin{figure}[tb]
	\centering
	\begin{subfigure}[b]{0.32\textwidth}
		\includegraphics[width=\textwidth]{%
      		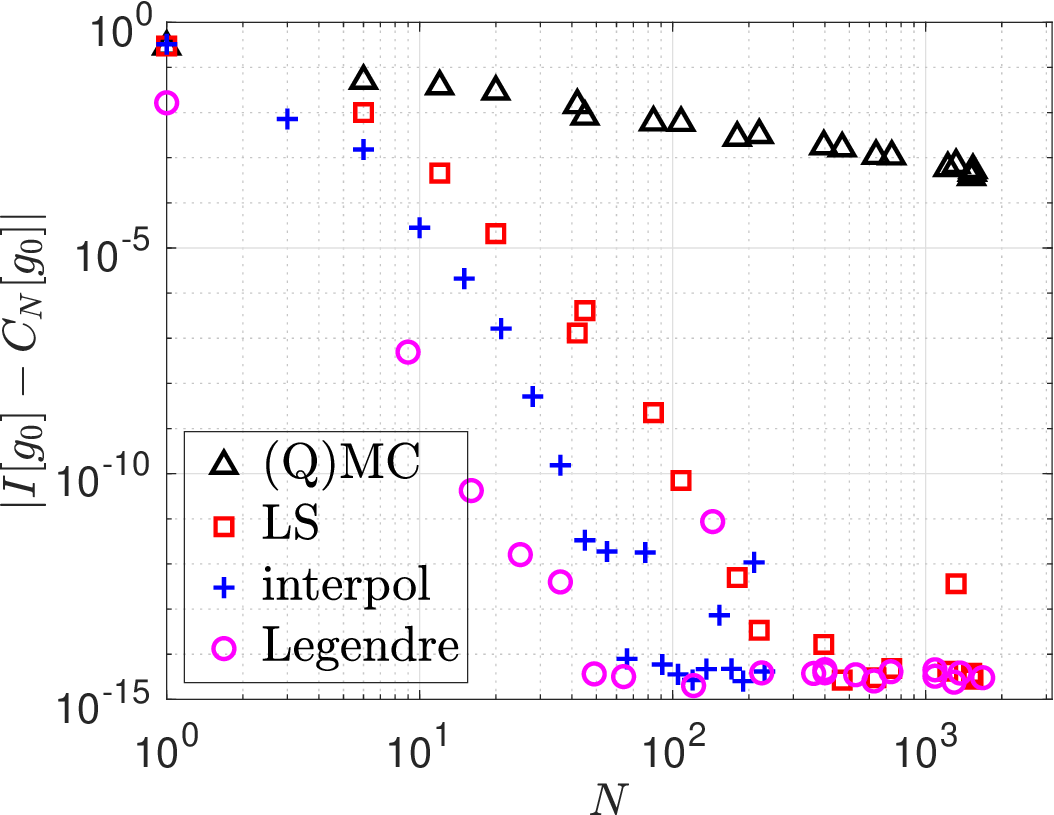} 
    		\caption{$g_1$, Halton}
    		\label{fig:accuracy_Genz1_dim2_1_Halton_Steinitz}
  	\end{subfigure}%
	~ 
	\begin{subfigure}[b]{0.32\textwidth}
		\includegraphics[width=\textwidth]{%
      		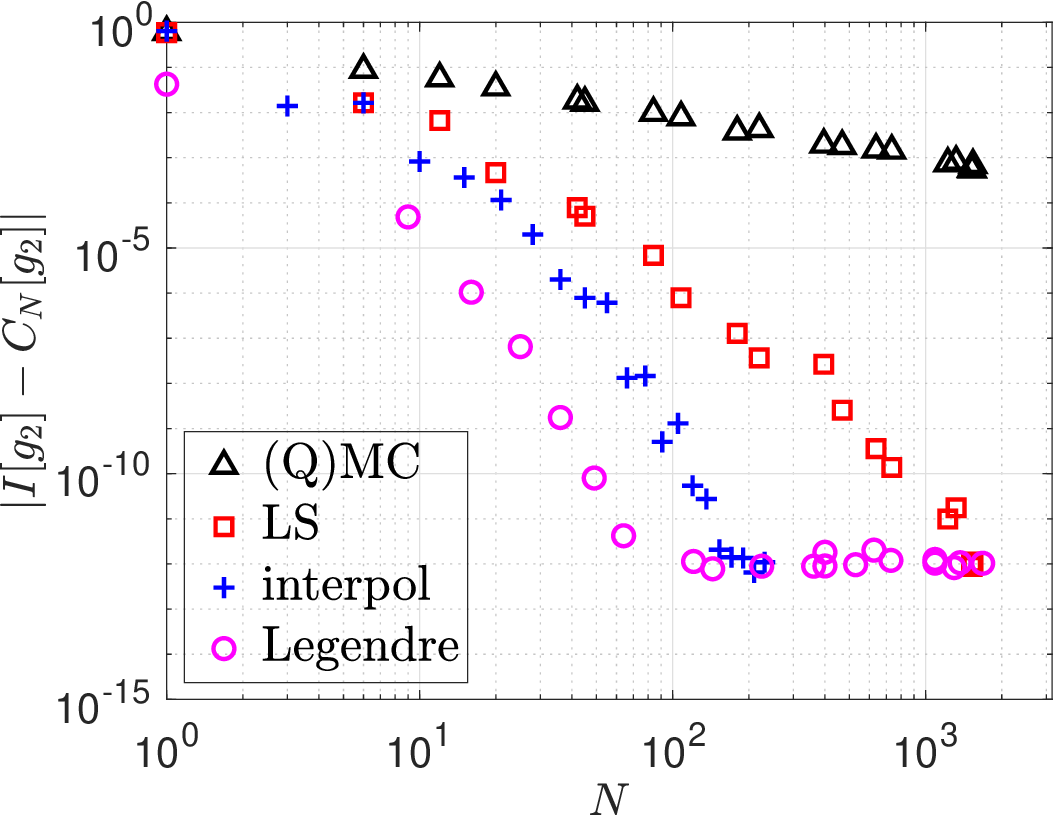} 
    		\caption{$g_2$, Halton}
    		\label{fig:accuracy_Genz2_dim2_1_Halton_Steinitz}
  	\end{subfigure}%
	~
	\begin{subfigure}[b]{0.32\textwidth}
		\includegraphics[width=\textwidth]{%
      		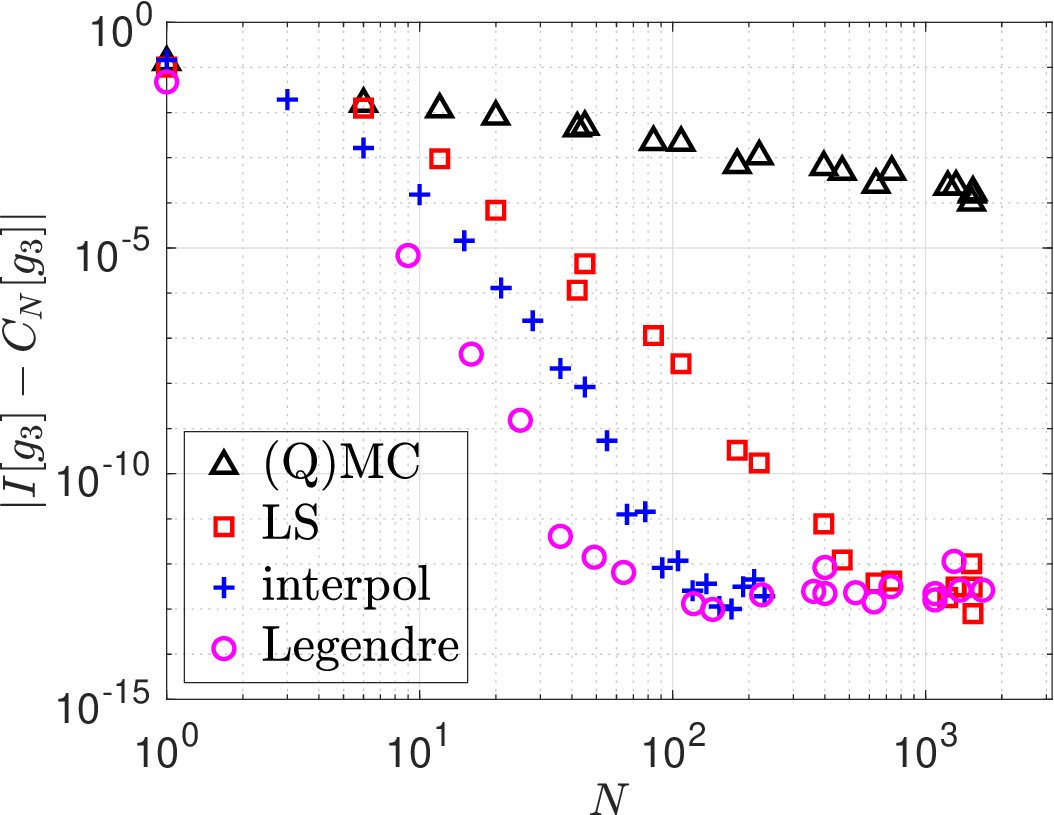} 
    		\caption{$g_3$, Halton}
    		\label{fig:accuracy_Genz3_dim2_1_Halton_Steinitz}
  	\end{subfigure}%
	\\ 
	\begin{subfigure}[b]{0.32\textwidth}
		\includegraphics[width=\textwidth]{%
      		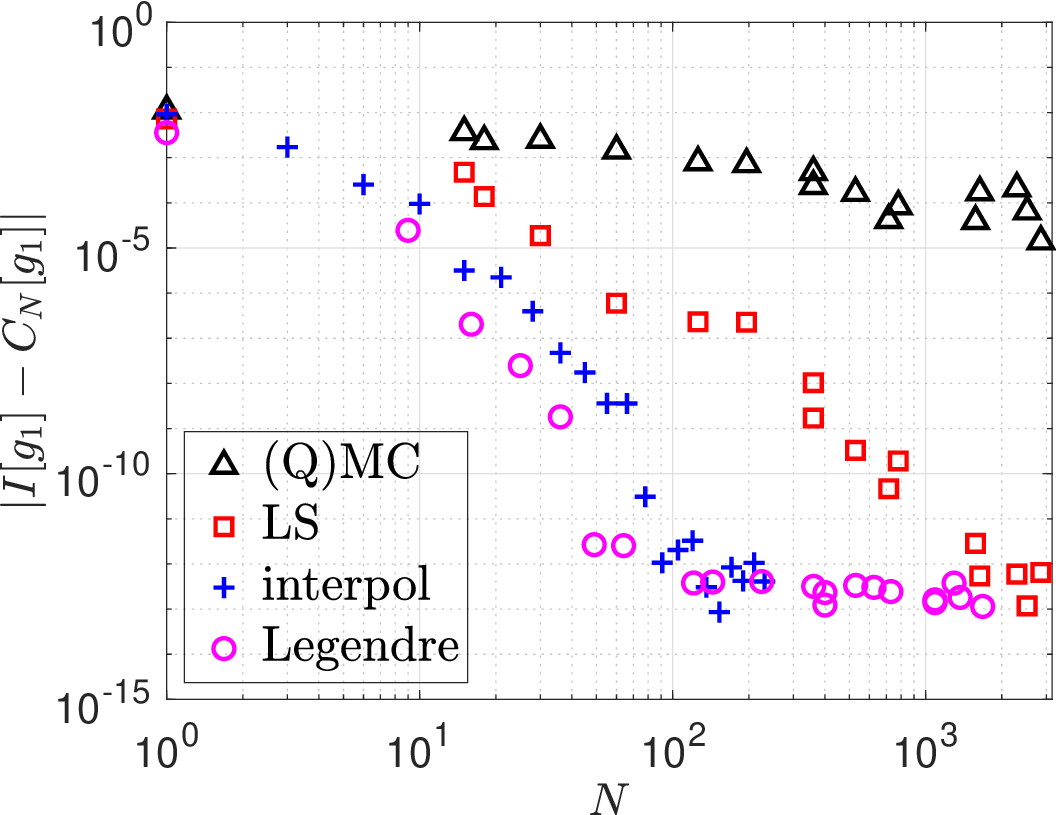} 
    		\caption{$g_1$, random}
    		\label{fig:accuracy_Genz1_dim2_1_random_Steinitz}
  	\end{subfigure}%
	~ 
	\begin{subfigure}[b]{0.32\textwidth}
		\includegraphics[width=\textwidth]{%
      		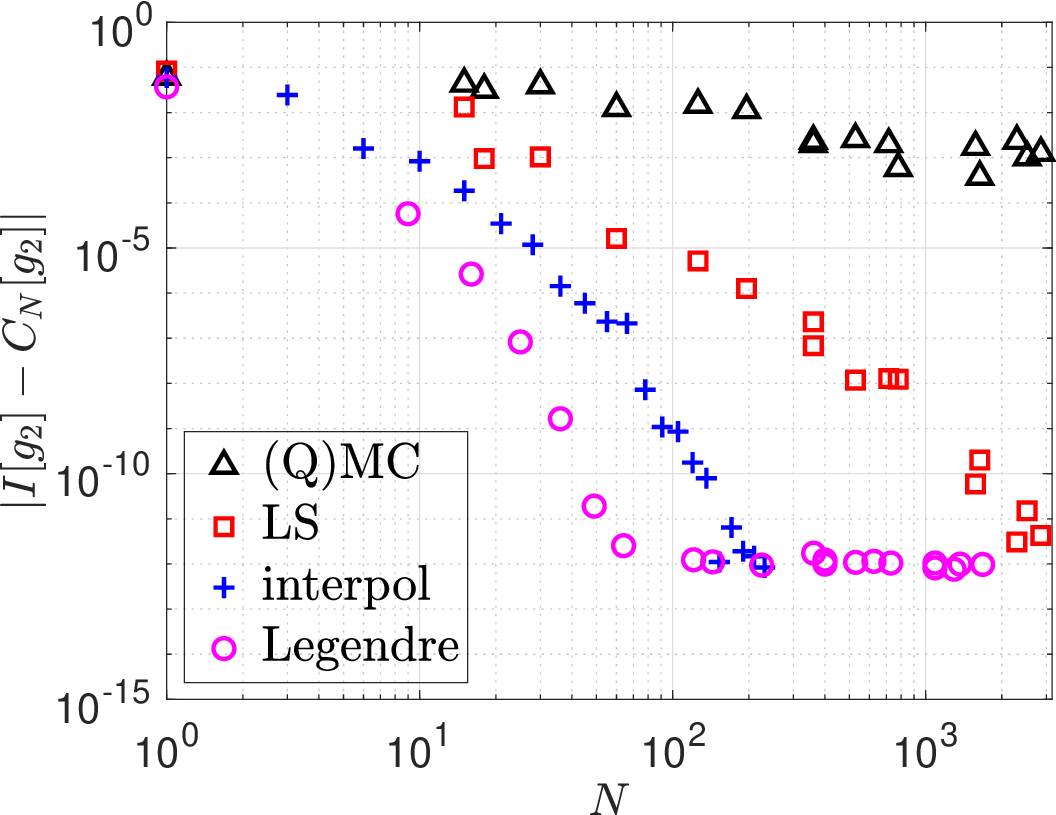} 
    		\caption{$g_2$, random}
    		\label{fig:accuracy_Genz2_dim2_1_random_Steinitz}
  	\end{subfigure}%
	~ 
	\begin{subfigure}[b]{0.32\textwidth}
		\includegraphics[width=\textwidth]{%
      		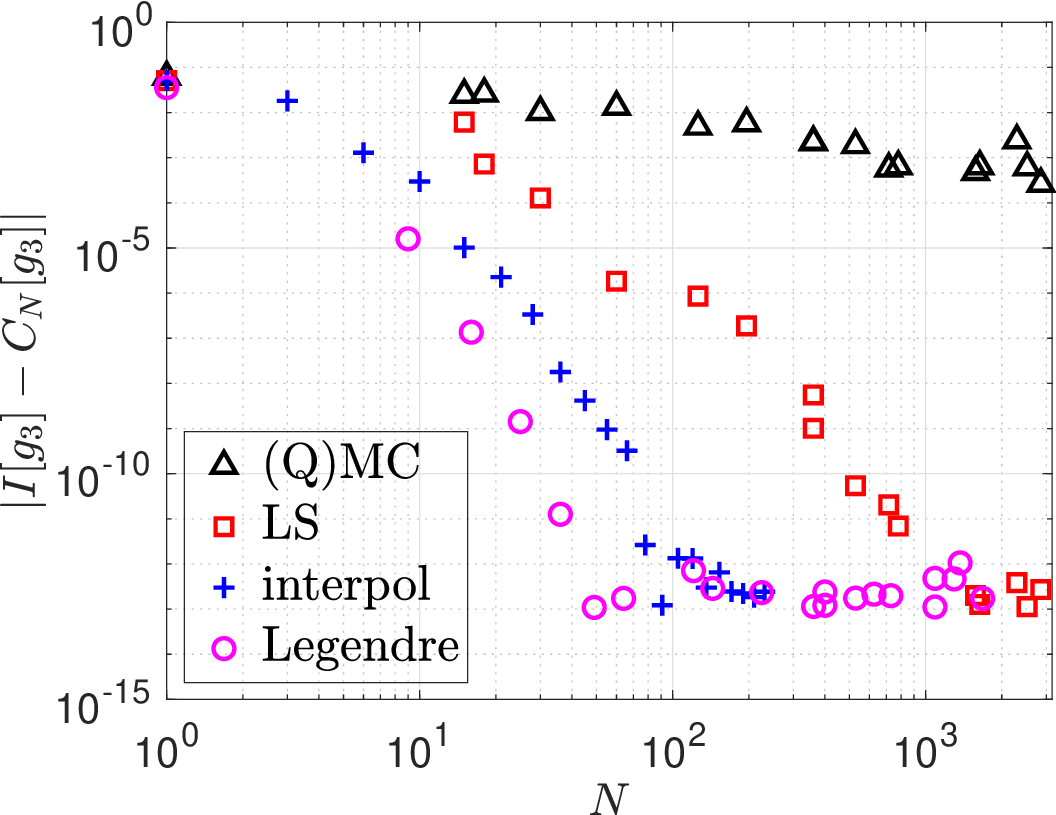} 
    		\caption{$g_3$, random}
    		\label{fig:accuracy_Genz3_dim2_1_random_Steinitz}
  	\end{subfigure}%
  	\caption{
		Errors for Genz' test functions on $\Omega = [0,1]^2$ with $\omega \equiv 1$. 
		The LS- and interpolatory CF are based on multivariate polynomials of increasing total degree ($\mathcal{F}_K(\Omega) = \mathbb{P}_m(\R^2)$). 
	}
  	\label{fig:errors_cube_1_dim2}
\end{figure}

\begin{figure}[tb]
	\centering
  	\begin{subfigure}[b]{0.32\textwidth}
		\includegraphics[width=\textwidth]{%
      		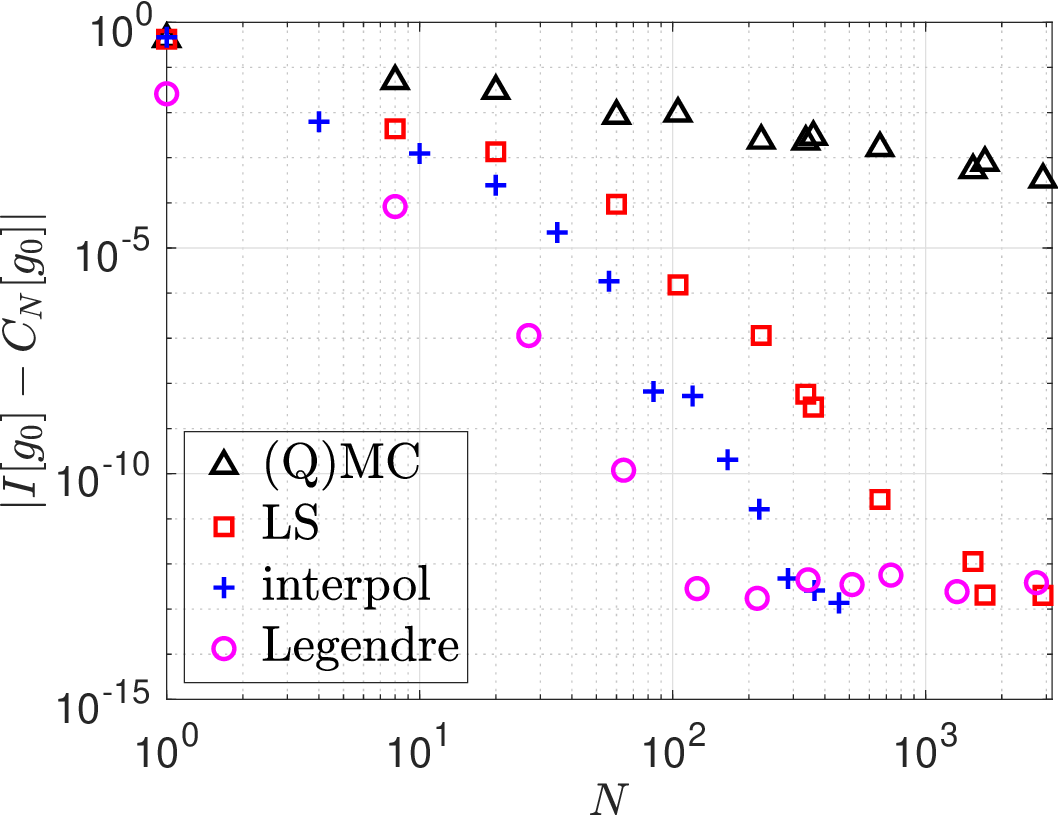} 
    		\caption{$g_1$, Halton}
    		\label{fig:accuracy_Genz1_dim3_1_Halton_Steinitz}
  	\end{subfigure}%
	~ 
	\begin{subfigure}[b]{0.32\textwidth}
		\includegraphics[width=\textwidth]{%
      		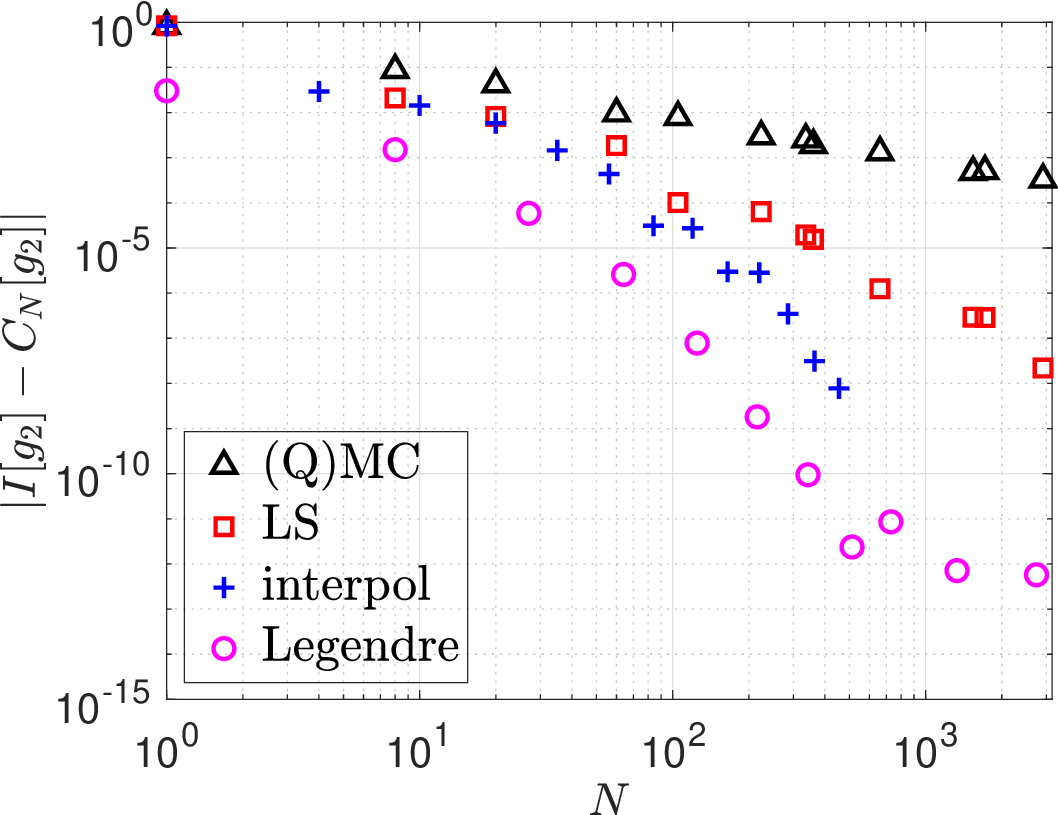} 
    		\caption{$g_2$, Halton}
    		\label{fig:accuracy_Genz2_dim3_1_Halton_Steinitz}
  	\end{subfigure}%
	~
	\begin{subfigure}[b]{0.32\textwidth}
		\includegraphics[width=\textwidth]{%
      		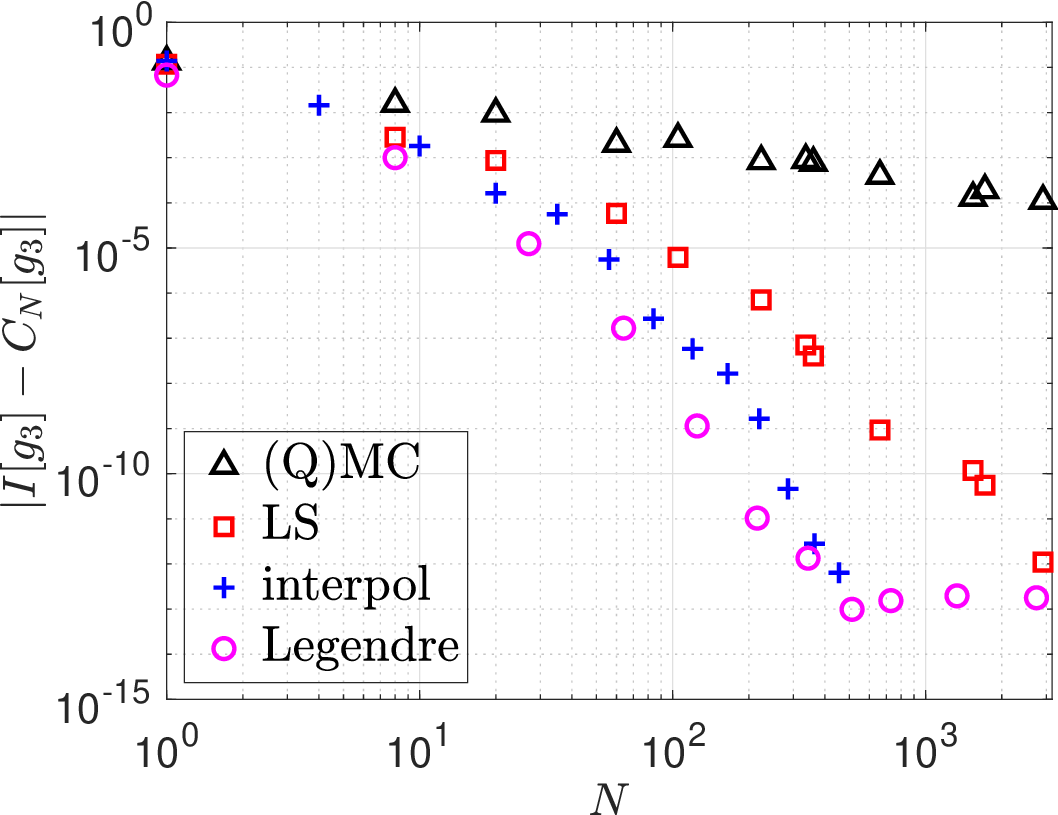} 
    		\caption{$g_3$, Halton}
    		\label{fig:accuracy_Genz3_dim3_1_Halton_Steinitz}
  	\end{subfigure}%
	\\ 
	\begin{subfigure}[b]{0.32\textwidth}
		\includegraphics[width=\textwidth]{%
      		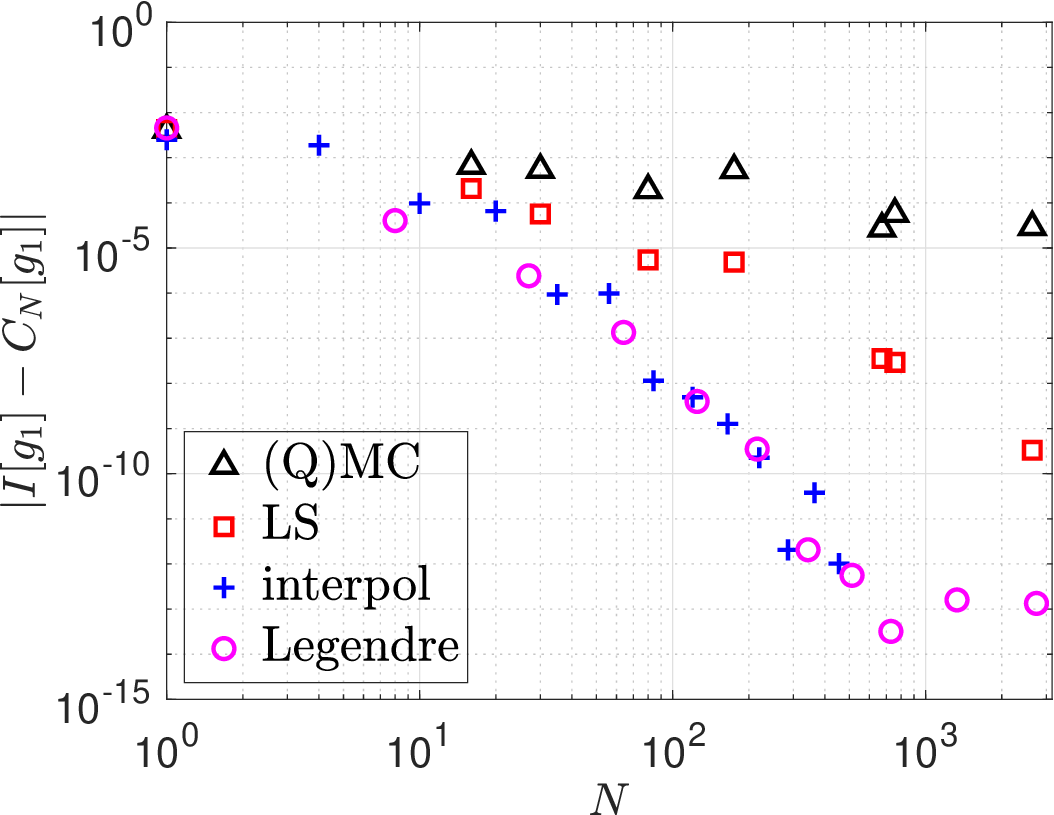} 
    		\caption{$g_1$, random}
    		\label{fig:accuracy_Genz1_dim3_1_random_Steinitz}
  	\end{subfigure}%
	~ 
	\begin{subfigure}[b]{0.32\textwidth}
		\includegraphics[width=\textwidth]{%
      		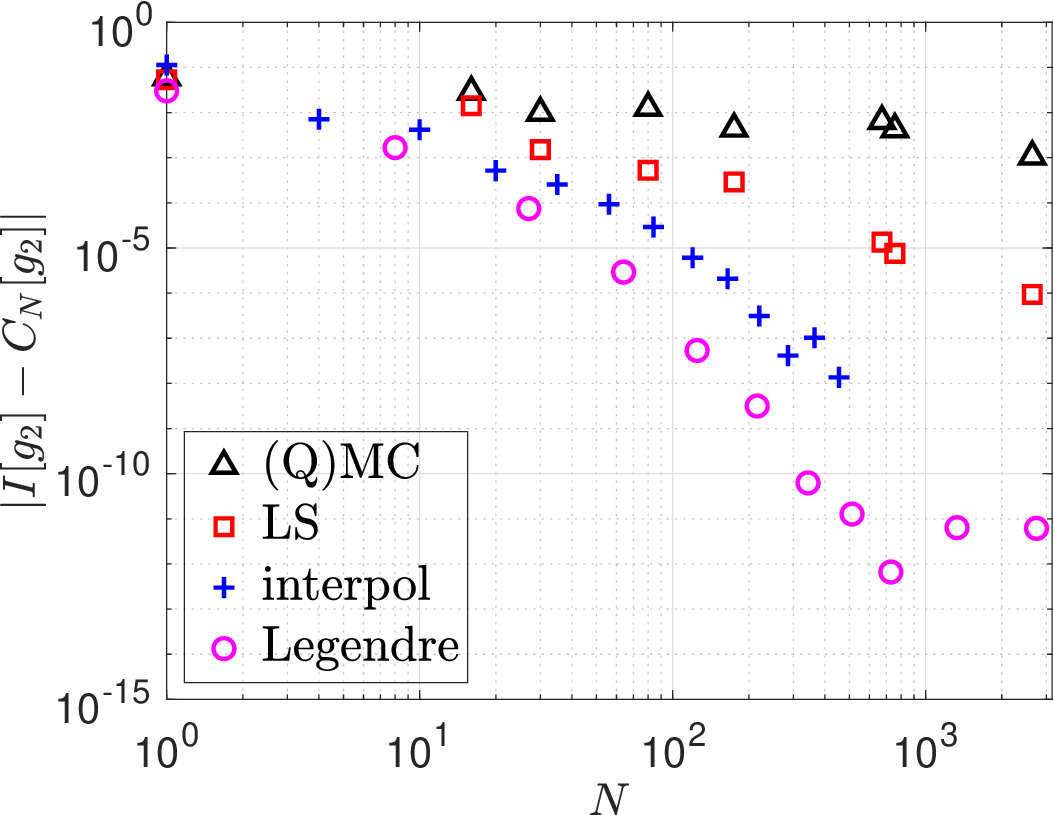} 
    		\caption{$g_2$, random}
    		\label{fig:accuracy_Genz2_dim3_1_random_Steinitz}
  	\end{subfigure}%
	~ 
	\begin{subfigure}[b]{0.32\textwidth}
		\includegraphics[width=\textwidth]{%
      		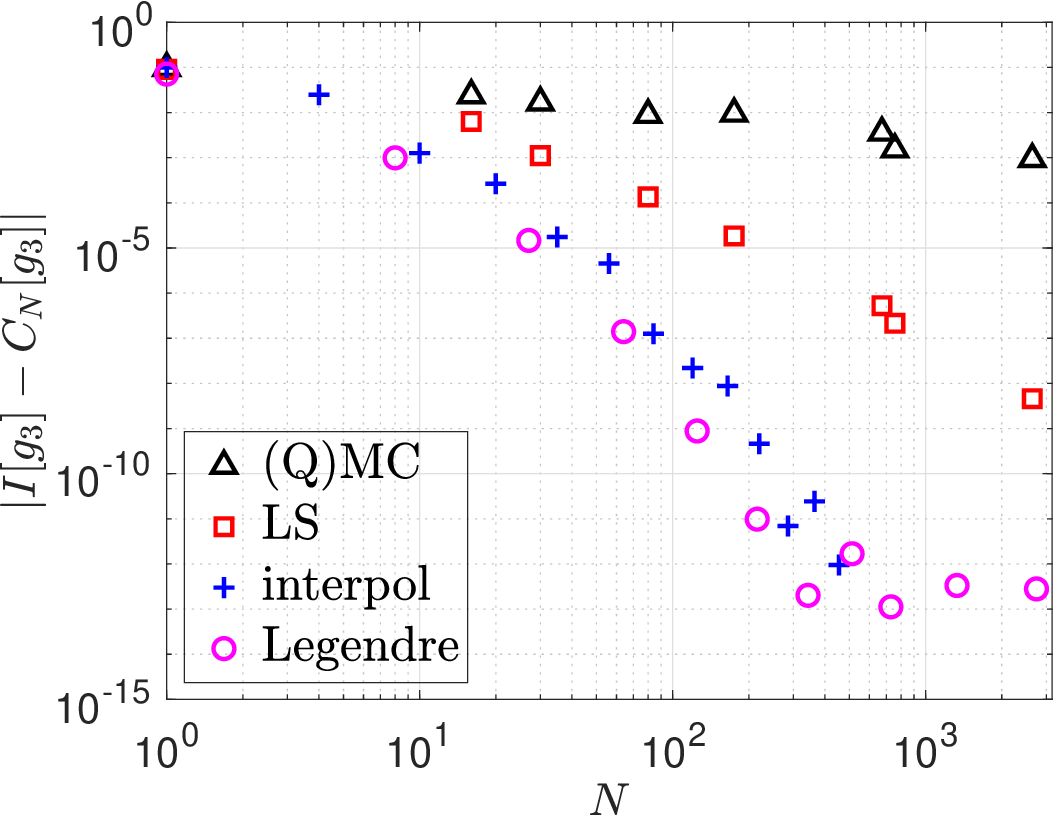} 
    		\caption{$g_3$, random}
    		\label{fig:accuracy_Genz3_dim3_1_random_Steinitz}
  	\end{subfigure}%
  	\caption{
		Errors for Genz' test functions on $\Omega = [0,1]^3$ with $\omega \equiv 1$. 
		The LS- and interpolatory CF are based on multivariate polynomials of increasing total degree ($\mathcal{F}_K(\Omega) = \mathbb{P}_m(\R^3)$). 
	}
  	\label{fig:errors_cube_1_dim3}
\end{figure}

\cref{fig:errors_cube_1_dim2,fig:errors_cube_1_dim3} illustrate the errors of different CFs for this test case in two ($d=2$) and three ($d=3$) dimensions, respectively. 
Compared are the positive high-order LS-CF, the corresponding interpolatory CF, the  (Q)MC method, and a product Legendre method. 
In this context, the term ``(Q)MC" refers to a CF of the form 
\begin{eq} 
	C_N[g] = \frac{| \Omega |}{N} \sum_{n=1}^N g(\mathbf{x}_n) 
	\quad \text{with} \quad 
	g = f \omega.
\end{eq} 
This corresponds to an MC method if the data points are sampled randomly and to a QMC method if the data points are not fully random but correspond to a (semi- or fully-) deterministic low-discrepancy sequence \cite{niederreiter1992random,caflisch1998monte,dick2013high}. 
Further, the LS-CF was constructed to be exact for all two- and three-dimensional polynomials up to total degree $m \in \{0,1,\dots,20\}$ and $m \in \{0,1,\dots,12\}$, respectively. 
For this simple test problem, it can be observed from \cref{fig:errors_cube_1_dim2,fig:errors_cube_1_dim3} that the product Legendre rule yields the most accurate results in most cases, followed by the interpolatory and underlying positive high-order LS-CF.

\subsection{A Nonconstant Weight Function}

We next consider a test cases with the nonconstant weight function $\omega(x_1,x_2) = (1-x_1^2)^{1/2} (1-x_2^2)^{1/2}$ on the hyper-cube $C_2 = [-1,1]^2$ in combination with the test function $f(x_1,x_2) = \arccos(x_1) \arccos(x_2)$. 
The results can be found in \cref{fig:accuracy_misc}. 
The LS-CFs were again constructed to be exact for algebraic polynomials up to an increasing total degree $m \in \{0,1,\dots,20\}$. 
Moreover, the LS-CF and corresponding interpolatory CF are again compared to a (Q)MC methods applied to the same data points as the LS-CF and a (transformed) product Legendre rule applied to $f \omega$ as an integrand. 
We can note from \cref{fig:accuracy_misc} that the positive interpolatory CF---and for Halton points even the LS-CF---yield more accurate results than the product Legendre rule this time.  

\begin{figure}[tb]
	\centering
  	\begin{subfigure}[b]{0.45\textwidth}
		\includegraphics[width=\textwidth]{%
      		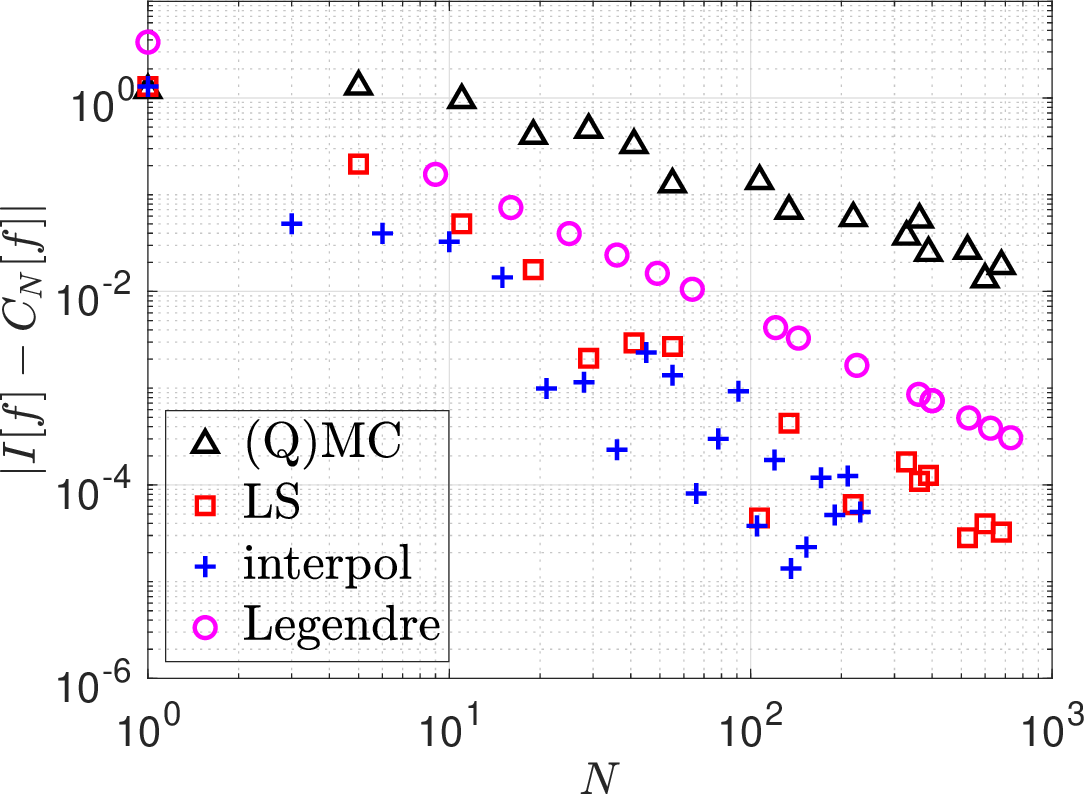} 
    		\caption{Halton points}
    		\label{fig:accuracy_misc1_dim2_cube_C2k_Halton_Steinitz}
  	\end{subfigure}%
	~ 
	\begin{subfigure}[b]{0.45\textwidth}
		\includegraphics[width=\textwidth]{%
      		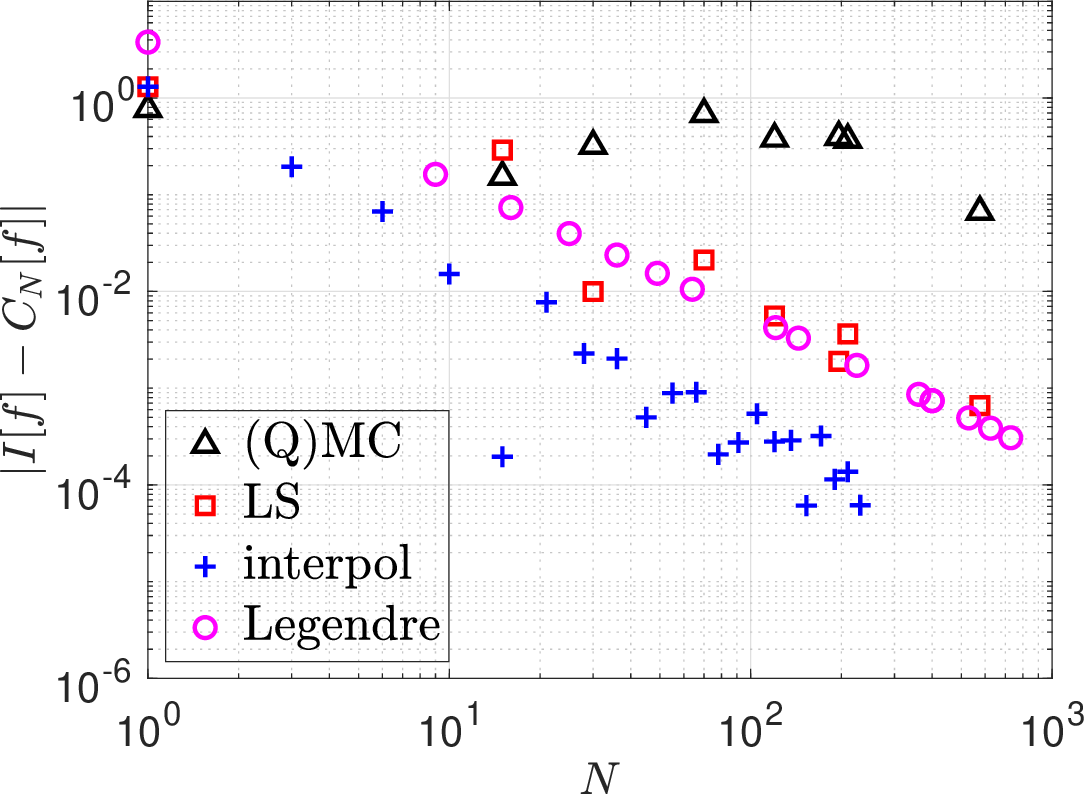} 
    		\caption{Random points}
    		\label{fig:accuracy_misc1_dim2_cube_C2k_random_Steinitz}
  	\end{subfigure}%
  	\caption{
		Errors for the polynomial-based LS-CF, the corresponding interpolatory CFs, the (Q)MC method using the same data points as the LS-CF, and a (transformed) product Legendre rule  
	}
  	\label{fig:accuracy_misc}
\end{figure}

\subsection{A Nonstandard Domain} 

As mentioned before, the proposed LS-CF has the advantage of being easily applicable to nonstandard domains, which is demonstrated now. 
Consider the two-dimensional domain $\Omega$ that is illustrated in \cref{fig:nonstandard_illustration_Halton,fig:nonstandard_illustration_random} with weight function $\omega \equiv 1$. 
\cref{fig:nonstandard_errors_Halton,fig:nonstandard_errors_random} report on the errors of the LS-CF and the corresponding interpolatory CF compared to a (Q)MC method. 
These are again constructed to be exact for algebraic polynomials up to an increasing total degree $m \in \{0,1,\dots,14\}$, and the test function is ${f(x_1,x_2) = \exp( x_1^2 + x_2^2 )}$. 
It should be noted that even in this case, of a nonstandard domain, the positive high-order LS-CF and the corresponding interpolatory CF display encouraging rates of convergence. 
At the same time, we were unable to use a simple product Legendre rule for this problem.  

\begin{figure}[tb]
	\centering
  	\begin{subfigure}[b]{0.45\textwidth}
		\includegraphics[width=\textwidth]{%
      		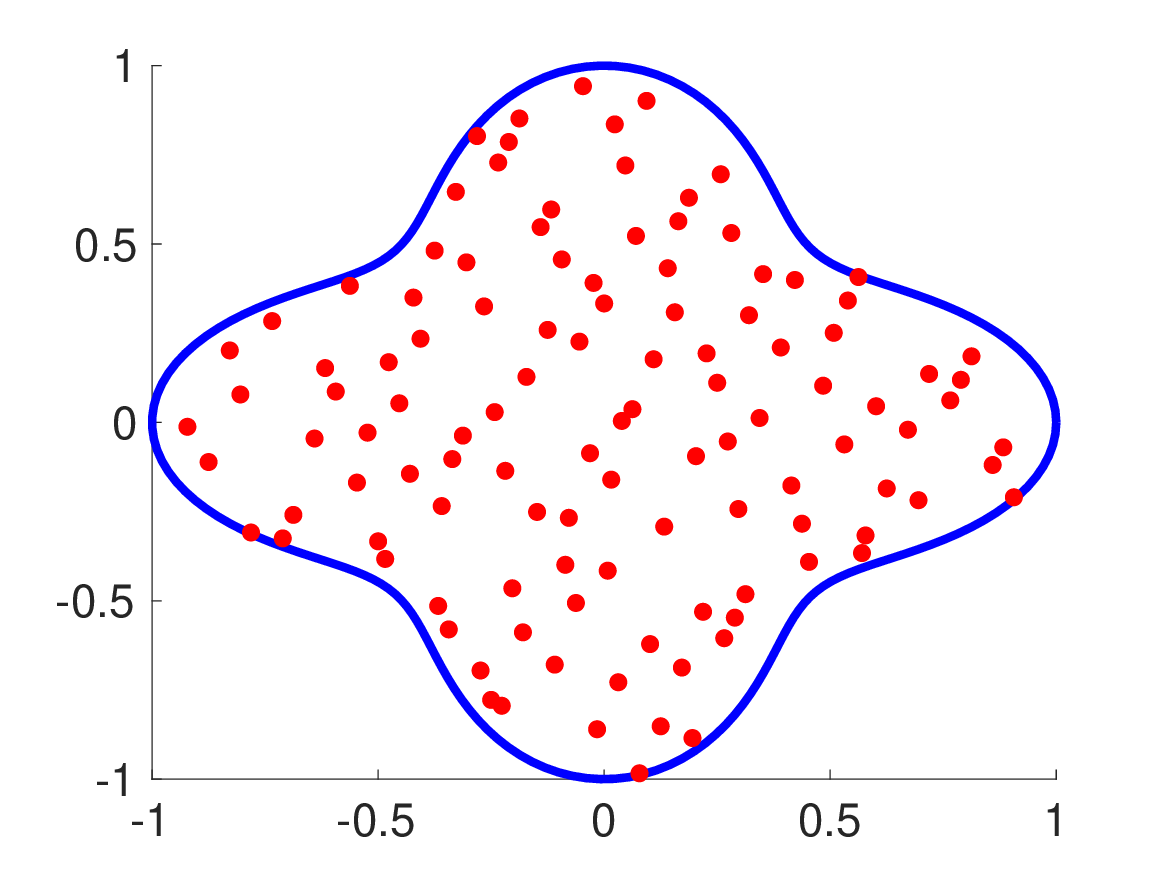} 
    		\caption{Domain with Halton points}
    		\label{fig:nonstandard_illustration_Halton}
  	\end{subfigure}%
	~
	\begin{subfigure}[b]{0.45\textwidth}
		\includegraphics[width=\textwidth]{%
      		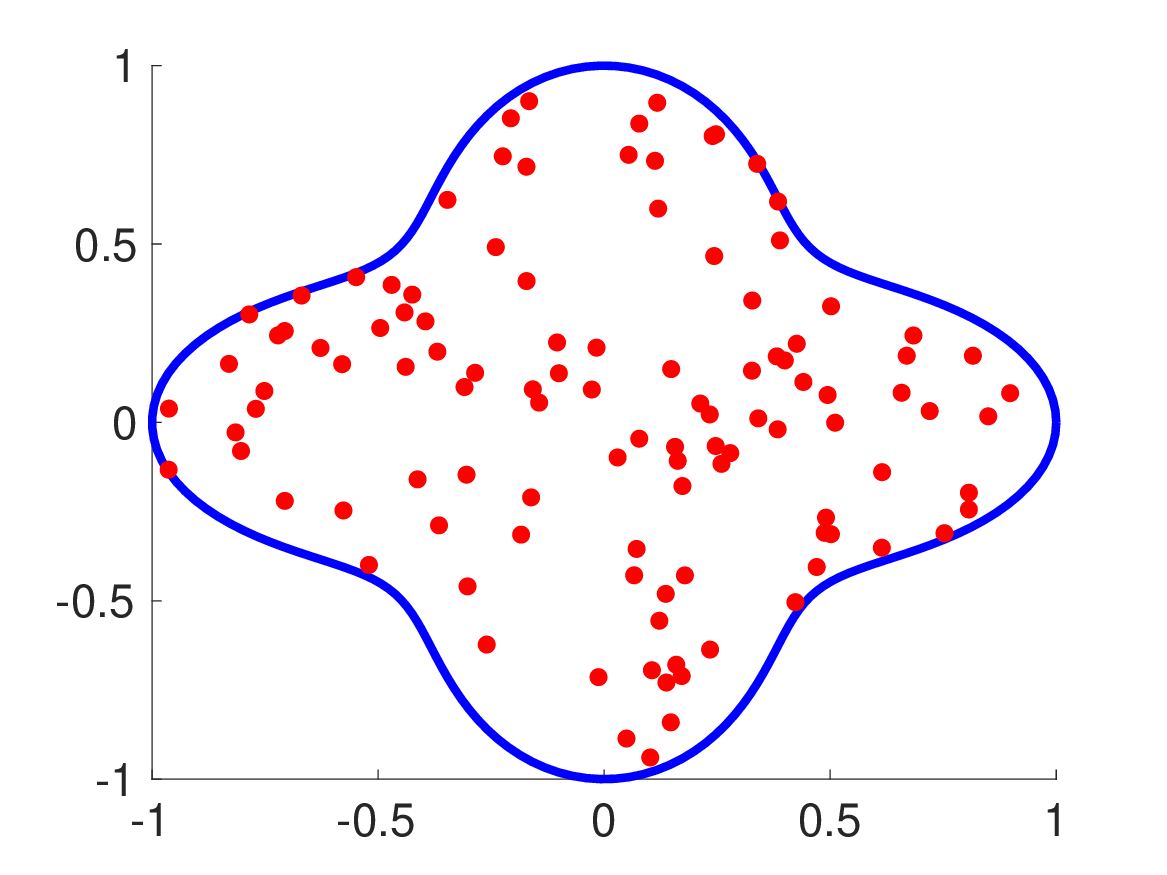} 
    		\caption{Domain with random points}
    		\label{fig:nonstandard_illustration_random}
  	\end{subfigure}%
	\\ 
	\begin{subfigure}[b]{0.45\textwidth}
		\includegraphics[width=\textwidth]{%
      		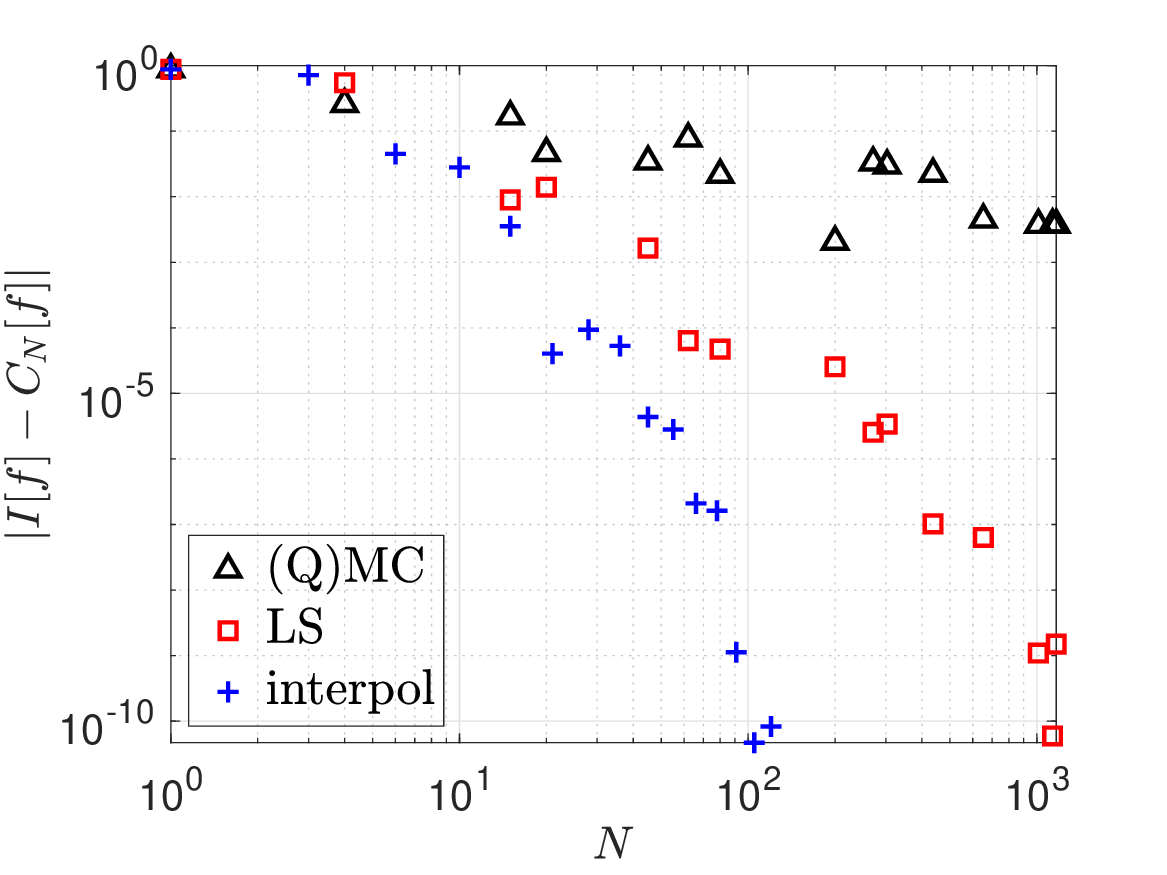} 
    		\caption{Errors for Halton points}
    		\label{fig:nonstandard_errors_Halton}
  	\end{subfigure}%
	~
	\begin{subfigure}[b]{0.45\textwidth}
		\includegraphics[width=\textwidth]{%
      		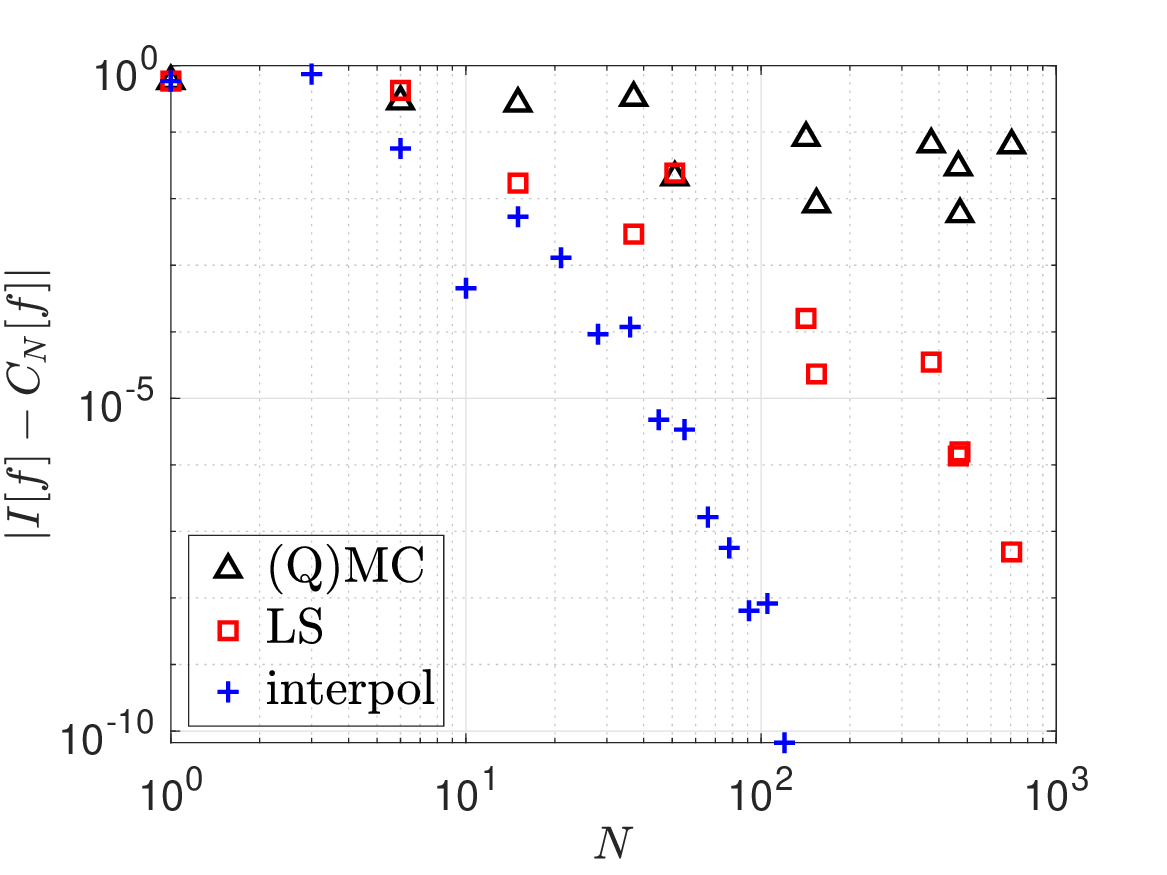} 
    		\caption{Errors for random points}
    		\label{fig:nonstandard_errors_random}
  	\end{subfigure}%
  	\caption{
		A two-dimensional nonstandard domain with Halton and random points and the corresponding errors for $\omega \equiv 1$ and $f(x_1,x_2) = \exp(x_1^2 + x_2^2)$ 
	}
  	\label{fig:nonstandard}
\end{figure}

\subsection{A Simple Periodic Problem} 

Beside being easily applicable to different domains and weight functions, the LS-CFs can also be constructed to be exact for different function spaces. 
The specific function space can be chosen based on some prior information about the function $f$ that is integrated. 
To start with a simple example, consider the periodic function $f(x) = \cos(\pi x) e^{\sin(\pi x)}$ on $\Omega = [-1,1]$ with weight function $\omega \equiv 1$. 
In this case, it is of advantage to construct the LS-CF to be exact for trigonometric functions rather than polynomials. 
The corresponding errors of the LS formulas on equidistant and random points that are exact for polynomials (``LS poly") and trigonometric functions (``LS trig") of increasing degree $m \in \{0,1,\dots,30\}$ are illustrated in \cref{fig:periodic}. 
We also chose this simple problem because it allows us to compare the LS formulas to the (composite) trapezoidal rule when equidistant points are used. 
In fact, we can observe in \cref{fig:periodic_equid} that the LS formula which is exact for trigonometric functions coincides with the trapezoidal rule in many cases. 
In particular, both of them yield highly accurate results for this periodic test problem. 
Furthermore, the LS formulas based on trigonometric functions is also able to yield more accurate results than the LS formula based on polynomials for random data points (see \cref{fig:periodic_random}). 

\begin{figure}[tb]
	\centering
  	\begin{subfigure}[b]{0.45\textwidth}
		\includegraphics[width=\textwidth]{%
      		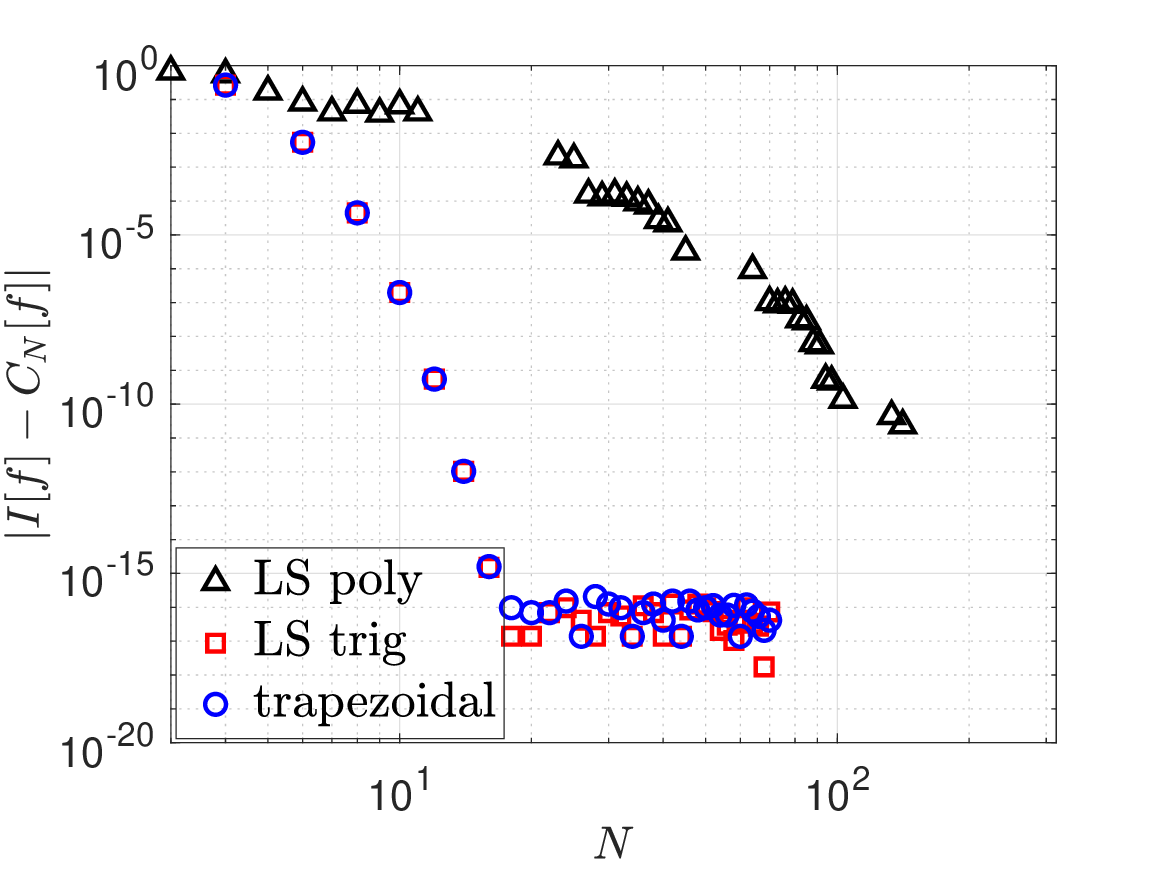} 
    		\caption{Equidistant points}
    		\label{fig:periodic_equid}
  	\end{subfigure}%
	~
	\begin{subfigure}[b]{0.45\textwidth}
		\includegraphics[width=\textwidth]{%
      		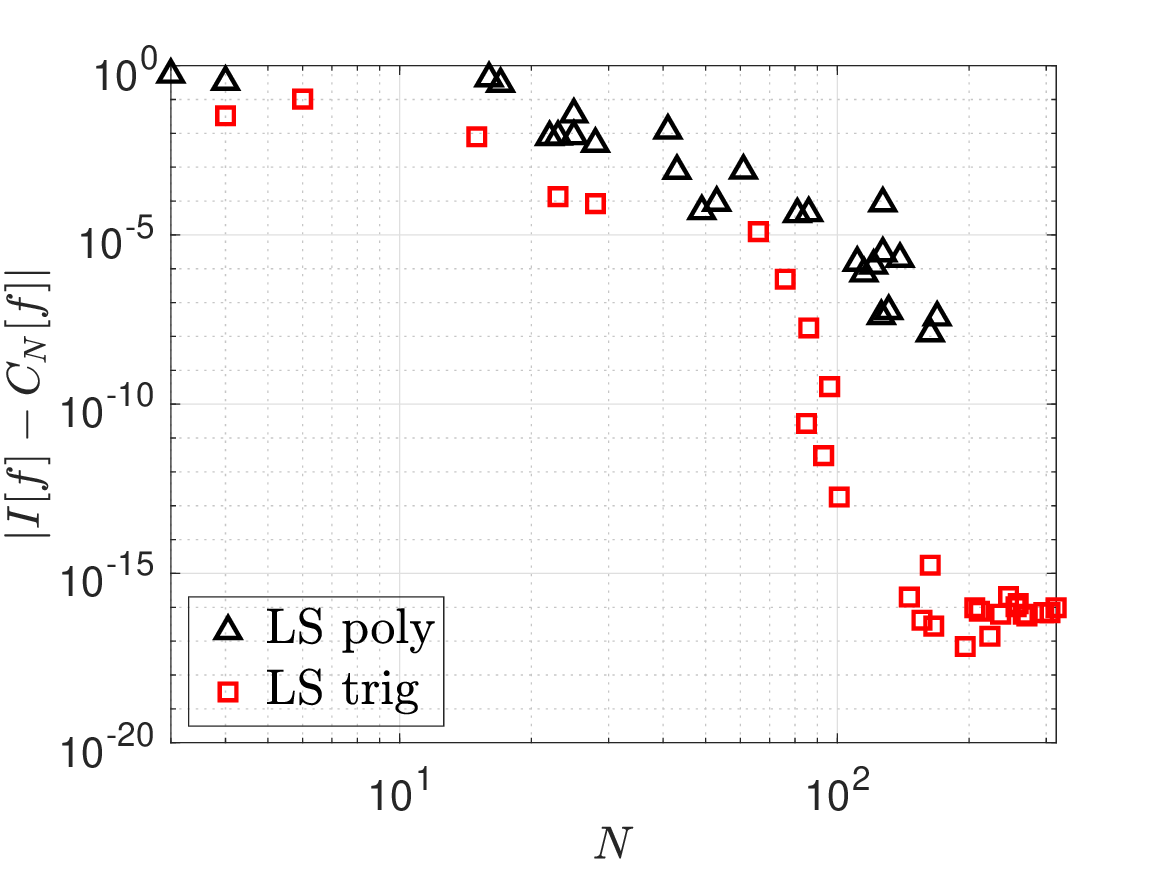} 
    		\caption{Random points}
    		\label{fig:periodic_random}
  	\end{subfigure}%
  	\caption{
		Errors for $f(x) = \cos(\pi x) e^{\sin(\pi x)}$ on $\Omega = [-1,1]$ with $\omega \equiv 1$ using stable high-order LS-CFs based on algebraic and trigonometric function spaces as well as the trapezoidal rule. 
	}
  	\label{fig:periodic}
\end{figure}

\subsection{Radial Basis Functions}

We next consider functions spaces based on positive definite radial basis functions (RBFs) \cite{fornberg2015primer}. 
CFs based on RBFs are a popular tool for scattered data \cite{sommariva2006numerical,fuselier2014kernel,reeger2018numerical,sommariva2021rbf}. 
This is because RBF interpolants can be ensured to uniquely exist for arbitrary point distributions, which is not the case for many other functions (Haar spaces) due to the Mairhuber--Curtis theorem \cite{mairhuber1956haar,curtis1959n}.\footnote{The Mairhuber--Curtis theorem tells us that if we want to have a well-posed multivariate scattered data interpolation problem, then the function space needs to depend on the data locations.}
The idea behind RBF-CFs is to form an RBF interpolant and to exactly integrate it. 
However, RBF formulas are not always ensured to be positive \cite{glaubitz2021towards2}, which can deteriorate their performance for noisy measurements. 
Here, we demonstrate that the LS approach can be used to stabilize RBF-CFs. 
This can be explained by an LS-CF replacing the (potentially unstable) RBF interpolant by a more stable LS approximation from the same RBF function space. 
\cref{fig:RBF} illustrates this for the RBF function space spanned by a constant and the functions $\varphi(\|\boldsymbol{x} - \mathbf{x}_k\|)$ on $\Omega = [0,1]^2$ using a Gaussian kernel $\varphi(r) = e^{(\varepsilon r)^2}$ with shape parameter $\varepsilon = 0.75$. 
Further the number of centers was $K \in \{1,2,\dots,60\}$. 
\cref{fig:RBF_min_weights} provides the values of the smallest weight for the CF based on exact integration of RBF interpolants (``RBF") and the stable LS-CF that is exact for the same RBF function space (``LS-RBF"). 
Note that the RBF formula is found to feature negative weights, which can result in stability issues, while the LS-RBF formula is having positive weights in all cases. 
Further, \cref{fig:RBF_Genz1_nonoise} reports on the errors of the RBF and LS-RBF formulas on $N$ random points applied to Genz' first test function $g_1$ on $\Omega = [0,1]^2$ with $\omega \equiv 1$. 
In this example, both formulas perform similarly. 
In \cref{,fig:RBF_Genz1_noise4,fig:RBF_Genz1_noise2} we repeat this experiment but add uniformly distributed noise of magnitude $10^{-4}$ and $10^{-2}$ to the function values at the data points. 
Observe that the accuracy of the RBF formulas deteriorates notably stronger than of the LS-RBF formula in the presence of noise, due to the improved stability of the latter. 
We made the same observation also for other point distributions and Genz test functions. 

\begin{figure}[tb]
	\centering
  	\begin{subfigure}[b]{0.45\textwidth}
		\includegraphics[width=\textwidth]{%
      		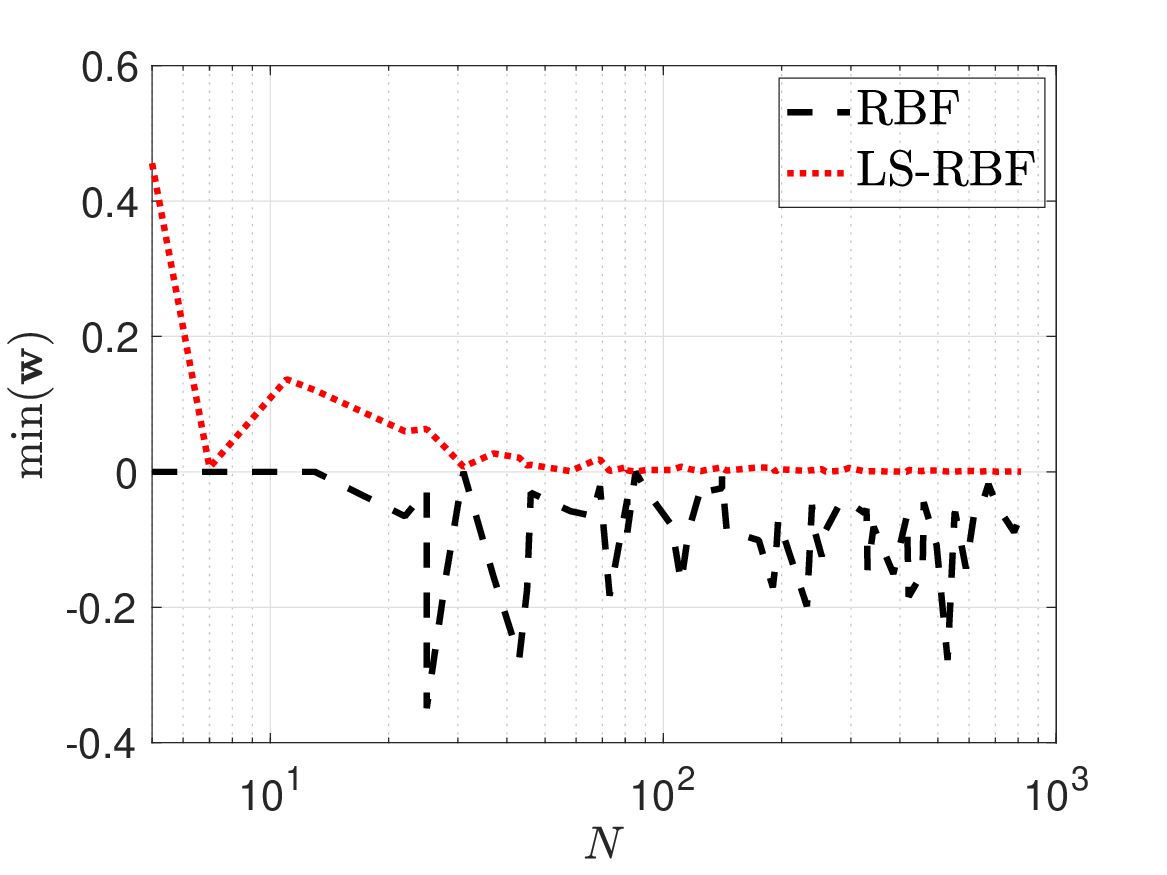} 
    		\caption{Smallest weights}
    		\label{fig:RBF_min_weights}
  	\end{subfigure}%
	~
	\begin{subfigure}[b]{0.45\textwidth}
		\includegraphics[width=\textwidth]{%
      		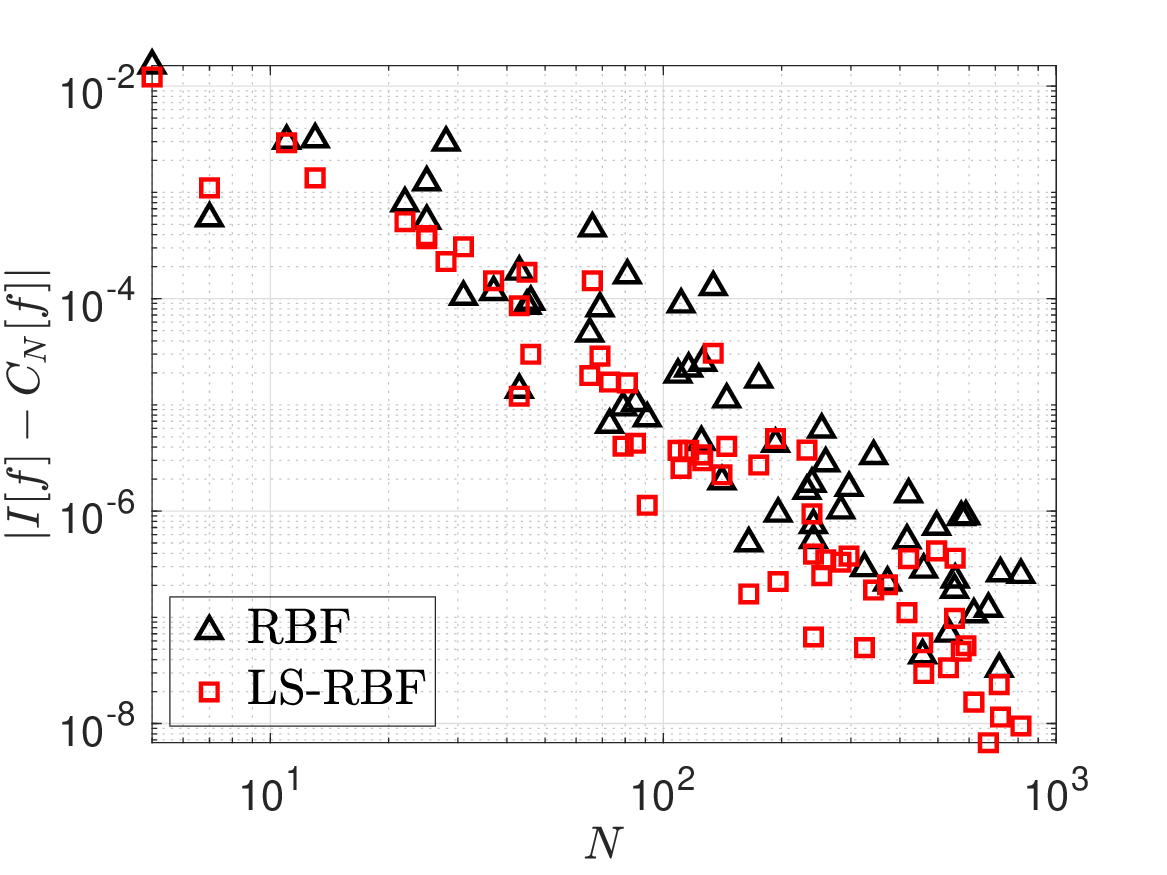} 
    		\caption{Errors, no noise}
    		\label{fig:RBF_Genz1_nonoise}
  	\end{subfigure}%
	\\ 
	\begin{subfigure}[b]{0.45\textwidth}
		\includegraphics[width=\textwidth]{%
      		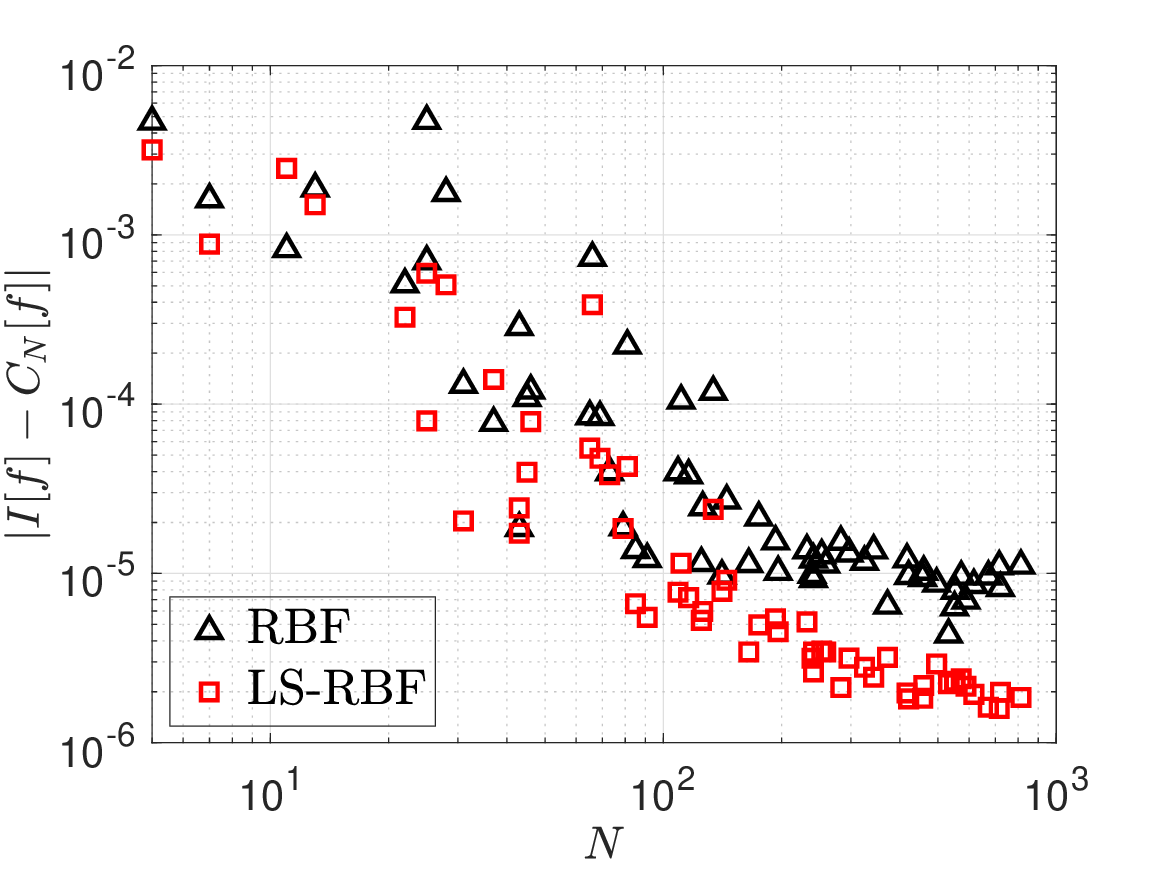} 
    		\caption{Errors, noise of magnitude $10^{-4}$}
    		\label{fig:RBF_Genz1_noise4}
  	\end{subfigure}%
	~
	\begin{subfigure}[b]{0.45\textwidth}
		\includegraphics[width=\textwidth]{%
      		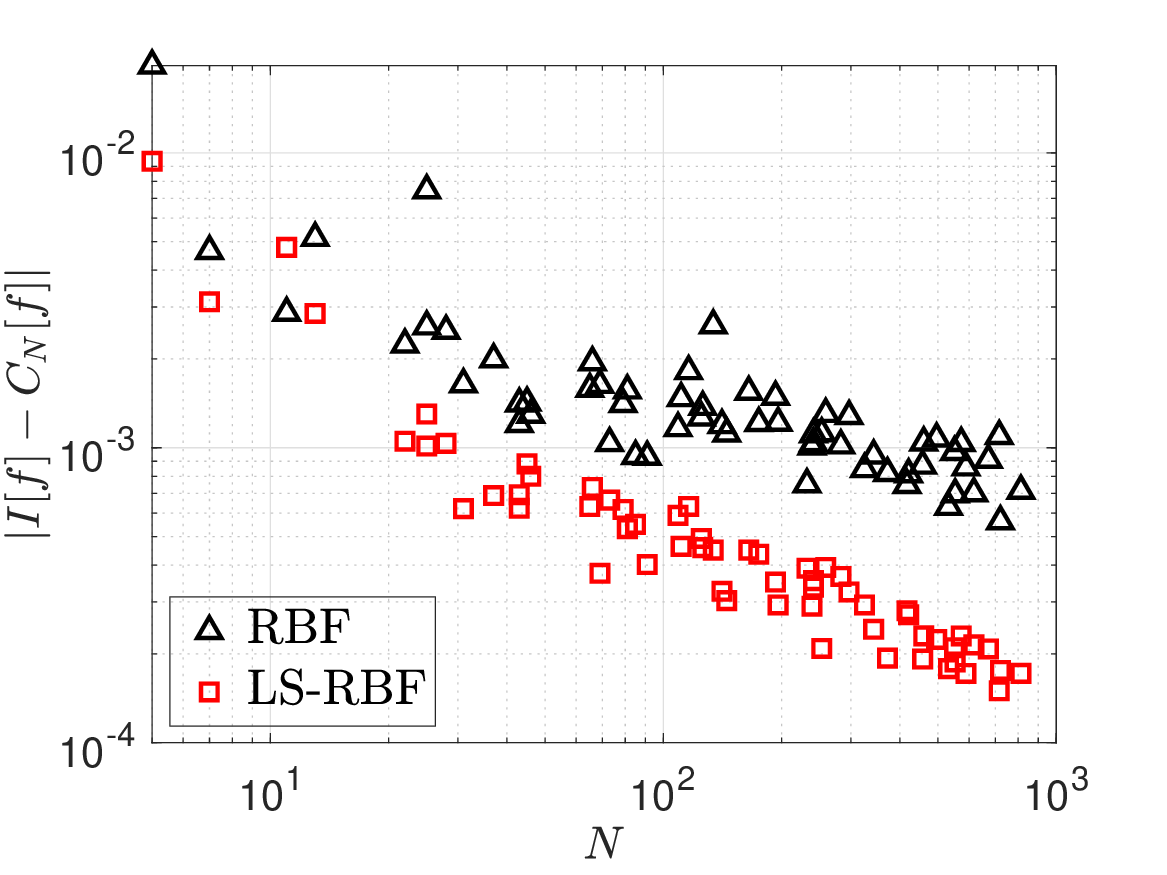} 
    		\caption{Errors, noise of magnitude $10^{-2}$}
    		\label{fig:RBF_Genz1_noise2}
  	\end{subfigure}%
  	\caption{
		Smallest weights and errors for Genz' first test function $g_1$ on $\Omega = [0,1]^2$ with $\omega \equiv 1$ using an RBF and LS-RBF formula. 
		In all cases random points and a Gaussian kernel $\varphi(r) = e^{(\varepsilon r)^2}$ with shape parameter $\varepsilon = 0.75$ was used.  
	}
  	\label{fig:RBF}
\end{figure}

\subsection{Summation-By-Parts Operators for General Function Spaces} 

Another useful application of the positive and exact LS-CFs presented here is the construction of summation-by-parts (SBP) operators for numerical differentiation, which are able to mimic integration-by-parts on a discrete level---thus the name.  
For this reason they are popular building blocks for systematically developing stable and high-order accurate numerical methods for time-dependent differential equations, see \cite{svard2014review,fernandez2014review}. 
SBP operators have been developed based on the idea that the solution is assumed to be well approximated by polynomials up to a certain degree, and the SBP operator should therefore be exact for them. 
However, polynomials might not provide the best approximation for some problems, and other function spaces should be considered. 
To illustrate this, consider the boundary layer solution in \cref{fig:expl_layer_approx}, which demonstrates the advantage of using an exponential instead of a polynomial approximation space. 

\begin{figure}[tb]
	\centering 
	\begin{subfigure}[b]{0.45\textwidth}
		\includegraphics[width=\textwidth]{%
      		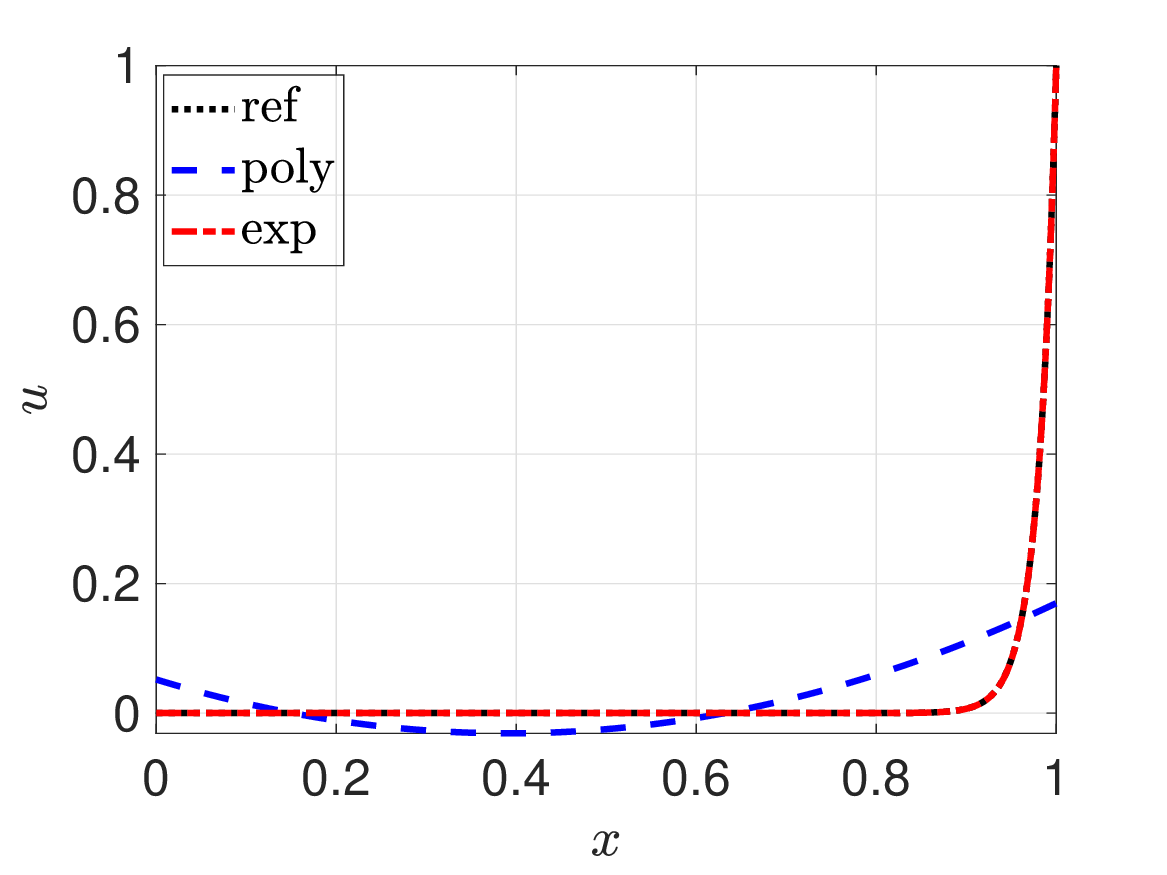} 
    		\caption{Boundary layer solution}
    		\label{fig:expl_layer_approx}
  	\end{subfigure}%
	~
  	\begin{subfigure}[b]{0.45\textwidth}
		\includegraphics[width=\textwidth]{%
      		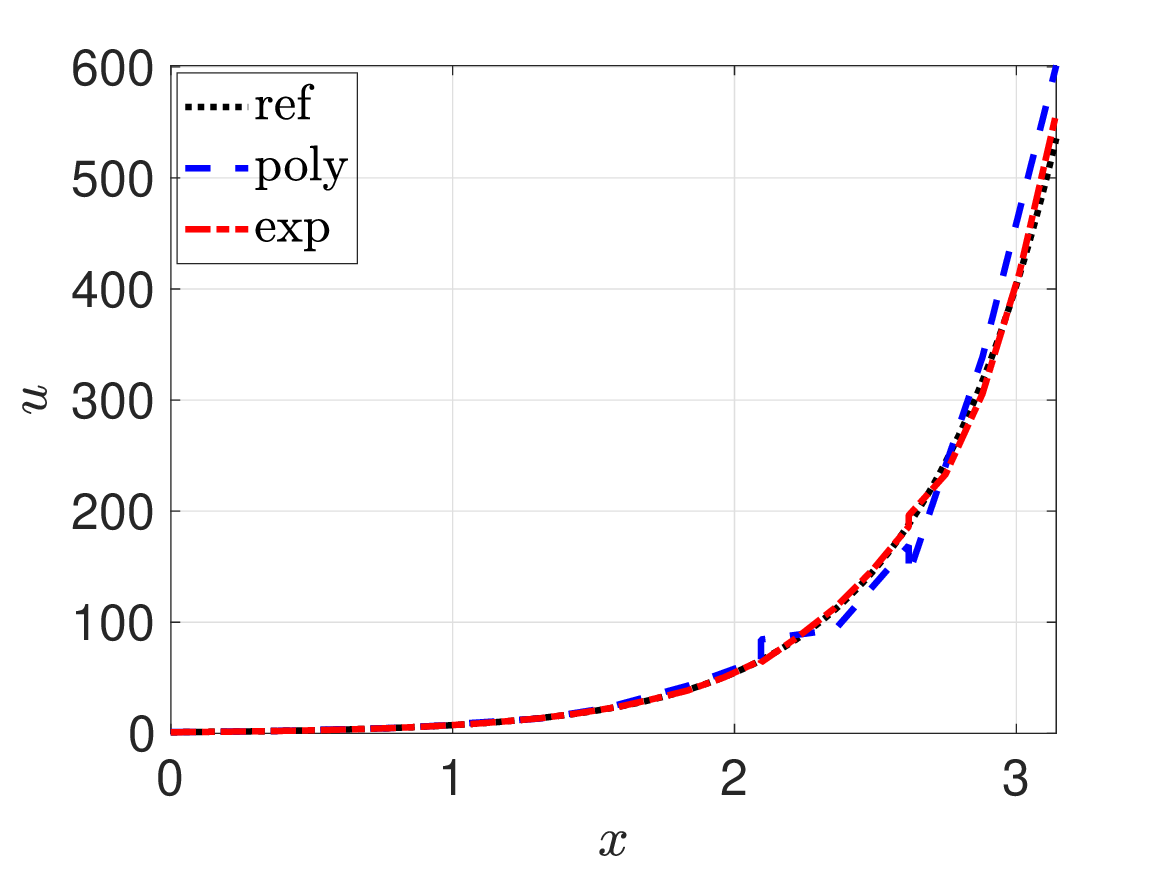} 
    		\caption{Numerical PDE solutions}
    		\label{fig:linear_source_SAT_K3_I6}
  	\end{subfigure}%
  	\caption{
  	Left: Polynomial (``poly") and exponential (``exp") least-squares function approximation from the space $\Span \{ 1,x, x^2 \}$ and $\Span \{ 1, x, e^{\alpha x} \}$ with suitably chosen $\alpha \in \R$, respectively, to a boundary layer function. 
	Right: Numerical solutions for the linear advection problem at $t=3.5$ obtained by a multi-block FSBP-SAT scheme using $6$ blocks with a polynomial and exponential approximation space, $\mathbb{P}_2 = \Span \{ 1,x, x^2 \}$ and $\mathcal{E}_2 = \Span \{ 1, x, e^{x} \}$, respectively. 
  	}
  	\label{fig:expl_approx}
\end{figure} 

Consequently, in a recent work \cite{glaubitz2022SBP}, we developed SBP operators for general functions spaces, referred to as FSBP operators. 
These allowed us to systematically develop stable and high-order accurate numerical methods that are based on general approximation spaces. 
The advantage of using such a method is demonstrated in \cref{fig:linear_source_SAT_K3_I6} for the inhomogeneous linear advection problem 
\begin{equation}\label{eq:linear_inhom} 
\begin{aligned} 
	\partial_t u + \partial_x u & = 2u, \quad && 0 < x < \pi, \\ 
	u(x,0) & = 1, \quad && 0 \leq x \leq \pi, \\ 
	u(0,t) & = 1, \quad && t \geq 0,
\end{aligned} 
\end{equation} 
with exact steady state solution $u(x) = e^{2x}$. 
The steady state solution can be expected to be better approximated using an exponential rather than a polynomial approximation space. 
However, the construction of an FSBP operator that is based on a function space $\mathcal{F}$ requires that we have a positive and $(\mathcal{F} \mathcal{F})'$-exact quadrature, where 
\begin{equation} 
	(\mathcal{F} \mathcal{F})' 
		= \left\{ \, (f g)' \mid f,g \in \mathcal{F} \, \right\}.
\end{equation} 
Such quadrature formulas can be constructed using the generalized LS approach introduced in the present work. 
As an example, consider the following exponential function space on $\Omega = [0,1]$: 
\begin{equation} 
	\mathcal{E}_2 = \Span \{ 1, x, e^{x} \} 
	\quad \text{with} \quad 
	(\mathcal{E}_2\mathcal{E}_2)' = \Span\{ 1, x, e^x, x e^x, e^{2x} \}
\end{equation} 
Using $N=5$ equidistant grid points, we found the LS formula with the following points and weights to be positive and $(\mathcal{E}_2\mathcal{E}_2)'$-exact: 
\begin{equation} 
\begin{aligned}
	\mathbf{x} = \left[0, \frac{1}{4}, \frac{1}{2}, \frac{3}{4}, 1 \right]^T, \quad 
	\mathbf{w} = \left[ \frac{2}{25}, \frac{9}{25}, \frac{3}{25}, \frac{9}{25}, \frac{2}{25} \right]^T,
\end{aligned}
\end{equation} 
where we have rounded the numbers to the second decimal place. 
The corresponding FSBP operator $D$ that was used to generate the numerical results in \cref{fig:linear_source_SAT_K3_I6} is 
\begin{equation} 
\renewcommand*{\arraystretch}{1.3} 
	D \approx
	\begin{bmatrix} 
		-\frac{329}{50} & \frac{859}{100} & -\frac{23}{50} & -\frac{127}{50} & \frac{49}{50}\\ -\frac{9}{5} & 0 & \frac{22}{25} & \frac{29}{20} & -\frac{53}{100}\\ \frac{7}{25} & -\frac{257}{100} & 0 & \frac{129}{50} & -\frac{29}{100}\\ \frac{53}{100} & -\frac{29}{20} & -\frac{89}{100} & 0 & \frac{181}{100}\\ -\frac{49}{50} & \frac{249}{100} & \frac{12}{25} & -\frac{213}{25} & \frac{653}{100} 
	\end{bmatrix}, 
\end{equation}
where we have again rounded the numbers to the second decimal place. 
See \cite{glaubitz2022SBP} for more details. 

The potential advantage of using non-polynomial function spaces to solve time-depenedent differential equations has also been pointed out in several other works. 
For instance, in \cite{kadalbajoo2003exponentially,kalashnikova2009discontinuous}, exponentially fitted schemes were used to solve singular perturbation problems. 
Discontinuous Galerkin methods based on non-polynomial approximation spaces were considered in \cite{yuan2006discontinuous}. 
There are also several works on essentially non-oscillatory (ENO) and weighted ENO (WENO) reconstructions based on non-polynomial function approximations \cite{christofi1996study,iske1996structure,hesthaven2019entropy}. 
In these cases, it would be of potential advantage to use a LS-CF that is exact for the corresponding function space.

\subsection{A Ten-Dimensional Example} 

For high-dimensional problems it becomes impractical to construct LS-CF that are exact for polynomials of an increasing (total) degree $m$. 
This is because the dimension of the corresponding function space $\mathbb{P}_m(\R^d)$ is given by $\dim \mathbb{P}_m(\R^d) = \binom{m+d}{m}$, which implies $\dim \mathbb{P}_m(\R^d) = \mathcal{O}(m^d)$. 
Thus, in a sense, the LS-CFs based on polynomials fall victim to the curse of dimensionality and cannot be applied directly in high dimensions.
However, we now demonstrate how appropriately chosen nonpolynomial function spaces can be used to extend the positive LS-CFs into a high-dimensional setting. 
Another idea, related to variance reduction in MC and QMC methods, will be discussed in \cref{sub:variance_reduction}. 

Instead of a polynomial function space, we consider an RBF function space for the ten-dimensional domain $\Omega = [0,1]^{10}$. 
For sake of simplicity, we use a Gaussian RBF $\varphi(r) = \exp(\varepsilon^2 r^2)$ with shape parameter $\varepsilon = 10^{-2} \sqrt{K}$, where $K$ is the number of centers used to build the RBF space. 
We used ten logarithmically spaced values for $K$ between $1$ and $10^3$. 
The dimension of the RBF space only depends on the number of centers and not on the dimension of the domain $\Omega$. 
We make no claim about the optimality of the above choices of a function space of Gaussian kernels and the chosen shape parameters. 
That said, \cref{fig:10D} demonstrates that even in high dimensions LS-CFs are able to yield more accurate results than the MC and QMC method when these are exact for appropriate RBF spaces. 
Adapting the locations of the centers to some prior knowledge of the function $f$ (e.\,g., $f$ rapidly decays away from the axes)  and using, for instance, sparse grids might further improve the results. 
While this would exceed the scope of this paper, it would also be of interest to combine high-dimensional LS-CF with separable (low-rank) \cite{beylkin2009multivariate,chevreuil2015least} or sparse approximations \cite{chkifa2015breaking,cohen2015approximation,adcock2017compressed}. 

\begin{figure}[tb]
	\centering
  	\begin{subfigure}[b]{0.45\textwidth}
		\includegraphics[width=\textwidth]{%
      		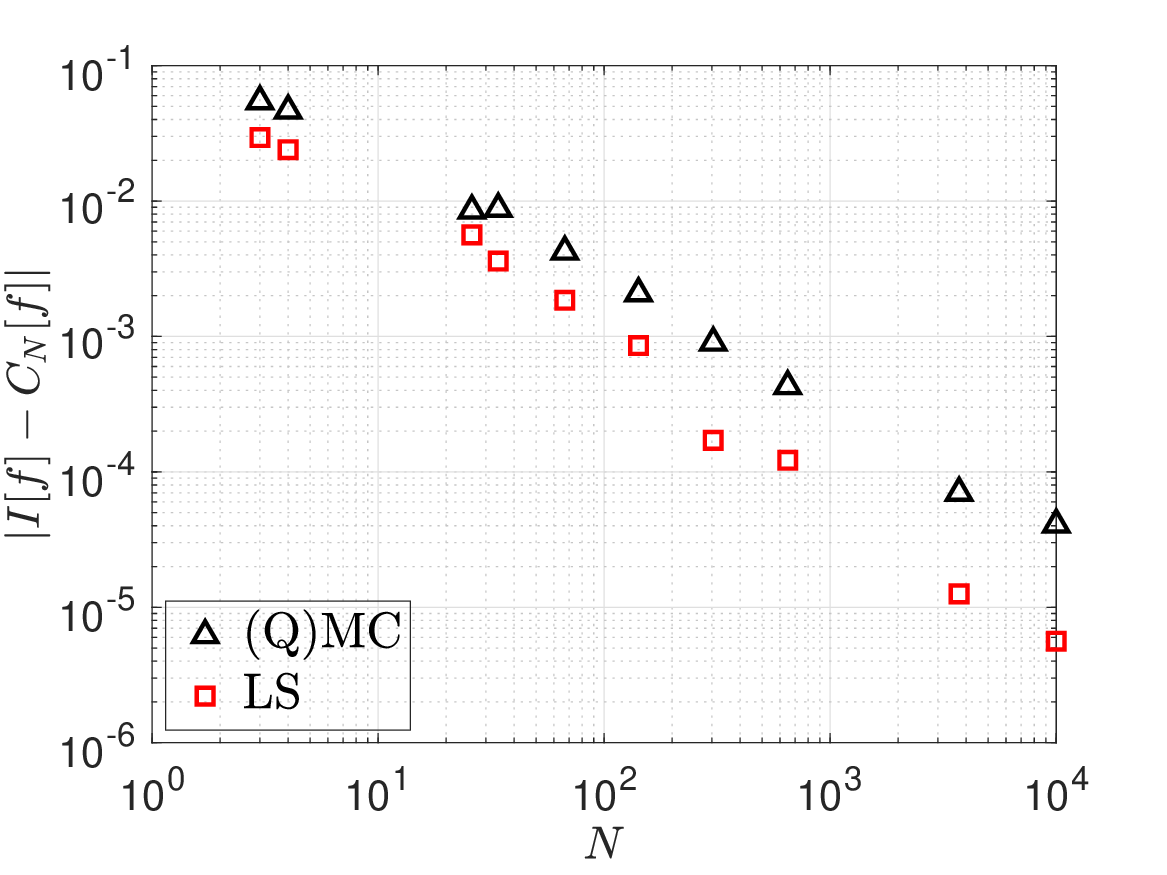} 
    		\caption{Halton points}
    		\label{fig:10D_Halton}
  	\end{subfigure}%
	~
	\begin{subfigure}[b]{0.45\textwidth}
		\includegraphics[width=\textwidth]{%
      		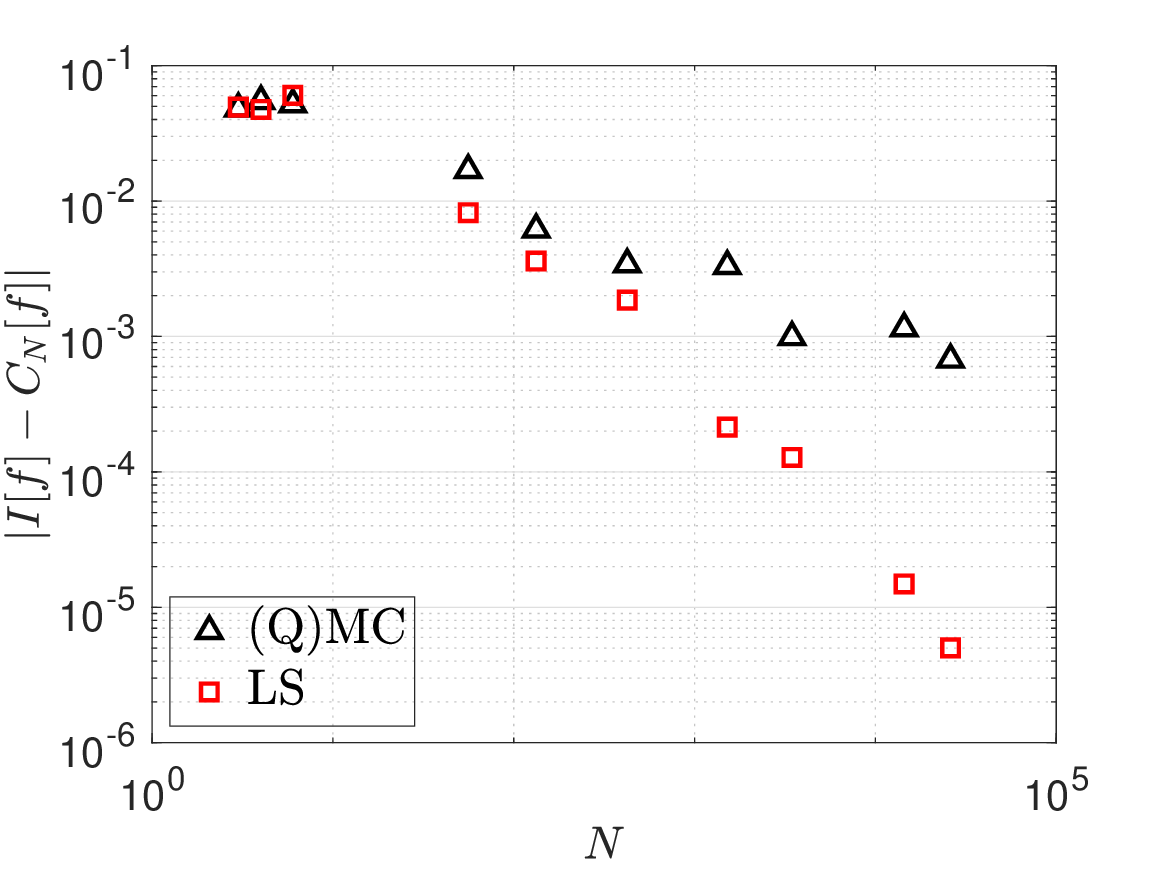} 
    		\caption{Random points}
    		\label{fig:10D_random}
  	\end{subfigure}%
  	\caption{
		Errors for Genz' third test function on $\Omega = [0,1]^{10}$ with $\omega \equiv 1$ using a stable LS-CFs based on Gaussian RBFs and a (Q)MC method on the same points. 
	}
  	\label{fig:10D}
\end{figure}

\subsection{Variance Reduction in MC and QMC Methods} 
\label{sub:variance_reduction}

Another possible applications of LS-CFs to high-dimensional problems is related to the idea of variance reduction in MC and QMC methods \cite{nakatsukasa2018approximate}. 
In the context of the present paper, the idea is to fix the maximum total degree $m$, say $m=1$ or $m=2$, and to only increase the number of data points. 
Henceforth, we refer to such a CF as an LS-QMC formula when random points are used and as an LS-MC formula when a low-discrepancy sequence is used. 
\cref{fig:LSMC} reports on the errors of these formulas for $m=1$ and $m=2$ compared to a usual (Q)MC formula for the ten-dimensional domain $\Omega = [0,1]^{10}$.
We can see that the LS-(Q)MC formulas show the same convergence rate as the (Q)MC formula but with a reduced constant in front of the error. 
A similar observation was made in \cite{nakatsukasa2018approximate}. 
In this context, the present work can be used to ensure positivity of LS-(Q)MC formulas if $N$ is sufficiently larger than the fixed maximum total degree $m$. 
We also refer to \cite[Appendix A]{nakatsukasa2018approximate} for an especially efficient implementation of the LS-(Q)MC formulas. 

\begin{figure}[tb]
	\centering
  	\begin{subfigure}[b]{0.45\textwidth}
		\includegraphics[width=\textwidth]{%
      		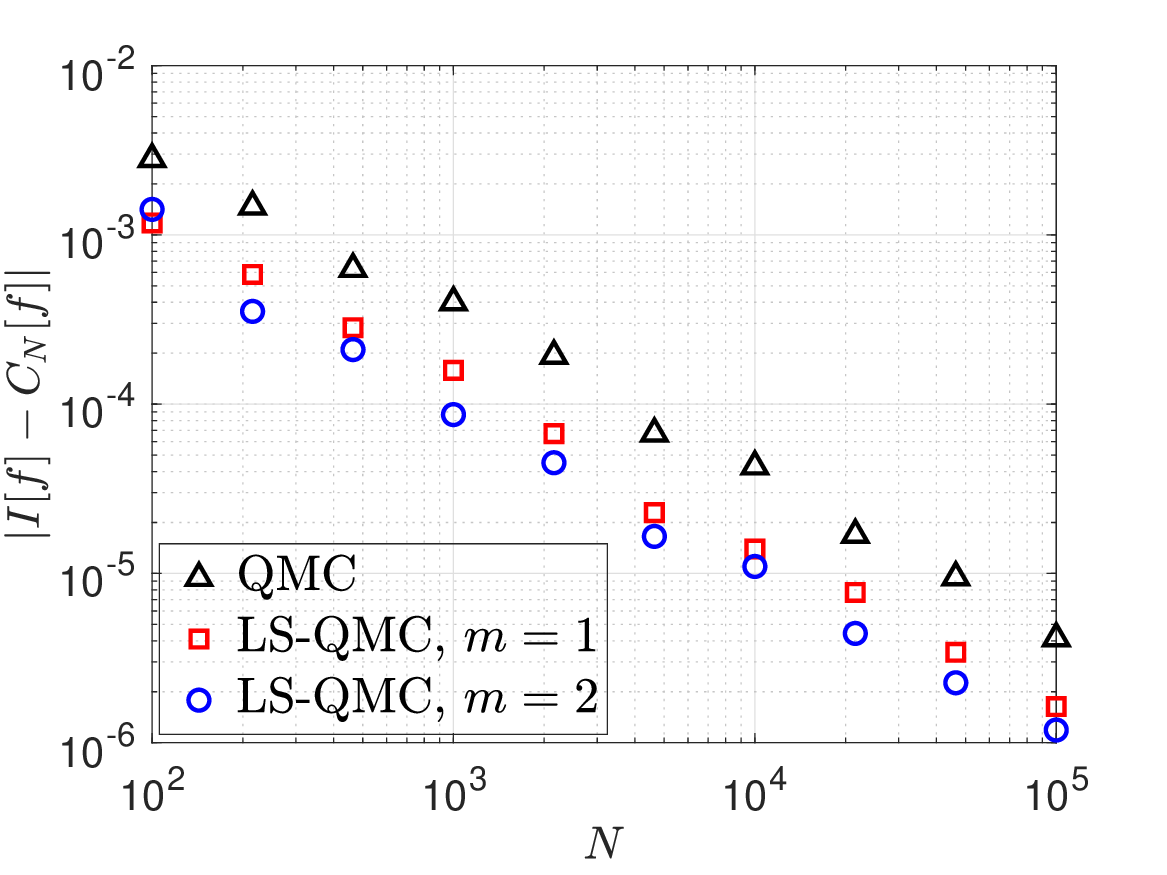} 
    		\caption{Halton points}
    		\label{fig:LSMC_Halton}
  	\end{subfigure}%
	~
	\begin{subfigure}[b]{0.45\textwidth}
		\includegraphics[width=\textwidth]{%
      		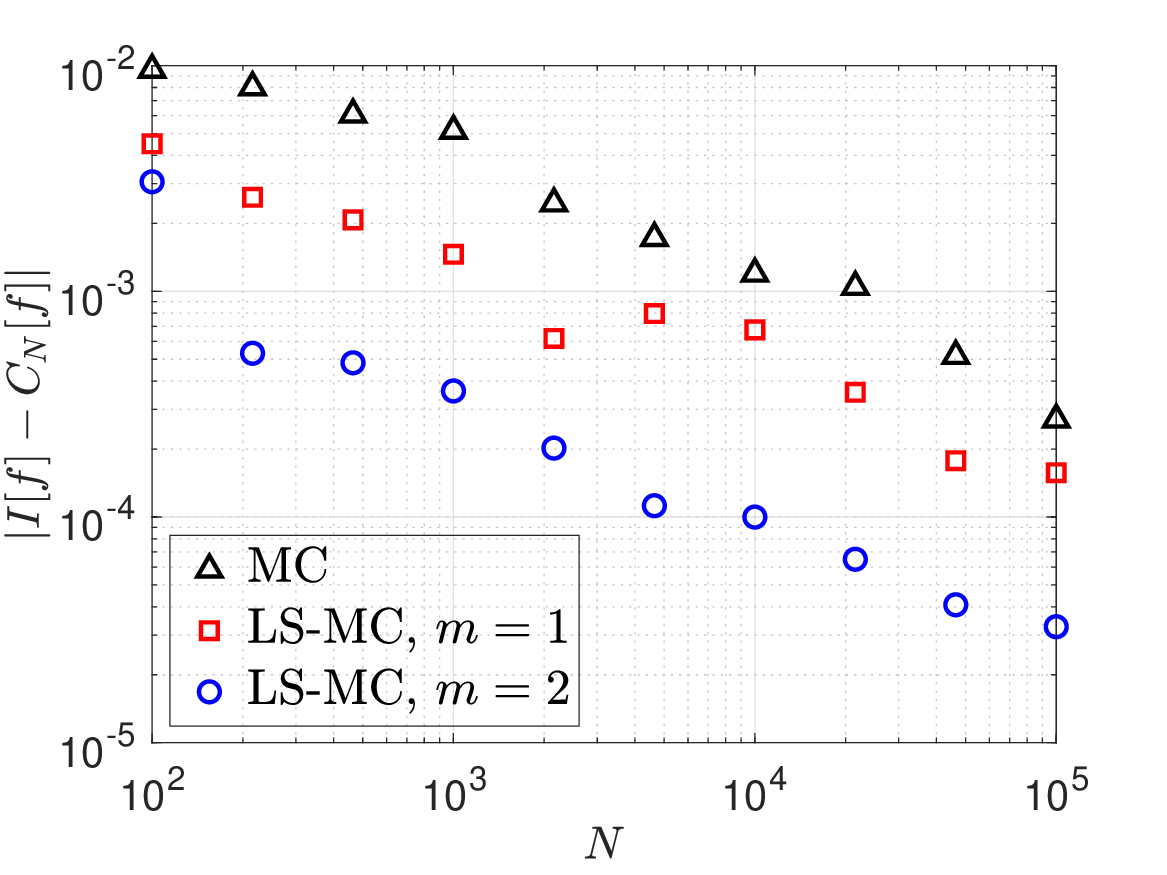} 
    		\caption{Random points}
    		\label{fig:LSMC_random}
  	\end{subfigure}%
  	\caption{
		Errors for Genz' third test function on $\Omega = [0,1]^{10}$ with $\omega \equiv 1$ using a (Q)MC method and the LS-(Q)MC method with $m=1$ and $m=2$ 
	}
  	\label{fig:LSMC}
\end{figure}  
\section{Concluding Thoughts} 
\label{sec:summary} 

We presented a simple procedure to construct provable positive and exact CFs in a general multi-dimensional setting. 
It was proved that under relatively mild restrictions such CFs always exist---and can be determined by the method of LS---when a sufficiently large number of equidistributed data points is used. 
This extends some previous results on LS formulas from one dimension \cite{wilson1970necessary,wilson1970discrete,huybrechs2009stable,glaubitz2020stableQRs} as well as multiple dimensions \cite{glaubitz2020stableCFs} (but restricted to function spaces of algebraic polynomials). 
At the same time, our findings can also be seen as an extension of the stable high-order randomized CFs discussed in \cite{migliorati2020stable}, which are positive and exact with a high probability, into a deterministic framework. 
Furthermore, similarities with certain methods for variance reduction in the context of MC and QMC methods \cite{nakatsukasa2018approximate} should be noted. 
Our results indicate that such high-order corrections are ensured to yield positive formulas if a sufficiently large number of data points is used. 
Finally, it is possible to combine the provable positive and exact LS-CFs with subsampling methods, therefore constructing positive interpolatory CFs.  
While these are already predicted by Tchakaloff's theorem (see \cref{thm:Tchakaloff}), it is often not clear how the data points should be chosen. 
Our findings indicate that some of the positive interpolatory CFs predicted by Tchakaloff are supported on sets of equidistributed points. 
This can be considered as an important---yet sometimes missing---justification and design criterion for CFs constructed based on optimization strategies. 
These include NNLS as well as linear programming approaches.   

In forthcoming works, we will address the application of the LS-CF introduced here to derive SBP operators for general function spaces \cite{glaubitz2022SBP} as well as to construct energy-stable RBF methods, a development that we have already initiated in \cite{glaubitz2021stabilizing,glaubitz2021towards}. 

\appendix 
\section{Steinitz' Method} 
\label{app:Steinitz} 

Given is a $K$-dimensional function space $\mathcal{F}_K(\Omega)$ and a positive and $\mathcal{F}_K(\Omega)$-exact CF,  
\begin{equation}\label{eq:pos-CF}
	C_N[f] = \sum_{n=1}^N w_n f(\mathbf{x}_n).
\end{equation}
Here, $N$ denotes the number of data points in a generic sense.  
If $N> K$, one successively reduces the number of data points until $N \leq K$ by going over to an appropriate subset of data points, while preserving positivity and $\mathcal{F}_K(\Omega)$-exactness. 
Recall that the vector space $\mathcal{F}_K(\Omega)$ has dimension $K$, and so does its algebraic dual space (the space of all linear functionals defined on $\mathcal{F}_K(\Omega)$). 
In particular, among the $N$ linear functionals 
\begin{equation} 
	L_n[f] = f(\mathbf{x}_n), \quad n=1,\dots,N,
\end{equation}
at most $K$ are linearly independent. 
That is, if $N > K$, there exist a vector of coefficients $\mathbf{a} = (a_1, \dots, a_M)^T$ such that 
\begin{equation}\label{eq:vec-a}
	a_1 L_1[f] + \dots + a_N L_N[f] = 0 \quad \forall f \in \mathcal{F}_K(\Omega)
\end{equation}
and $a_n > 0$ for at least one $n$.
Let ${\sigma= \max_{1 \leq n \leq N} a_n/w_n}$. 
Then, $\sigma > 0$, $\sigma w_n - a_n \geq 0$ for all $n$, and $\sigma w_n - a_n = 0$ for at least one $n$. 
From \cref{eq:vec-a} one therefore has 
\begin{equation}
	I[f] = \frac{\sigma w_1 - a_1}{\sigma} L_1[f] + \dots + \frac{\sigma v_N - a_N}{\sigma} L_N[f] 
	\quad \forall f \in \mathcal{F}_K(\Omega).
\end{equation} 
Note that one of the coefficients is zero and, together with the corresponding linear functional (data point), can be removed.
Hence, on $\mathcal{F}_K(\Omega)$, the integral $I$ can be expressed as a linear combination of not more than $N-1$ of the linear functionals $L_1,\dots,L_N$ with positive coefficients. 
Iterating this process, one finally arrives at a positive interpolatory CF with $N \leq K$ while being exact for all $f \in \mathcal{F}_K(\Omega)$. 

\begin{algorithm}[tb]
\caption{The Steinitz Method}
\label{algo:Steinitz}
\begin{algorithmic}[1]
	\While{$K < N$} 
    		\State{Compute $\Phi = \Phi(X)$ and $\text{null}(\Phi)$} 
      	\State{Determine $\mathbf{a} \in \text{null}(\Phi) \setminus \{\mathbf{0}\}$ s.\,t.\ $a_n > 0$ for at least one $n$ (see \cref{rem:a})} 
		\State{Compute $\sigma = \max_{n} a_n/w_n$} 
		\State{Overwrite the cubature weights: $w_n = (\sigma w_n - a_n)/\sigma$} 
		\State{Remove all zero weights as well as the corresponding data points} 
		\State{$N = N -  \#\{ \, w_n \mid w_n = 0, \ n=1,\dots,N \, \}$} 
    \EndWhile 
\end{algorithmic}
\end{algorithm} 

An algorithmic description of the Steinitz method is provided in \cref{algo:Steinitz}. 
Thereby, $\text{null}(\Phi) = \{ \mathbf{a} \in \R^N \mid \Phi \mathbf{a} = \mathbf{0} \}$ denotes the null space of the matrix $\Phi$. 

\begin{remark}\label{rem:a}
Note that \cref{eq:vec-a} is equivalent to $\mathbf{a} \in \text{null}(\Phi)$. 
Essentially every ${\mathbf{a} \in \text{null}(\Phi) \setminus \{\mathbf{0}\}}$ can be used (if $a_n \leq 0$ for all $n=1,\dots,M$, one can go over to $-a$). 
Moreover, it was shown \cref{lem:solution-space} that $\text{null}(\Phi)$ has dimension $N-K$. 
Hence, as long as $N > K$, such a vector of coefficients $\mathbf{a}$ can always be found. 
\end{remark}


\bibliographystyle{siamplain}
\bibliography{literature}

\end{document}